\documentclass[10pt]{amsart}
\usepackage[utf8]{inputenc}

\usepackage{amssymb,amsthm,amsmath}
\usepackage{mathrsfs}
\usepackage{enumerate}
\usepackage{graphicx,xcolor}
\usepackage{setspace}
\usepackage[hidelinks]{hyperref}
\usepackage{comment}
\usepackage[paper=a4paper, left=1.2in, right=1.2in, top=1.2in, bottom=1.2in]{geometry}
\newcommand{\ls}{\leqslant}
\newcommand{\gr}{\geqslant}

\newcommand{\R}{\mathbb{R}}

\DeclareMathOperator{\vol}{vol}

\renewcommand{\leq}{\leqslant}
\renewcommand{\geq}{\geqslant}
\renewcommand{\le}{\leqslant}
\renewcommand{\ge}{\geqslant}

\usepackage{dsfont}

\newtheorem{theorem}{Theorem}
\newtheorem{lemma}[theorem]{Lemma}
\newtheorem{corollary}[theorem]{Corollary}
\newtheorem{proposition}[theorem]{Proposition}

\theoremstyle{remark}
\newtheorem{remark}[theorem]{Remark}

\theoremstyle{definition}
\newtheorem{defn}[theorem]{Definition}

\title{On Sections of Convex Bodies in John's Position and of Generalised $B_p^n$ Balls}

\author{David Alonso-Guti\'{e}rrez}

\address{\'{A}rea de an\'{a}lisis matem\'{a}tico, Departamento de matem\'{a}ticas, Facultad de Ciencias, Universidad de Zaragoza, Pedro Cerbuna 12, 50009 Zaragoza (Spain), IUMA}

\email{alonsod@unizar.es}

\author{Silouanos Brazitikos}

\address{Department of Mathematics \& Applied Mathematics, University of Crete, Voutes Campus, 70013 Heraklion, Greece}

\email{silouanb@uoc.gr}

\author{Giorgos Chasapis}

\address{Department of Mathematics, University of Ioannina, University Campus, 45110 Ioannina, Greece}

\email{gchasapis@uoi.gr}

\date{\today}


\begin{document}

\begin{abstract}
We revisit an ingenious argument of K. Ball to provide sharp estimates for the volume of sections of a convex body in John's position. Our technique combines the geometric Brascamp-Lieb inequality with a generalised Parseval-type identity. This lets us complement some earlier results of the first two named authors, as well as generalise the classical estimates of Meyer-Pajor and Koldobsky regarding extremal sections of $B_p^n$ balls to a broader family of norms induced by a John's decomposition of the identity in $\mathbb{R}^n$.
\end{abstract}
\maketitle
		
\section{Introduction}
	The study of hyperplane sections and projections of convex bodies is a classical and actively developing area of modern convex geometry, with deep connections to functional analysis, geometric tomography, and high-dimensional probability. Extremal questions about $k$-dimensional sections — for example, determining maximal or minimal volumes of such sections — encode subtle quantitative information about the geometry of a body and are related to central problems of the field, cf. for example the Busemann–Petty problem \cite{BP} and the slicing problem (recently settled in \cite{KL}). We refer the reader to the recent survey \cite{NT} as well as the monographs \cite{Kold-book}, \cite{BGVV} for a detailed account of history, related works and advances in the field as well as numerous applications.
	
The present work further explores and complements some of the results of the recent paper \cite{AB}, where the first two named authors initiated a systematic study of sections of convex bodies placed in \emph{John's position}. John's position is a canonical normalization, which happens when the maximum volume ellipsoid contained in the convex body $K\subset\mathbb R^n$ is the Euclidean unit ball $B_2^n$. In such case, John \cite{John} exhibited contact points $v_1,\dots,v_m\in\partial K\cap S^{n-1}$ and positive weights $c_1,\dots,c_m$ that satisfy
$$
	\mathrm{Id}_n=\sum_{j=1}^m c_j\,v_j\otimes v_j,\qquad \sum_{j=1}^m c_j v_j=0,
$$
where $\mathrm{Id}_n$ denotes the identity in $\mathbb{R}^n$, $\partial K$ denotes the boundary of $K$, and $S^{n-1}$ denotes the Euclidean sphere. These identities provide a flexible algebraic framework that allows to reduce multidimensional volume problems to a combination of one-dimensional estimates and functional inequalities.
	
Recall that if $(v_j)_{j=1}^m\subseteq S^{n-1}$ and $(c_j)_{j=1}^m\subseteq(0,\infty)$ are as above, then necessarily $c_j\in(0,1]$ for every $1\ls j\ls m$. Besides, $\displaystyle{\sum_{j=1}^m c_j=n}$ and, for any integrable functions $(f_j)_{j=1}^m:\mathbb{R}\to[0,\infty)$,
\begin{equation}\label{eq:BrasLieb}
\int_{\mathbb{R}^n}\prod_{j=1}^mf_j^{c_j}(\langle x,v_j\rangle)\,dx\ls \prod_{j=1}^m\left(\int_\mathbb{R} f_j(t)\,dt\right)^{c_j}.
\end{equation}
The above so-called geometric form of the classical Brascamp-Lieb inequality \cite{BL} was suggested by Ball (see \cite[Lemma 2]{Ba}), who originally applied it to efficiently estimate the volume of $k$-dimensional sections of the unit cube $B_\infty^n$ in $\mathbb{R}^n$ showing, in particular, that for every $k$-dimensional linear subspace $H$ of $\mathbb{R}^n$
\begin{equation}\label{eq:Ball-1}
\vol_k(B_\infty^n\cap H)\ls \left(\frac{n}{k}\right)^{\frac{k}{2}}\vol_k(B_\infty^k).
\end{equation}
We remark that this estimate is optimal if and only if $k$ divides $n$. Using again \eqref{eq:BrasLieb} in conjunction with a Fourier-analytic argument, Ball established in the same work the inequality
\begin{equation}\label{eq:Ball-2}
\vol_k(B_\infty^n\cap H)\ls 2^{\frac{n-k}{2}}\vol_k(B_\infty^k),
\end{equation}
which is an optimal estimate whenever $k\gr n/2$. We stress that Ball's approach in this case relies heavily on the product structure of the cube: the Fourier transform of the indicator of a product body factorizes as a product of one-dimensional Fourier transforms, and this factorization is the key to obtaining sharp bounds. In the absence of such coordinate independence — i.e.\ for general bodies in John's position — the direct product decomposition is no longer available, and the Fourier method appears a priori inapplicable.


One of the results in \cite{AB} was a sharp generalisation of Ball's inequality \eqref{eq:Ball-1} for $k$-dimensional sections of any centrally symmetric convex body in John's position. We show in Section \ref{subsec:Symmetric} below that in the case $k\gr\frac{n}{2}$, this estimate is sharp and the upper bound provided by the right-hand side in \eqref{eq:Ball-2} is not true in general. We showcase this by constructing convex polytopes $L$ in John's position satisfying $\displaystyle{\vol_k(L\cap \mathbb{R}^k)=\left(\frac{n}{k}\right)^{k/2}\vol_k(B_\infty^k)}$ (see Theorem \ref{thm:counterexample} below). Nevertheless, we will provide sharper estimates for the volume of sections of centrally symmetric convex bodies in John's position (see Theorem \ref{thm:UpperBoundVolume}) by a $k$-dimensional linear subspace $H$ that take into account the geometry of the associated decomposition of the identity. This leads to an additional complication compared to Ball's setting, namely a non-trivial maximisation problem that remains to handle (cf. Proposition \ref{prop:exact-max}). Under the assumption that $c_j\Vert P_Hv_j\Vert_2^2\gr\frac{1}{2}$ for every $1\ls j\ls m$ such that $P_Hv_j\neq0$, provides better estimates than the upper bound in \eqref{eq:Ball-1} and, in some cases, than the upper bound in \eqref{eq:Ball-2} (see Corollary \ref{cor:UpperBoundVolumeBallsIntegralEstimate} and Proposition \ref{prop:ComparisonEstimates}).
	
The main conceptual advance of the present work is to show how the product-structure requirement in Ball's argument for the cube can be circumvented by combining a Parseval-type identity with the Brascamp–Lieb inequality. In our setting, this Parseval representation (Theorem~\ref{thm:parseval}) plays the role of the product decomposition in the cube case: the volume is represented as an integral of a product of univariate factors, but now the exponents and prefactors reflect the geometric data of John's decomposition rather than coordinate independence. Once the Parseval representation is in place, a Brascamp–Lieb reduction (Corollary~\ref{cor:UpperBoundParsevalAndBrascampLieb}) converts the multidimensional integral into a product of one-dimensional integrals with exponents determined by the Euclidean norms of the projections of our vectors.  Therefore, whenever we need to bound the volume (or the volume of a section) of a symmetric polytope, there are two equivalent representations.  One is the integral of the product of the indicator functions, and the other is the integral (over the complementary subspace) of the Fourier transform.  After passing to the Fourier representation we apply the Brascamp–Lieb inequality; the choice of exponents is governed by the sizes of the projections of the original vectors onto the complementary subspace. We will go through this procedure in detail in Section \ref{sec:Parseval}.

In parallel to the study in \cite{AB}, we have been also concerned with the extent at which our methods can provide meaningful estimates for the Wills functional of sections of centrally symmetric convex bodies in John's position (see Section \ref{sec:wills} for definitions and background). We provide an estimate that eventually yields a second proof of our aforementioned result on sections of centrally symmetric convex bodies, as well as a result on the mean width of sections, already witnessed in \cite{AB}.
	
Our second result extends the known bounds for sections of the $\ell_p^n$--balls, $p\in[1,2]$ (due to Meyer--Pajor \cite{MP} and Koldobsky \cite{Kold}).  Notably, our proof shows that the same method used for cube-slicing applies here as well.  We first express the relevant integral via Parseval's identity, then apply the Brascamp--Lieb inequality, and finally invoke a one-dimensional estimate for $\gamma_p$, the Fourier transform of $e^{-|x|^p}$. The technique allows us to consider a family of generalised $\ell_p$ norms of the form $\left(\sum_{j=1}^m \alpha_j|\langle \cdot,v_j\rangle|^p\right)^{1/p}$, given a John decomposition $(c_j,v_j)_{j=1}^m\subseteq (0,1]\times S^{n-1}$ in $\mathbb{R}^n$ and any positive scalars $(\alpha_j)_{j=1}^m$. We show for example that, for any $p\in[1,2]$, if $K_p$ is the closed unit ball for this norm then, for every $k$-dimensional subspace $H$ of $\mathbb{R}^n$,
\[
\vol_k(K_p\cap H)\ls \prod_{j=1}^m\left(\frac{\sqrt{c_j}}{\alpha_j^{1/p}}\right)^{1-c_j\Vert P_H(\sqrt{c_j}v_j)\Vert_2^2}\vol_k(B_p^k).
\]

In the classical case, our approach yields an intermediate bound for arbitrary sections of the $\ell_1^n$--ball, that depends only on the Euclidean norms of the projections of the standard basis vectors onto the subspace.  Consequently, the resulting bound is more sensitive than that of Meyer--Pajor and Koldobsky. Our results in Section \ref{sec:Kp} should be compared to some of the results of \cite{Bar} (see Section 3.3 therein, where a sharp estimate, in the case $k$ divides $n$, is also obtained for $\vol_k(B_p^n\cap H)$ for $p\gr 2$).


We also use the Fourier analytic approach in order to obtain upper estimates for the volume of sections of non-necessarily centrally symmetric convex bodies in John's positions, providing new estimates under some conditions on the associated decomposition of the identity (see Theorem \ref{thm:nonsym-sec-john}).

The paper is organized as follows: In Section \ref{sec:Preliminaries} we introduce the preliminary concepts that we need in order to state our results, as well as some known estimates with which we will compare our new estimates. In Section \ref{sec:GeneralSetting} we will provide the general setting  for our results and fix some notation that will be used through the whole paper. In Section \ref{sec:Parseval} we will present the Fourier analytic method and the Parseval-type inequality, which we will exploit in order to prove our new estimates. In Section \ref{sec:VolumesOfSections} we will apply such method to obtain new upper bounds for the volumes of centrally symmetric and non-necessarily centrally symmetric convex bodies in John's position. In Section \ref{sec:wills} we will focus on the use of our approach to provide estimates for the Wills functional of sections of centrally symmetric convex bodies in John's position, recovering the estimates we obtained for the volume. Finally, in Section \ref{sec:Kp} we will apply this method to obtain estimates for the volumes of sections of the aforementioned convex bodies $K_p$.

\section{Preliminaries}\label{sec:Preliminaries}

In this Section we are going to present some preliminary concepts and results, which will be needed in order to state and prove our results.

\subsection{Convex bodies in John's position and decompositions of the identity}\label{subsec:John's position}

A convex body $K\subseteq\R^n$ is said to be in John's position if the maximal volume ellipsoid contained in $K$ is the Euclidean ball. A classical result, (see \cite{John} and \cite{Ba3}) characterizes that a convex body $K\subseteq\R^n$ is in John's position if and only if there exists $(v_j)_{j=1}^m\subseteq \partial K\cap S^{n-1}$ and $(c_j)_{j=1}^m\subseteq(0,\infty)$ such that
\begin{equation}\label{eq:John-dec}
	\mathrm{Id}_n=\sum_{j=1}^m c_j\,v_j\otimes v_j,\qquad \sum_{j=1}^m c_j v_j=0,
\end{equation}
where $\mathrm{Id}_n$ denotes the identity in $\R^n$. Notice that, in such case, taking traces we have that necessarily $\displaystyle{\sum_{j=1}^m c_j=n}$. Notice also that for every $1\ls k\ls m$ we have
$$
1=\Vert v_k\Vert_2^2=\sum_{j=1}^m c_j\langle v_k, v_j\rangle^2\geq c_k\langle v_k,v_k\rangle^2=c_k,
$$
and then $(c_j)_{j=1}^m\subseteq(0,1]$. It is clear that if $K\subseteq\R^n$ is a centrally symmetric convex body and $(v_j)_{j=1}^n\subseteq \partial K\cap S^{n-1}$ and $(c_j)_{j=1}^m\subseteq(0,1]$ satisfy that
$$
\mathrm{Id}_n=\sum_{j=1}^m c_j\,v_j\otimes v_j,
$$
then $(\pm v_j)_{j=1}^m\subseteq \partial K\cap S^{n-1}$, $\left(\frac{c_j}{2}\right)_{j=1}^m\subseteq\left(0,\frac{1}{2}\right]\subseteq(0,1]$,
$$
\mathrm{Id}_n=\sum_{j=1}^m \frac{c_j}{2}\,v_j\otimes v_j+\sum_{j=1}^m \frac{c_j}{2}\,(-v_j)\otimes (-v_j),\quad\textrm{ and }\quad\sum_{j=1}^m \frac{c_j}{2} v_j+\sum_{j=1}^m \frac{c_j}{2}(-v_j)=0.
$$
Thus, the fact that a centrally symmetric convex body $K\subseteq\R^n$ is in John's position is characterized by the fact that $B_2^n\subseteq K$ and the existence of contact points $(v_j)_{j=1}^n\subseteq \partial K\cap S^{n-1}$ and coefficients $(c_j)_{j=1}^m\in(0,1]$ providing a decomposition of the identity.

It is also clear that if $(x_j)_{j=1}^n$ is an orthonormal basis in $\R^n$, then they provide a decomposition of the identity with coefficients $c_j=1$ for every $1\ls j\ls n$.


Let us point out that if  $(c_j,v_j)_{j=1}^m\subseteq(0,1]\times S^{n-1}$ satisfy that $\displaystyle{\mathrm{Id}_n=\sum_{j=1}^m c_j\,v_j\otimes v_j}$ and $H\in G_{n,k}$ is a $k$-dimensional linear subspace, calling $J=\{1\ls j\ls m\,:P_Hv_j\neq0\}$, we have that for every $x\in H$
$$
x=\sum_{j=1}^m c_j\langle x,P_H v_j\rangle P_H v_j=\sum_{j\in J} c_j\langle x,P_H v_j\rangle P_H v_j=\sum_{j\in J} c_j\Vert P_Hv_j\Vert_2^2\left\langle x,\frac{P_H v_j}{\Vert P_H v_j\Vert_2}\right\rangle \frac{P_H v_j}{\Vert P_H v_j\Vert_2}.
$$
That is, denoting by $\mathrm{Id}_H$ the identity in $H$,
\begin{equation}\label{eq:IdentityProjection}
\mathrm{Id}_H=\sum_{j\in J}\tilde{c}_j u_j\otimes u_j,
\end{equation}
where $\tilde{c}_j=c_j\Vert P_H v_j\Vert_2^2$ and $u_j=\frac{P_H v_j}{\Vert P_H v_j\Vert_2}$ for every $j\in J$. In particular, the non-zero orthogonal projections of an orthonormal basis on $\R^n$ onto a linear subspace $H$ provide a decomposition of the identity on $H$. The following proposition, which is well known (see for example \cite[Lemma 2.1]{Iv}) shows that any decomposition of the identity is obtained as the orthogonal projections of the vectors in an orthonormal basis in a higher-dimensional space. We include a proof for the sake of completeness.
	
	\begin{proposition}\label{prop:H^bot}
		Let $1\ls k\ls m$. Let $H\in G_{m,k}$ and let $(u_j)_{j=1}^m\subseteq S^{m-1}\cap H$ and $(\tilde{c}_j)_{j=1}^m\subseteq(0,1]$ such that
		\[
		\mathrm{Id}_H = \sum_{j=1}^m \tilde{c}_j u_j\otimes u_j.
		\]
		Then, there exists $(x_j)_{j=1}^m$, an orthonormal basis of $\mathbb{R}^m$ and $(w_j)_{j=1}^m\subseteq S^{m-1}\cap H^\perp$ such that
		\[
		P_H x_j= \sqrt{\tilde{c}_j}u_j \qquad \hbox{ and }\qquad P_{H^\bot}x_j=\sqrt{1-\tilde{c}_j}w_j
		\]
		for every $1\ls j\ls m$. Moreover,
		\[
		\mathrm{Id}_{H^\bot} = \sum_{j=1}^m (1-\tilde{c}_j)w_j\otimes w_j.
		\]
	\end{proposition}
	
	\begin{proof}
		We can assume, without loss of generality, that $H=\textrm{span}\{e_1,\dots,e_k\}\subseteq\R^m$, where $(e_j)_{j=1}^m$ denotes the canonical basis of $\R^m$. Calling $\tilde{u}_j=\sqrt{\tilde{c}_j}u_j$ we have that for every $1\ls i\ls k$,
		$$
		1=\Vert e_i\Vert_2^2=\sum_{j=1}^m\langle e_i,\tilde{u}_j \rangle^2.
		$$
		Besides, for every $1\ls i_1<i_2\ls k$ we have that
		$$
		0=\langle e_{i_1},e_{i_2}\rangle=\sum_{j=1}^m\langle e_{i_1},\tilde{u}_j\rangle\langle e_{i_2},\tilde{u}_j\rangle.
		$$
		Therefore, calling, for $1\ls j\ls k$, $\tilde{v}_j\in\R^m$ the $j$-th row of the $k\times m$ matrix $[\tilde{u}_1,\dots,\tilde{u}_m]$ we have that $\{\tilde{v}_j\,:\,1\ls j\ls k\}$ is a set of $k$ orthonormal vectors in $\R^m$. Completing this set to an orthormal basis of $\R^m$, $\{\tilde{v}_j\,:\,1\ls j\ls m\}$ and taking, for $1\ls j\ls m$, $x_j\in\R^m$ the $j$-th column of the matrix $[\tilde{v}_1,\dots,\tilde{v}_m]^t$, we have that
		\begin{itemize}
			\item $(x_j)_{j=1}^m\subseteq\R^m$ is an orthonormal basis of $\R^m$,
			\item $P_Hx_j=\tilde{u}_j=\sqrt{\tilde{c}_j}u_j$ for every $1\ls j\ls m$.
		\end{itemize}
		Besides, if $H^\perp=\textrm{span}\{e_{k+1}\dots e_m\}$ is the orthogonal complement of $H$ in $\R^m$, we have that for every $1\ls j\ls m$
		$$
		\Vert P_{H^\perp}x_j\Vert_2^2=\Vert x_j\Vert_2^2-\Vert P_{H}x_j\Vert_2^2=1-\Vert \sqrt{\tilde{c}_j}u_j\Vert_2^2=1-\tilde{c}_j
		$$
		and, therefore, there exist $(w_j)_{j=1}^m\in S^{m-1}\cap H^\perp$ such that $P_{H^\perp}x_j=\sqrt{1-\tilde{c}_j}w_j$. Moreover, since
		$$
		\mathrm{Id}_m=\sum_{j=1}^m x_j\otimes x_j,
		$$
		we have that
		$$
		Id_{H^\perp}=\sum_{j=1}^m P_{H^\perp}x_j\otimes P_{H^\perp}x_j=\sum_{j=1}^m(1-\tilde{c}_j)w_j\otimes w_j.
		\eqno\qedhere$$
	\end{proof}

Through the text, we will denote by $B_\infty^n$ the $n$-dimensional cube, and by $S_n$ the n-dimensional regular simplex in John's position.
\subsection{Brascamp-Lieb inequality}\label{subsec:Brascamp-Lieb}

Brascamp-Lieb inequality (see \cite{BL}) states that, given $m\gr n$, $(v_j)_{j=1}^m\in\R^n$, and $(c_j)_{j=1}^m\in(0,1]$ such that $\displaystyle{\sum_{j=1}^m c_j=n}$, for any integrable $(f_j)_{j=1}^m:\R\to[0,\infty)$ we have
$$
\int_{\R^n}\prod_{j=1}^m f_j^{c_j}(\langle x, v_j\rangle)dx\leq D\prod_{j=1}^m\left(\int_{\R}f_j(t)dt\right)^{c_j},
$$
where
$$
D=\sup\left\{\frac{\int_{\R^n}\prod_{j=1}^mg_j^{c_j}(\langle x,v_j\rangle)dx}{\prod_{j=1}^m\left(\int_\R g_j(t)dt \right)^{c_j}}\,:\,g_j(x)=e^{-\lambda_jx^2}\right\}=\sup\left\{\sqrt{\frac{\prod_{j=1}^m\lambda_j^{c_j}}{\textrm{det}(\sum_{j=1}^mc_j\lambda_j u_j\otimes u_j)}}\,:\,\lambda_j>0\right\}.
$$
In \cite{Ba}, Ball showed that if the vectors $(v_j)_{j=1}^m$ are unit vectors which, together with the numbers $(c_j)_{j=1}^m$, provide a decomposition of the identity, then  $D=1$, providing the so-called geometric Brascamp-Lieb inequality
\begin{theorem}[Geometric Brascamp-Lieb inequality]\label{thm:GeometricBrascampLieb}
Let  $(c_j,v_j)_{j=1}^m\subseteq(0,1]\times S^{n-1}$ such that $\displaystyle{\mathrm{Id}_n=\sum_{j=1}^m c_j\,v_j\otimes v_j}$. Then, for any integrable $(f_j)_{j=1}^m:\R\to[0,\infty)$, we have
$$
\int_{\R^n}\prod_{j=1}^m f_j^{c_j}(\langle x, v_j\rangle)dx\leq \prod_{j=1}^m\left(\int_{\R}f_j(t)dt\right)^{c_j}.
$$
\end{theorem}

\subsection{The intrinsic volumes and the Wills functional}\label{subsec:IntrinsicVolumesWills}

Given any compact convex set $K\subseteq\R^n$, by Steiner's formula (see \cite[equation (4.1)]{Sch}), the volume of $K+tB_2^n$ can be expressed as a polynomial in the variable $t\gr0$,
$$
\vol_n(K+tB_2^n)=\sum_{i=0}^n{{n}\choose{i}}W_i(K)t^i,
$$
where the numbers $W_i(K)$ are called the querma\ss integrals of $K$. We have that $W_0(K)=\vol_n(K)$, $nW_1(K)=|\partial K|$ is the surface area of $K$, and $W_{n-1}(K)=\vol_n(B_2^n)w(K)$ is a multiple of the mean width of $K$, where the mean width of $K$ is
$$
w(K)=\int_{S^{n-1}}h_K(\theta)d\sigma(\theta),
$$
being $h_K(\theta)=\sup_{x\in K}\langle x,\theta\rangle$, the support function of $K$ in the direction $\theta$, and $d\sigma$, the uniform probability measure on the Euclidean sphere. Besides, $W_n(K)=\vol_n(B_2^n)$.

If $K$ is contained in a $k$-dimensional subspace $H\in G_{n,k}$, one can also compute the querma\ss integrals with respect to $H$, identified with $\R^k$. Denoting by $W_i^{(k)}(K)$, for $i=0,\dots,k$, the querma\ss integrals with respect to $H$, we have that (see e.g. \cite[Property~3.1]{SY93})
$$
W_{i}^{(k)}(K)=\frac{{{n}\choose{n-k+i}}}{{{k}\choose{i}}}\frac{\vol_i(B_2^i)}{\vol_{n-k+i}(B_{2}^{n-k+i})}W_{n-k+i}(K),\quad
\forall0\ls i\ls k,
$$
while $W_i(K)=0$ for all $0\ls i<n-k$. In order to avoid the
issue that querma\ss integrals depend on the space where the convex
body is embedded, McMullen \cite{Mc75} defined the  intrinsic volumes
of a compact convex set $K\subseteq\R^n$ as
$$
V_i(K)=\frac{{{n}\choose{i}}}{\vol_{n-i}(B_2^{n-i})}W_{n-i}(K),\quad
\forall 0\ls i\ls n.
$$
In particular, $V_0(K)=1$, $V_1(K)=\frac{n\vol_n(B_2^n)}{\vol_{n-1}(B_2^{n-1})}w(K)$, and $V_n(K)=\vol_n(K)$.

In \cite{W73}, Wills introduced and studied the functional
\begin{equation}\label{eq:FWills}
\mathcal{W}(K)=\sum_{i=0}^n V_i(K)=1+\sum_{i=1}^nV_i(K),
\end{equation}
motivated by its possible relation with the lattice point enumerator $G_n(K)=\sharp(K\cap\mathbb{Z}^n)$, the cardinality of $K\cap\mathbb{Z}^n$. Notice that, since the $i$-th intrinsic volume of $K$ is homogeneous of degree $i$, for any $\lambda\geq 0$ we have
$$
\mathcal{W}(\lambda K)=1+\sum_{i=1}^nV_i(\lambda K)=1+\sum_{i=1}^n\lambda^iV_i(K),
$$
and then, if $K\subseteq\R^n$ is a convex body,
\begin{equation}\label{eq:WillsVolumeMeanWidth}
\textrm{vol}_n(K)=\lim_{\lambda\to\infty}\frac{\mathcal{W}(\lambda K)}{\lambda^n},\quad\textrm{and}\quad V_1(K)=\lim_{\lambda\to0^+}\frac{\mathcal{W}(\lambda K)-1}{\lambda}.
\end{equation}

We refer the reader to \cite{W2}, \cite{H}, \cite{McM2}, or \cite{AHY} for more properties on the Wills functional. For our purposes, we emphasize the following one (see \cite[Eq.(1.3)]{H}): For any convex body $K\subseteq\R^n$
\begin{equation}\label{eq:WillsIntegral}
\mathcal{W}(K)=\int_{\R^n}e^{-\pi d^2(x,K)}dx,
\end{equation}
where $d(\cdot, K)$ denotes the Euclidean distance to $K$.

\subsection{Sections of convex bodies in John's position}\label{subsec:SectionsJohnsPositionKnownEstimates}

In \cite{AB}, the authors initiated a systematic study of the sections of convex bodies in John's position, in order to extend inequality \eqref{eq:Ball-1} to sections of centrally symmetric convex bodies in John's position and, using the same approach, obtain estimates for other geometric parameters. We summarize some of the results obtained for sections of centrally symmetric convex bodies in the following theorem:

\begin{theorem}\label{thm:SectionsSymmetricJohn}
Let $K\subseteq\R^n$ be a centrally symmetric convex body in John's position and let $(c_j,v_j)_{j=1}^m\subseteq (0,1]\times(\partial K\cap S^{n-1})$ be its associated decomposition of the identity. Let $H\in G_{n,k}$ and $J=\{1\ls j\ls m\,:P_Hv_j\neq0\}$. Then,
\begin{itemize}
\item $\displaystyle{\vol_k(K\cap H)\ls\prod_{j\in J}\left(\frac{1}{\Vert P_H v_j\Vert_2}\right)^{c_j\Vert P_Hv_j\Vert_2^2}\vol_k(B_\infty^k)\leq\left(\frac{n}{k}\right)^\frac{k}{2}\vol_k(B_\infty^k).}$
\item $\displaystyle{w(K\cap H)\ls\frac{\Gamma\left(\frac{k}{2}\right)}{\sqrt{\pi}\Gamma\left(\frac{k+1}{2}\right)}\sum_{j\in J}c_j\Vert P_Hv_j\Vert_2\ls\sqrt{\frac{n}{k}}w(B_\infty^k)}$, which is equivalent to \\
$\displaystyle{V_1(K\cap H)\ls 2\sum_{j\in J}c_j\Vert P_Hv_j\Vert_2\ls\sqrt{\frac{n}{k}}V_1(B_\infty^k)}$.
\item $\displaystyle{\mathcal{W}(\lambda(K\cap H))\ls\prod_{j\in J}\left(1+\frac{2\lambda}{\Vert P_Hv_j\Vert_2}\right)^{c_j\Vert P_Hv_j\Vert_2^2}\ls\mathcal{W}\left(\lambda\sqrt{\frac{n}{k}}B_\infty^k\right)}$, for every $\lambda\geq0$.
\end{itemize}
\end{theorem}
Observe that, by \eqref{eq:WillsVolumeMeanWidth}, the inequalities for the volume and the mean width can be deduced from the inequalities for the Wills functional.

The following estimate was also obtained for sections of not necessarily centrally symmetric convex bodies:

\begin{theorem}\label{thm:SectionsNonSymmetricJohn}
Let $K\subseteq\R^n$ be a convex body in John's position. Let $(c_j,v_j)_{j=1}^m\subseteq (0,1]\times(\partial K\cap S^{n-1})$ be its associated decomposition of the identity and let $H\in G_{n,k}$. Then,
$$
\vol_k(K\cap H)\ls\frac{n^\frac{k}{2}}{k^\frac{k}{2}}\frac{(n+1)^\frac{k+1}{2}}{(k+1)^\frac{k+1}{2}}\prod_{j=1}^m\Vert P_F\tilde{v}_j\Vert_2^{\delta_j\Vert P_F\tilde{v}_j\Vert_2^2}\vol_k(S_k)\ls\frac{1}{(k+1)^\frac{n-k}{2(n+1)}}\frac{(n(n+1))^\frac{k}{2}}{(k(k+1))^\frac{k}{2}}\vol_k(S_k).
$$
where, $F=H\times\R\in G_{k+1,n+1}$ and, for every $1\ls j\ls m$,
\begin{itemize}
\item $\tilde{v}_j=\sqrt{\frac{n}{n+1}}\left(-v_j,\frac{1}{\sqrt{n}}\right)$ and
\item $\delta_j=\frac{n+1}{n}c_j$.
\end{itemize}
\end{theorem}
\subsection{The Fourier transform}\label{subsec:FourierTransform}

Given any integrable function $f:\R^n\to\R$, we denote by $\widehat{f}$ the Fourier transform
$$
\widehat{f}(y)=\int_{\R^n}f(x)e^{i\langle x,y\rangle}dx,\qquad y\in\R^n.
$$
Notice that, from the definition, if for any $\alpha>0$ we call $g(x)=f(\alpha x)$, then
\begin{equation}\label{eq:FourierScaling}
\widehat{g}(y)=\frac{\widehat{f}\left(\frac{y}{\alpha}\right)}{\alpha^n},\quad\forall y\in\R^n.
\end{equation}

It is well known that if $a>0$, then the Fourier transform of the indicator function of $[-a,a]$ is, for every $y\in\R$
\begin{equation}\label{eq:FourierTransformCharacteristic}
\widehat{\mathds{1}}_{[-a,a]}(y)=\begin{cases}
\frac{2\sin(ay)}{y}&\textrm{ if }a\neq0\\
2&\textrm{ if }a=0,
\end{cases}
\end{equation}
and that if $f$ is the standard Gaussian density, $f(x)=\frac{1}{\sqrt{2\pi}}e^{-\frac{x^2}{2}}$, then
\begin{equation}\label{eq:FourierTransformGaussian}
\widehat{f}(y)=\int_\R\frac{e^{-\frac{x^2}{2}}}{\sqrt{2\pi}}e^{ixy}dx=\frac{2}{\sqrt{2\pi}}\int_0^\infty e^{-\frac{x^2}{2}}\cos(xy)dx=e^{-\frac{y^2}{2}},\qquad y\in\R.
\end{equation}

Let us denote by ${\mathcal S}(\mathbb{R}^n)$, the Schwartz space of all $C^{\infty)}(\mathbb{R}^n)$ functions $f:\mathbb{R}^n\to \mathbb{R}$ with derivatives that are rapidly decreasing, that is,
$$
{\mathcal S}(\mathbb{R}^n)=\{f\in\mathcal{C}^{\infty)}(\R^n)\,:\forall \alpha,\beta, \Vert f\Vert_{\alpha,\beta}<\infty\},
$$
where, for any multiindex $\alpha,\beta$,
$$
\Vert f\Vert_{\alpha,\beta}=\sup_{x\in\R^n}\left|x_{i_1}^{\alpha_1}\dots x_{i_m}^{\alpha_m}\frac{\partial^{|\beta|}f}{\partial x_{j_1}^{\beta_1}\dots\partial x_{j_k}^{\beta_k}}(x)\right|,
$$
being $\displaystyle{\sum_{i=1}^m\alpha_i=|\alpha|}$ and $\displaystyle{\sum_{j=1}^k\beta_j=|\beta|}$. Fourier inversion theorem states that if $f\in\mathcal{S}(\R^n)$, then, for every $x\in\R^n$
$$
f(x)=\frac{1}{(2\pi)^n}\int_{\R^n}\widehat{f}(y)e^{-i\langle x,y\rangle}dy=\frac{1}{(2\pi)^n}\widehat{\widehat{f}}(-x).
$$
In particular,
\begin{equation}\label{eq:FourierInversionAt0}
f(0)=\frac{1}{(2\pi)^n}\int_{\R^n}\widehat{f}(y)dy.
\end{equation}

In \cite{Be}, Beckner proved the following sharp Hausdorff-Young inequality: If $p\geq 2$, then the Fourier transform extends to a bounded liner operator from $L^{p^\prime}(\R^n)$ to $L^p(\R^n)$ and
\begin{equation}\label{eq:Beckner}
\left(\int_{\R^n}|\widehat{f}(y)|^pdy\right)^{1/p}\leq(2\pi)^{1/p}\left(\frac{p^{1/p}}{{p^\prime}^{1/p^\prime}}\right)^{n/2}\left(\int_{\R^n}|f(y)|^{p^\prime}dy\right)^{1/p^\prime},
\end{equation}
where $p^\prime$ denotes the conjugate exponent of $p$, such that $\frac{1}{p}+\frac{1}{p^\prime}=1$. In particular, together with Fourier inversion theorem, the Fourier transform extends to a bounded linear operator from $L^2(\R^n)$ to $L^2(\R^n)$ and we have Plancherel's identity:
$$
\left(\int_{\R^n}|\widehat{f}(y)|^2dy\right)^{1/2}=\sqrt{2\pi}\left(\int_{\R^n}|f(y)|^{2}dy\right)^{1/2}.
$$

\subsection{Volumes of sections of the cube via the Fourier transform}\label{subsec:SectionsCube}

In \cite{Ba}, Ball used a Fourier analytic argument to prove that, for any $k$-dimensional linear subspace $H\in G_{n,k}$,
$$
\vol_k(B_\infty^n\cap H)\ls 2^{\frac{n-k}{2}}\vol_k(B_\infty^k),
$$
with equality for some $H$ if $k\geq \frac{n}{2}.$

In order to obtain this estimate, he combined the fact that the Fourier transform of the density of the orthogonal projection onto $H^\perp$ of a random vector uniformly distributed on $B_\infty^n$, factorizes as the product of Fourier transforms indicator functions on intervals. This fact, together with Fourier inversion formula and the Brascamp-Lieb inequality (Theorem \ref{thm:GeometricBrascampLieb}) implies that, if $(e_j)_{j=1}^n$ is the canonical basis in $\R^n$ and $H\in G_{n,k}$ is a $k$-dimensional linear subspace such that $P_He_j\neq0$ for every $1\leq j\leq n$, denoting by  $\tilde{c}_j=\Vert P_He_j\Vert_2^2$, we have
\begin{equation}\label{eq:BallsEstimate}
\vol_k(B_\infty^n\cap H)\ls \vol_k(B_\infty^k)\prod_{j=1}^n\left(\frac{I_{\frac{1}{1-\tilde{c}_j}}}{\pi\sqrt{1-\tilde{c}_j}}\right)^{1-\tilde{c}_j},
\end{equation}
where, for any $p\geq1$, $I_p$ denotes
\begin{equation}\label{eq:DefinitionI_p}
I_p=\int_\R\left|\frac{\sin t}{t}\right|^pdt
\end{equation}
 and, if for some $1\ls j\ls n$  we have that $\tilde{c}_j=\Vert P_He_j\Vert_2^2=1$, the corresponding factor is understood as $1$. This estimate, combined with the following estimate proved in \cite{B}
\begin{equation}\label{eq:BallsIntegralEstimate}
\frac{I_p}{\pi}\ls\sqrt{\frac{2}{p}},\quad\forall p\gr2,
\end{equation}
with equality if and only if $p=2$, proves the estimate in the case that $\tilde{c}_j=\Vert P_He_j\Vert_2^2\gr\frac{1}{2}$ for every $1\ls j\ls n$. In the case that there exists $1\ls j\ls n$ such that $\Vert P_He_j\Vert_2^2<\frac{1}{2}$, an induction argument provides the estimate.

\subsection{Volumes of sections of the $\ell_p^n$-balls}\label{subsec:SectionsB_p^n}

In \cite{MP}, Meyer and Pajor also used a Fourier analytic approach in order to estimate the volumes of sections of $B_p^n$, the unit balls of $\ell_p^n$. They proved (see \cite[Lemma II.6]{MP})  that for any $1\leq p\leq2$ and any $k$-dimensional linear subspace $H\in G_{n,k}$
$$
\vol_k(B_p^n\cap H)=\frac{\vol_k(B_p^k)}{(2\pi)^{n-k}}\int_{\R^{n-k}}\prod_{i=1}^n\widehat{f_{\alpha_p,p}}\left(\sum_{j=1}^{n-k}a_i^jt_j\right)dt,
$$
where the vectors $u_j=(a_1^j,\dots a_n^j)$, $1\ls j\ls n-k$ form an orthonormal basis of $H^\perp$, $\alpha_p=2\Gamma\left(1+\frac{1}{p}\right)$, and $f_{\alpha,p}(x)=e^{-\alpha|x|^p}$. Equivalently,
\begin{equation}\label{eq:MeyerPajor}
\vol_k(B_p^n\cap H)=\frac{\vol_k(B_p^k)}{(2\pi)^{n-k}}\int_{H^\perp}\prod_{i=1}^n\widehat{f_{\alpha_p,p}}\left(\langle z,e_i\rangle\right)dz,
\end{equation}
where $(e_i)_{i=1}^n$ the canonical basis in $\R^n$.

In the particular case that $p=1$, using that $\widehat{f_{\alpha_1,1}}(s)=\frac{1}{1+\frac{s^2}{4}}$ for every $s\in\R$ (see Lemma \ref{lem:FourierTransformExponential} below, where, for the sake of completeness, $\widehat{f_{\alpha,1}}$ is computed for every $\alpha>0$), it was proved (see \cite[Lemma III.7]{MP}) that for every $H\in G_{n,k}$
\begin{equation}\label{eq:MeyerPajorSectionsB_1^n}
\vol_k(B_1^n\cap H)\geq\left(\frac{n}{\pi}\right)^\frac{n-k}{2}\frac{\Gamma\left(\frac{n+k}{2}\right)}{\Gamma(n)}\vol_k(B_1^k).
\end{equation}

\subsection{Generalized $\ell_p^n$-balls}\label{subsec:GeneralizedB_p^n}

Let us consider, given $(v_j)_{j=1}^m\subseteq S^{n-1}$, $(\alpha_j)_{j=1}^m\subseteq(0,\infty)$, and $p\geq1$, the convex body $K_p=K_p\left((v_j)_{j=1}^m,(\alpha_j)_{j=1}^m\right)$ defined as the unit ball of the norm
$$
\Vert x\Vert_{K_p}=\left(\sum_{j=1}^m\alpha_j|\langle x,v_j\rangle|^p\right)^\frac{1}{p}.
$$
Notice that, if $m=n$, $v_j=e_j$ for every $1\ls j\ls n$, and $\alpha_j=1$, then $K_p=B_p^n$.

A direct application of the Brascamp-Lieb inequality shows (see, for instance, \cite[Proposition 3.1]{A}) that for any $(c_j)_{j=1}^m\subseteq(0,1)$ such that $\displaystyle{\sum_{j=1}^m}c_j=n$, we have
$$
\vol_n(K_p)\leq \vol_k(B_p^k)D\left((v_j)_{j=1}^m,(c_j)_{j=1}^m\right)\prod_{j=1}^m\left(\frac{c_j}{\alpha_j}\right)^\frac{c_j}{p},
$$
where $D\left((v_j)_{j=1}^m,(c_j)_{j=1}^m\right)$ denotes the constant in the Brascamp-Lieb inequality corresponding to the vectors $(v_j)_{j=1}^m$ and the scalars $(c_j)_{j=1}^m$. In particular, if $(c_j,v_j)_{j=1}^m\subseteq (0,1]\times S^{n-1}$ provide a decomposition of the identity, for any $(\alpha_j)_{j=1}^m$ we have (see also \cite[Proposition 8]{Ba1991})
$$
\vol_n(K_p)\leq \vol_k(B_p^k)\prod_{j=1}^m\left(\frac{c_j}{\alpha_j}\right)^\frac{c_j}{p}.
$$
Moreover, since for any $H\in G_{n,k}$, if we denote by $J=\{1\leq j\leq m\,:\,P_Hv_j\neq0\}$, we have that,
$$
\Vert x\Vert_{K_p\cap H}=\left(\sum_{j\in J}\alpha_j\Vert P_Hv_j\Vert_2^p|\langle x,u_j\rangle|^p\right),\quad\forall x\in H
$$
where $u_j=\frac{P_Hv_j}{\Vert P_Hv_j\Vert}\in S^{n-1}\cap H$, we obtain that
$$
K_p\left((v_j)_{j=1}^m,(\alpha_j)_{j=1}^m\right)\cap H=K_p\left((u_j)_{j\in J},(\alpha_j\Vert P_Hv_j\Vert_2^p)_{j\in J}\right)
$$
and, taking into account \eqref{eq:IdentityProjection}, we have that if $(c_j,v_j)_{j=1}^m\subseteq (0,1]\times(\partial K\cap S^{n-1})$ provide a decomposition of the identity, then
\begin{equation}\label{eq:UpperBoundSectiosKp}
\vol_k(K_p\cap H)\leq\vol_k(B_p^k)\prod_{j=1}^m\left(\frac{c_j\Vert P_Hv_j\Vert_2^{2-p}}{\alpha_j}\right)^\frac{c_j\Vert P_Hv_j\Vert_2^2}{p}.
\end{equation}
In particular,
$$
\vol_k(B_p^n\cap H)\leq\vol_k(B_p^k)\prod_{j=1}^m\Vert P_He_j\Vert_2^{\left(\frac{2}{p}-1\right)\Vert P_He_j\Vert_2^2}.
$$

\section{General setting}\label{sec:GeneralSetting}

In this section we are going to fix the notation that we will use through the whole paper. Given $(c_j,v_j)_{j=1}^m\subseteq (0,1]\times S^{n-1}$  providing a decomposition of the identity in $\R^n$, and $H\in G_{n,k}$, we will call $J=\{1\leq j\leq m\,:\,P_Hv_j\neq0\}$, and $m_0$ the cardinality of $J$. For every $j\in J$ we will denote
\begin{itemize}
\item $\tilde{c}_j=c_j\Vert P_H v_j\Vert_2^2$,
\item $u_j=\frac{P_Hv_j}{\Vert P_Hv_j\Vert_2}\in S^{n-1}\cap H$,
\item $t_j=\frac{1}{\Vert P_Hv_j\Vert_2}$.
\end{itemize}
With this notation, we have that if $L$ is the centrally symmetric polytope
$$
L:=\{x\in\R^n\,:\,|\langle x,v_j\rangle|\leq 1,\,\forall 1\leq j\leq m\},
$$
then
$$
L\cap H=\{x\in H\,:\,|\langle x,u_j\rangle|\leq t_j,\,\forall j\in J\}.
$$
Besides, by \eqref{eq:IdentityProjection}, $(\tilde{c}_j,u_j)_{j\in J}\subseteq (0,1]\times(S^{n-1}\cap H)$ provide a decomposition of the identity in $H$.

By Proposition \ref{prop:H^bot}, we can identify $\R^n=\textrm{span}\{e_1,\dots, e_n\}\subseteq\R^{m_0}$, with $(e_j)_{j=1}^{m_0}$ the canonical basis of $\R^{m_0}$, and $H\in G_{m_0,k}$ and find an orthonormal basis $(x_j)_{j=1}^{m_0}$ in $\R^{m_0}$ such that
$$
P_Hx_j=\sqrt{\tilde{c}_j}u_j\quad\forall 1\leq j\leq m_0.
$$
Moreover, denoting by
$$
\tilde{J}=\{j\in J\,:\,\Vert P_Hx_j\Vert\neq1\}=\{j\in J\,:\,\Vert P_{H^\perp}x_j\Vert\neq0\}=\{j\in J\,:\,\tilde{c}_j\neq1\},
$$
we will call $m_1$ the cardinality of $\tilde{J}$ and, for every $j\in \tilde{J}$,
$$
w_j=\frac{P_{H^\perp}x_j}{\Vert P_{H^\perp}x_j\Vert_2}=\frac{P_{H^\perp}x_j}{\sqrt{1-\tilde{c}_j}}.
$$
Thus, we have that $(1-\tilde{c}_j,w_j)_{j\in \tilde{J}}\subseteq (0,1]\times(S^{n-1}\cap H^\perp)$ provide a decomposition of the identity in $H^\perp$.

Whenever  $(c_j,v_j)_{j=1}^m\subseteq (0,1]\times S^{n-1}$  provide a decomposition of the identity in $\R^n$ with the additional assumption that $\displaystyle{\sum_{j=1}^mc_jv_j=0}$, we will denote
$$
C:=\{x\in\R^n\,:\,\langle x,v_j\rangle\leq 1,\,\forall 1\leq j\leq m\}.
$$

We will also call, for every $1\leq j\leq m$
$$
\tilde{v}_j=\sqrt{\frac{n}{n+1}}\left(-v_j,\frac{1}{\sqrt{n}}\right),\quad\textrm{and}\quad\delta_j=\frac{n+1}{n}c_j.
$$
With this notation, we have that $\displaystyle{\sum_{j=1}^m\delta_j\tilde{v}_j=(0,\sqrt{n+1})\in \R^{n+1}}$, and $(\delta_j,\tilde{v}_j)_{j=1}^m\in (0,1]\times S^{n}$ provide a decomposition of the identity in $\R^{n+1}$. We will denote by $L$ the cone
$$
L=\{y\in\R^{n+1}\,:\,\langle y,\tilde{v}_j\rangle\geq0,\,\forall 1\leq j\leq m\}.
$$
It was proved in \cite{Ba1991} that
$$
L=\left\{(x,r)\in\R^{n+1}\,:\,r\geq0,\,x\in\frac{r}{\sqrt{n}}C\right\}.
$$
Moreover, given $H\in G_{n,k}$, we will denote $F=\textrm{span}\{(x,\sqrt{n})\,:\,x\in H\}=H\times\R\in G_{n+1,k+1}$ Notice that for every $1\leq j\leq m$ we have that $P_F\tilde{v}_j\neq0$. We will denote, for every $1\leq j\leq m$,
\begin{itemize}
\item $u_j=\frac{P_F\tilde{v}_j}{\Vert P_F\tilde{v}_j\Vert_2}$
\item $\tilde{\delta}_j=\delta_j\Vert P_F\tilde{v}_j\Vert_2^2$.
\end{itemize}
Notice that $(\tilde{\delta}_j,u_j)_{j=1}^m$ provide a decomposition of the identity in $F$.

\section{Parseval's identity and consequences}\label{sec:Parseval}

In this section we present the aforementioned Parseval-type representation and how, when combined with an application of the Brascamp-Lieb inequality, it provides an upper bound for the integral of the product of functions evaluated at the scalar products of the variable against the vectors that form a decomposition of the identity. We will also compare this upper bound with the upper bound we would obtain by directly applying Brascamp-Lieb inequality and show that in some cases we obtain a tighter estimate (see Remark \ref{rem:UpperBoundParseval-BrascampLieb} below).

Our starting point is an appropriate general form of Parseval's identity. For a function $f:\R^n\to\R$, we denote by $\widehat{f}$ the Fourier transform
$$
\widehat{f}(y)=\int_{\R^n}f(x)e^{i\langle x,y\rangle}dx,\qquad y\in\R^n.
$$
We denote by ${\mathcal S}(\mathbb{R}^n)$ the Schwartz space of all $C^{\infty)}$ functions $f:\mathbb{R}^n\to \mathbb{R}$ with derivatives that are rapidly decreasing. To overcome any integrability issues and for the sake of clarity, we formulate the results of this Section for functions in ${\mathcal S}(\mathbb{R}^n)$. For our purposes however, we will need them to be applicable for families of characteristic functions on bounded intervals. We have included the technical details that justify how such an extension is possible in an Appendix at the end of the manuscript. We remark that similar Parseval-type identities for \textit{two} functions have been established, in connection to Busemann-Petty type problems for sections and projections of convex bodies, by Koldobsky \cite{K99} and Koldobsky, Ryabogin and Zvavitch \cite{KRZ} (see also \cite{Kold-book}). It is noteworthy that ``Fourier--Brascamp--Lieb'' inequalities in the likes of Corollary \ref{cor:UpperBoundParsevalAndBrascampLieb} have been explored in the works of Bennett, Bez, Buschenhenki, Cowling and Flock \cite{BBBCF2020} and Bennett and Jeong \cite{BJ}, partly motivated by K. Ball's work as well. See also \cite{BC} for further manifestation of this concept in the more general setting of LCA groups.

\begin{theorem}[Parseval-type inequality]\label{thm:parseval}
		Let $m\in\mathbb{N}$, $n_1,\ldots,n_m\in\mathbb{N}$ and set $N=n_1+\ldots+n_m$. Let $H\in G_{N,k}$  be a $k$-dimensional linear subspace in $\mathbb{R}^N=\bigoplus_{j=1}^m \R^{n_j}$. Then, for every family of functions $(f_j)_{j=1}^m\subseteq(\mathbb{R}^{n_j})$, we have 
		
\[
		\int_H \prod_{j=1}^m f_j(P_{\R^{n_j}}y)\,dy = \frac{1}{(2\pi)^{N-k}}\int_{H^\bot}\prod_{j=1}^m \widehat{f_j}(P_{\R^{n_j}}z)\,dz.
		\]
	\end{theorem}
	
	\begin{proof}
		By the definition of the Fourier transform, we have that for every $1\ls j\ls m$ and every $z_j\in\R^{n_j}$
		$$
		\widehat{f_j}(z_j)=\int_{\R^{n_j}}f_j(x_j)e^{i\langle x_j,z_j\rangle}dx_j.
		$$
		Therefore,
		\begin{eqnarray*}
			\int_{H^\perp}\prod_{j=1}^m\widehat{f_j}(P_{\R^{n_j}}z)dz&=&\int_{H^\perp}\int_{\R^{n_1}}\dots\int_{\R^{n_m}}f_1(x_1)\dots f_m(x_m)e^{i\sum_{j=1}^m\langle x_j, P_{\R^{n_j}}z\rangle}dx_m,\dots dx_1dz\cr
			&=&\int_{H^\perp}\int_{\R^{n_1}}\dots\int_{\R^{n_m}}f_1(x_1)\dots f_m(x_m)e^{i\sum_{j=1}^m\langle x_j, z\rangle}dx_m,\dots dx_1dz\cr
			&=&\int_{H^\perp}\int_{\R^N}\prod_{j=1}^mf_j(P_{\R^{n_j}}(x))e^{i\sum_{j=1}^m\langle P_{\R^{n_j}}(x), z\rangle}dxdz\cr
			&=&\int_{H^\perp}\int_{\R^N}\prod_{j=1}^mf_j(P_{\R^{n_j}}(x))e^{i\langle x, z\rangle}dxdz.\cr
		\end{eqnarray*}
		Now, by Fubini's theorem,
		\begin{align*}
		\int_{H^\perp}\int_{\R^N}\prod_{j=1}^mf_j(P_{\R^{n_j}}(x)) &  e^{i\langle x, z\rangle}dxdz =\\
			&=\int_{H^\perp}\int_{H}\int_{H^\perp}\prod_{j=1}^mf_j(P_{\R^{n_j}}(x_H+x_{H^\perp}))e^{i\langle x_H+x_{H^\perp}, z\rangle}dx_{H^\perp}dx_Hdz\\
			&=\int_{H^\perp}\int_{H}\int_{H^\perp}\prod_{j=1}^m(f_j\circ P_{\R^{n_j}})(x_H+x_{H^\perp})e^{i\langle x_{H^\perp}, z\rangle}dx_{H^\perp}dx_Hdz
		\end{align*}
		Calling, for every $x_H\in H$, $F_{x_H}:H^\perp\to\R$ the function
		$$
		F_{x_H}(x_{H^\perp})=\prod_{j=1}^m(f_j\circ P_{\R^{n_j}})(x_H+x_{H^\perp}),
		$$
		we have that for every $z\in H^\perp$
		$$
		\int_{H^\perp}\prod_{j=1}^m(f_j\circ P_{\R^{n_j}})(x_H+x_{H^\perp})e^{i\langle x_{H^\perp}, z\rangle}dx_{H^\perp}=\widehat{F_{x_H}}(z)
		$$
		and then, by Fubini's theorem and Fourier's inversion formula \eqref{eq:FourierInversionAt0},
		\begin{eqnarray*}
			&&\int_{H^\perp}\int_{H}\int_{H^\perp}\prod_{j=1}^m(f_j\circ P_{\R^{n_j}})(x_H+x_{H^{\perp}})e^{i\langle x_{H^\perp}, z\rangle}dx_{H^{\perp}}dx_Hdz\cr
			&=&\int_{H^\perp}\int_{H}\widehat{F_{x_H}}(z)dx_Hdz=\int_{H}\int_{H^\perp}\widehat{F_{x_H}}(z)dzdx_H=(2\pi)^{N-k}\int_H F_{x_H}(0)dx_H\cr
			&=&(2\pi)^{N-k}\int_{H}\prod_{j=1}^m(f_j\circ P_{\R^{n_j}})(x_H)dx_H=(2\pi)^{N-k}\int_H \prod_{j=1}^m f_j(P_{\R^{n_j}}y)dy,
		\end{eqnarray*}
which completes the proof.
	\end{proof}

As a consequence, combining Theorem \ref{thm:parseval} with Proposition \ref{prop:H^bot} and the Brascamp-Lieb inequality (Theorem \ref{thm:GeometricBrascampLieb}), we obtain the following corollary.
	\begin{corollary}\label{cor:UpperBoundParsevalAndBrascampLieb}
		Let $1\leq n\leq m$ and $(c_j,v_j)_{j=1}^m\subseteq(0,1]\times S^{n-1}$ providing a decomposition of the identity in $\R^n$. Let $1\leq k\leq m$, $H=\textrm{span}\{v_j\,:\,1\leq j\leq k\}$, and $(f_j)_{j=1}^k\subseteq {\mathcal S}(\mathbb{R})$. Then,
		$$
		\int_{H}\prod_{j=1}^kf_j(\sqrt{c_j}\langle x,v_j\rangle)dx\leq\frac{1}{(2\pi)^{k-n}}\prod_{j=1}^k\left(\int_{\R} |\widehat{f_j}^\frac{1}{1-c_j}(\sqrt{1-c_j}t)|dt\right)^{1-c_j},
		$$
where, if for some $1\ls j\ls k$ we have that $c_j=1$, the corresponding factor on the right-hand side is understood as $\displaystyle{|\widehat{f_j}(0)|=\left|\int_{\R}f_j(x)dx\right|}.$
	\end{corollary}
	
	\begin{proof}
		Let us identify $\R^n$ with $H_1:=\textrm{span}\{e_j\,:\,1\leq j\leq n\}\subseteq\R^{m}$, where $(e_j)_{j=1}^{m}$ denotes the canonical basis in $\R^{m}$. By Proposition \ref{prop:H^bot} there exists $(x_j)_{j=1}^m$, an orthonormal basis of $\mathbb{R}^m$ 
such that for every $1\leq j\leq m$
		\[
		P_{H_1} x_j= \sqrt{c_j}v_j 
		\]
		for every $1\ls j\ls m$. In particular, $(x_j)_{j=1}^k$ is an orthonormal basis of $\R^k=\textrm{span}\{x_j\,:\,1\leq j\leq k\}$ and, since $H=\textrm{span}\{v_j\,:\,1\leq j\leq k\}\subseteq H_1$, we have that for every $1\leq j\leq k$
$$
P_{H} x_j= \sqrt{c_j}v_j
$$
Considering $H^\perp$, the orthogonal complement of $H$ in $\R^k=\textrm{span}\{x_j\,:\,1\leq j\leq k\}$, there exist $(w_j)_{j=1}^k\subseteq S^{k-1}\cap H^\perp$ such that
$$
P_{H^\bot}x_j=\sqrt{1-c_j}w_j
$$
for every $1\leq j\leq k$ and
		\[
		\mathrm{Id}_{H^\bot} = \sum_{j=1}^k (1-c_j)w_j\otimes w_j.
		\]
		
		
		Writing coordinates with respect to the orthonormal basis of $\mathbb{R}^{k}$ given by $(x_j)_{j=1}^{k}$, we have that for every $x\in H$
		$$
		\langle x,x_j\rangle=\langle x,\sqrt{c_j}v_j\rangle\quad\forall 1\ls j\ls k
		$$
		and for every $y\in H^\perp$
		$$
		\langle y,x_j\rangle=\langle y,\sqrt{1-c_j}w_j\rangle,\quad\forall 1\ls j\ls k
		$$
		By Theorem \ref{thm:parseval} we have that
		$$
		\int_{H}\prod_{j=1}^kf_j(\sqrt{c_j}\langle x,v_j\rangle)dx=\frac{1}{(2\pi)^{k-n}}\int_{H^\bot}\prod_{j=1}^k \widehat{f_j}(\sqrt{1-c_j}\langle y,w_j\rangle)dy.
		$$
		Since
		\[
		\mathrm{Id}_{H^\perp} = \sum_{j=1}^k (1-c_j)w_j\otimes w_j,
		\]
		by the geometric Brascamp-Lieb inequality we have that
\begin{align*}
		\int_{H^\bot}\prod_{j=1}^k \widehat{f_j}(\sqrt{1-c_j}\langle y,w_j\rangle)dy &\leq\int_{H^\bot}\prod_{j=1}^k |\widehat{f_j}(\sqrt{1-c_j}\langle y,w_j\rangle)|dy\\
		        &\leq\prod_{j=1}^k \left(\int_{\R}|\widehat{f_j}(\sqrt{1-c_j}t)|^\frac{1}{1-c_j}dt\right)^{1-c_j}.
\end{align*}
Let us point out that if for some $1\ls j\ls k$ we have that $c_j=1$, then $|\widehat{f_j}(\sqrt{1-c_j}\langle y,w_j\rangle)|=|\widehat{f_j}(0)|$ and the term $(1-c_j)w_j\otimes w_j$ in the decomposition of the identity in $H^\perp$ is 0. Thus, the application of the Brascamp-Lieb inequality gives the same upper bound with the term corresponding to $c_j=1$ being understood as $|\widehat{f_j}(0)|$ and, following this convention,
		\[
		\int_{H}\prod_{j=1}^kf_j(\sqrt{c_j}\langle x,v_j\rangle)dx\leq\frac{1}{(2\pi)^{k-n}}\prod_{j=1}^k\left(\int_{\R} |\widehat{f_j}^\frac{1}{1-c_j}(\sqrt{1-c_j}t)|dt\right)^{1-c_j}.\qedhere
		\]
	\end{proof}
	
Let us compare the upper bound obtained in the previous corollary with the upper bound obtained by directly applying Brascamp-Lieb inequality.

	\begin{remark}\label{rem:UpperBoundParseval-BrascampLieb}
		Notice that if for every $1\leq j\leq k$ we have that $\frac{1}{1-c_j}\geq2$, which happens if and only if $\frac{1}{2}\ls c_j<1$, then, by Beckner's sharp Haussdorff-Young inequality \eqref{eq:Beckner}, we have that
		\begin{eqnarray*}
			\left(\int_{\R} |\widehat{f_j}^\frac{1}{1-c_j}(\sqrt{1-c_j}t)|dt\right)^{1-c_j}&=&\left(\frac{1}{\sqrt{1-c_j}}\right)^{1-c_j}\left(\int_{\R}|\widehat{f_j}(s)|^\frac{1}{1-c_j}ds\right)^{1-c_j}\cr
			 &\leq&(2\pi)^{1-c_j}\left(\frac{1}{\sqrt{1-c_j}}\right)^{1-c_j}\left(\frac{\left(\frac{1}{\sqrt{c_j}}\right)^{c_j}}{\left(\frac{1}{\sqrt{1-c_j}}\right)^{1-c_j}}\right)\left(\int_{\R}|f_j(s)|^\frac{1}{c_j}ds\right)^{c_j}\cr
			&=&(2\pi)^{1-c_j}\left(\frac{1}{\sqrt{c_j}}\right)^{c_j}\left(\int_{\R}|f_j(s)|^\frac{1}{c_j}ds\right)^{c_j}\cr
		\end{eqnarray*}
and, if $c_j=1$,
$$
\displaystyle{|\widehat{f_j}(0)|=\left|\int_{\R}f_j(s)ds\right|\leq\int_{\R}|f_j(s)ds|=(2\pi)^{1-c_j}\left(\frac{1}{\sqrt{c_j}}\right)^{c_j}\left(\int_{\R}|f_j(s)|^\frac{1}{c_j}ds\right)^{c_j}.}
$$
		Therefore, taking into account the established convention whenever $c_j=1$,
		\begin{eqnarray*}
			\int_{\R^n}\prod_{j=1}^kf_j(\sqrt{c_j}\langle x,v_j\rangle)dx&\leq&\frac{1}{(2\pi)^{k-n}}\prod_{j=1}^k(2\pi)^{1-c_j}\left(\frac{1}{\sqrt{c_j}}\right)^{c_j}\left(\int_{\R}|f_j(s)|^\frac{1}{c_j}ds\right)^{c_j}\cr
			&=&\prod_{j=1}^k\left(\frac{1}{\sqrt{c_j}}\right)^{c_j}\left(\int_{\R}|f_j(s)|^\frac{1}{c_j}ds\right)^{c_j}.
		\end{eqnarray*}
		On the other hand, applying directly Brascamp-Lieb inequality we obtain
		\begin{eqnarray*}
			\int_{\R^n}\prod_{j=1}^kf_j(\sqrt{c_j}\langle x,v_j\rangle)dx&\leq&\prod_{j=1}^k\left(\int_{\R} |f_j(\sqrt{c_j}t)|^\frac{1}{c_j}dt\right)^{c_j}\cr
			&=&\prod_{j=1}^k\left(\frac{1}{\sqrt{c_j}}\right)^{c_j}\left(\int_{\R} |f_j(t)|^\frac{1}{c_j}dt\right)^{c_j}.\cr
		\end{eqnarray*}
		Therefore, the estimate we obtain in Corollary \ref{cor:UpperBoundParsevalAndBrascampLieb} applying Brascamp-Lieb inequality to the Fourier transforms is better than the one we obtain applying Brascamp-Lieb inequality to the functions itself whenever $\frac{1}{2}\ls c_j\ls 1$ for every $1\leq j\leq k$.
		\end{remark}

\section{Volumes of sections of convex bodies in John's position}\label{sec:VolumesOfSections}

In this section we are going to provide estimates for the volumes of sections of convex bodies in John's position. We will provide estimates for centrally symmetric convex bodies in John's position and for general non-necessarily symmetric convex bodies in John's position.
\subsection{Sections of centrally symmetric convex bodies}\label{subsec:Symmetric}

In this section we will focus on centrally symmetric convex bodies $K\subseteq\R^n$ in John's position, with associated decomposition of the identity $(c_j,v_j)_{j=1}^m\in (0,1]\times(\partial K\cap S^{n-1})$. We will use the notation established in Section \ref{sec:GeneralSetting}. First, we show that it is not true that we can obtain the analogous bound to \eqref{eq:Ball-2} for the $k$-dimensional sections of $K$ that Ball obtained for the cube, i.e. It is not true that for every $1\ls k\ls n$ with $k\gr \frac{n}{2}$,
	\[
	\vol_k(K\cap H) \ls 2^{\frac{n-k}{2}}\vol_k(B_\infty^k).
	\]
This is showcased by the following example.	
\begin{theorem}\label{thm:counterexample}
There exists $k\in\mathbb{N}$ arbitrarily large such that for every $n\in\mathbb{N}$ with $\frac{n}{2}< k< n$ there exists a convex body $L$ in $\mathbb{R}^n$ which is in John's position, while
\[
\vol_k(L\cap H)=\left(\frac{n}{k}\right)^\frac{k}{2}\vol_k(B_\infty^k)>2^{\frac{n-k}{2}}\vol_k(B_\infty^k),
\]
where $H=\textrm{span}\{e_1,\dots e_k\}$.
\end{theorem}
\begin{proof}
Let $k\in\mathbb{N}$ be such that there exists a $k\times k$ Hadamard matrix (for instance, let $k$ be any power of 2) and take any $k< n< 2k$. Let $m=2k> n$ and define unit vectors $v_1,\ldots,v_m$ as follows: Start with an arbitrary $k\times k$ Hadamard matrix $H_k$ with columns $\eta_1,\ldots,\eta_k$, that is, the components of each $\eta_j$ are $\pm 1$ and $\langle\eta_i,\eta_j\rangle=0$ for every $i\neq j$. The matrix
\[
H_{m} = \begin{pmatrix}
H_k & H_k\\
H_k & -H_k
\end{pmatrix}
\]
is then a $m\times m$ Hadamard matrix. Let $B$ be the upper $n\times m$ submatrix of $H_m$. If $\beta_1,\ldots,\beta_m$ are the columns of $B$, we define $v_j:=n^{-\frac{1}{2}}\beta_j$ for every $j=1,\ldots,m$. Clearly, the $v_j$'s are unit vectors in $\mathbb{R}^n$ and if we let $c_j=n/(2k)$ for every $j=1,\ldots,m$ we can check that
\[
\sum_{j=1}^m c_jv_j\otimes v_j= \frac{1}{2k}\sum_{j=1}^m \beta_j\otimes\beta_j = \mathrm{Id}_n.
\]
This shows that $(c_j,v_j)_{j=1}^m$ induce a decomposition of the identity in $\mathbb{R}^n$. The symmetric convex polytope $L=\{x\in\mathbb{R}^n:|\langle x, v_j\rangle|\ls 1, j=1,\ldots,m\}$ is then in John's position and note that, identifying $H$ with $\R^k$,
\begin{align*}
L\cap H &= \left\{x\in H:|\langle x,P_{H}v_j\rangle|\ls 1,\, j=1,\ldots,m\right\}\\
                   &= \left\{x\in H:|\langle x,P_{H}\beta_j\rangle|\ls \sqrt{n},\, j=1,\ldots,m\right\}\\
                   &= \left\{x\in H:|\langle x,\eta_j\rangle|\ls \sqrt{n},\, j=1,\ldots,k\right\},
\end{align*}
since, by our construction, $P_{H}\beta_j=P_{H}\beta_{k+j} = \eta_j$ for every $j=1,\ldots,k$. Note also that if we let $W=k^{-\frac{1}{2}}H_k$, then we have that
\[
W^TW = \frac{1}{k}H_k^TH_k = \mathrm{Id}_k.
\]
We can ultimately write
\begin{align*}
\vol_k(L\cap H) &= \vol_k(\{x\in \R^k:\|H_k^Tx\|_\infty\ls \sqrt{n}\})= \vol_k\left(\left\{Wy\in\mathbb{R}^k:\|y\|_\infty\ls \sqrt{\frac{n}{k}}\right\}\right)\\
&=\vol_k\left(W\left(\sqrt{\frac{n}{k}}B_\infty^k\right)\right) = \left(\frac{n}{k}\right)^\frac{k}{2}\vol_k(B_\infty^k),
\end{align*}
since $\mathrm{det}(W^TW)=1$. Finally, notice that the function $2^{x-1}$ is strictly convex in $\R$ and then $x>2^{x-1}$ for every $x\in(1,2)$. Consequently, if $\frac{n}{2}< k< n$ we have that $1<\frac{n}{k}<2$ and then
\[
\frac{n}{k}>2^{\frac{n}{k}-1}\Leftrightarrow\left(\frac{n}{k}\right)^\frac{k}{2}>2^{\frac{n-k}{2}}.\qedhere
\]
\end{proof}

In the following Theorem, we show that, even though the analogous estimate to \eqref{eq:Ball-2} is not true in general, we have the following analogous estimate to \eqref{eq:BallsEstimate}, for every centrally symmetric convex body in John's position.
\begin{theorem}\label{thm:UpperBoundVolume}
Let $K\subseteq\R^n$ be a centrally symmetric convex body in John's position and let  $(c_j,v_j)_{j=1}^m\subseteq (0,1]\times(\partial K\cap S^{n-1})$  be its associated decomposition of the identity. Let $H\in G_{n,k}$ be a $k$-dimensional subspace and let $J=\{1\leq j\leq m\,:\,P_Hv_j\neq0\}$. Then
$$
\vol_k(K\cap H)\leq \vol_k(B_\infty^k)\prod_{j\in J}\left(
\frac{I_{\frac{1}{1-\tilde{c}_j}}}{\pi\sqrt{1-\tilde{c}_j}}\right)^{1-\tilde{c}_j}c_j^{\frac{\tilde c_j}{2}},
$$
where, for any $p>1$, $\displaystyle{I_p=\int_\R\left|\frac{\sin x}{x}\right|^pdx}$, and for every $j\in J$, $\tilde{c}_j=c_j\Vert P_Hv_j\Vert_2^2$. If for some $j\in J$ we have that $\tilde{c}_j=1$, then necessarily $c_j=\Vert P_H v_j\Vert_2=1$ and the corresponding factor is understood as 1.
\end{theorem}

\begin{proof}
Let, for every $j\in J$, $f_j:\R\to[0,\infty)$ be the function $f_j=\mathds{1}_{\left[-\sqrt{\tilde{c}_j}t_j,\sqrt{\tilde{c}_j}t_j\right]}$. We have that
\begin{eqnarray*}
\vol_k(K\cap H)&\leq&\vol_k(L\cap H)=\int_H\prod_{j\in J}\mathds{1}_{\left[-t_j,t_j\right]}\left(\langle x, u_j\rangle\right)dx\cr
&=&\int_H\prod_{j\in J}\mathds{1}_{\left[-\sqrt{\tilde{c}_j}t_j,\sqrt{\tilde{c}_j}t_j\right]}\left(\sqrt{\tilde{c}_j}\langle x,u_j\rangle\right)dx
\end{eqnarray*}
Identifying $H$ with $\R^k$, and taking into account that $(\tilde{c}_j,u_j)_{j\in J}\subseteq (0,1]\times(\partial K\cap S^{n-1})$ provide a decomposition of the identity in $H$, by Corollary \ref{cor:UpperBoundParsevalAndBrascampLieb} (see also the Appendix), we have that
\begin{eqnarray*}
\vol_k(K\cap H)&\leq&\int_H\prod_{j\in J}\mathds{1}_{\left[-\sqrt{\tilde{c}_j}t_j,\sqrt{\tilde{c}_j}t_j\right]}\left(\sqrt{\tilde{c}_j}\langle x,u_j\rangle\right)dx\cr
&\leq&\frac{1}{(2\pi)^{m_0-k}}\prod_{j\in J}\left(\int_\R|\widehat{f_j}(\sqrt{1-\tilde{c}_j}t)|^\frac{1}{1-\tilde{c}_j}dt\right)^{1-\tilde{c}_j},
\end{eqnarray*}
where, if $j\not\in\tilde{J}=\{j\in J\,:\,\tilde{c}_j\neq1\}$, i.e., if $\tilde{c}_j=1$, the corresponding factor is understood as
$$|\widehat{f_j}(0)|=\left|\int_\R f_j(x)dx\right|=2\sqrt{\tilde{c_j}}t_j=2\sqrt{c_j}=2,$$
since, if $\tilde{c}_j=1$, then necessarily $c_j=1$.

By \eqref{eq:FourierTransformCharacteristic}, we have
$$
\widehat{f_j}(y)=\begin{cases}\frac{2\sin(\sqrt{\tilde{c}_j}t_jy)}{y}=\frac{2\sin(\sqrt{c_j}t)}{y}&\textrm{ if }y\in\mathbb{R}\setminus\{0\}\\ 2\sqrt{\tilde{c_j}}t_j=2\sqrt{c_j}&\textrm{ if }y=0,\end{cases}
$$
and we obtain that
\begin{eqnarray*}
\vol_k(K\cap H)&\leq&\frac{1}{(2\pi)^{m_0-k}}\prod_{j\in J}\left(\int_\R|\widehat{f_j}(\sqrt{1-\tilde{c}_j}t)|^\frac{1}{1-\tilde{c}_j}dt\right)^{1-\tilde{c}_j}\cr
&=&\frac{2^{m_0}}{(2\pi)^{m_0-k}}\prod_{j\in \tilde{J}}\left(\int_\R\left|\frac{\sin(\sqrt{c_j}\sqrt{1-\tilde{c}_j}t)}{\sqrt{1-\tilde{c_j}}t}\right|^\frac{1}{1-\tilde{c_j}}dt\right)^{1-\tilde{c}_j}\cr
&=&\frac{\vol_k(B_\infty^k)}{\pi^{m_0-k}}\prod_{j\in \tilde{J}}\left(\int_\R\left|\frac{\sin(\sqrt{c_j}\sqrt{1-\tilde{c}_j}t)}{\sqrt{1-\tilde{c_j}}t}\right|^\frac{1}{1-\tilde{c_j}}dt\right)^{1-\tilde{c}_j}\cr
&=&\vol_k(B_\infty^k)\prod_{j\in \tilde{J}}\left(\frac{1}{\pi}\int_\R\left|\frac{\sin(u)}{u}\right|^\frac{1}{1-\tilde{c_j}}\frac{du}{\sqrt{1-\tilde{c}_j}}\right)^{1-\tilde{c}_j}c_j^\frac{\tilde{c}_j}{2}\cr
&=&\vol_k(B_\infty^k)\prod_{j\in \tilde{J}}\left(
\frac{I_{\frac{1}{1-\tilde{c}_j}}}{\pi\sqrt{1-\tilde{c}_j}}\right)^{1-\tilde{c}_j}c_j^{\frac{\tilde c_j}{2}},
\end{eqnarray*}
as we wanted to prove.
\end{proof}

In order to provide an upper bound for the volume of $k$-dimensional sections of convex bodies which is independent of of the values of $c_j$ and $\Vert P_H v_j\Vert_2$, we will make use of the following proposition:

\begin{proposition}\label{prop:exact-max}
Let $m\in\mathbb N$ and let $k\in\mathbb Z$, $n\in\mathbb R$. Assume that the compact set
\[
M=\Big\{(\tilde x_j,x_j)_{j=1}^m:\ \tfrac12\le \tilde x_j\le x_j\le 1,\,
\sum_{j=1}^m \tilde x_j=k,\ \sum_{j=1}^m x_j\le n\Big\}.
\]
is non-empty and let
\[
f(\tilde x,x)=\sum_{j=1}^m \tilde x_j\log x_j,\quad\forall (\tilde x,x)=(\tilde x_j,x_j)_{j=1}^m\in M.
\]
Then,
\[
\max_{(\tilde x,x)\in M} f(\tilde x,x)=
\begin{cases}
0 &\textrm{ if } n\ge m,\\
(m-k)\,\log\!\Big(\dfrac{m+n-2k}{2(m-k)}\Big), & \textrm{ if }n<m,
\end{cases}
\]
with the convention that the second line equals $0$ when $k=m$.


Moreover, if $n<m$,  calling $r:=2k-m\in\mathbb Z$, which satisfies $0\le r\le m$, the maximum is attained at $(\tilde x,x)_{j=1}^m$ satisfying
\[
(\tilde x_j,x_j)=(1,1)\quad \text{for } j=1,\dots,r,
\qquad
(\tilde x_j,x_j)=\Big(\tfrac12,c\Big)\quad \text{for } j=r+1,\dots,m,
\]
where
\[
c=\frac{n-r}{m-r}=\frac{m+n-2k}{2(m-k)}\in\Big[\tfrac12,1\Big).
\]
\end{proposition}

\begin{proof}
First of all, notice that if $M\neq\emptyset$, then there exists $(\tilde{x},x)\in M$ and then
$$
\frac{m}{2}\ls\sum_{j=1}^m\tilde{x}_j\ls m.
$$
Thus, $\frac{m}{2}\ls k\ls m$ and then $0\ls r\ls m$. Notice also that, since $0<\frac{1}{2}\ls \tilde{x}_j\ls x_j\ls 1$ we have that $f(\tilde{x},x)\ls0$ for every $(\tilde{x},x)\in M$. Therefore, since $M$ is compact, $f$ attains a maximum on $M$.

If $n\gr m$, then there exists $(\tilde{x},x)\in M$ with $x_j=1$ for every $1\ls j\ls m$ and $f(\tilde{x},x)=0$. Thus, $\max_{(\tilde{x},x)\in M}f(\tilde{x},x)=0$.

Let us assume that $n<m$. Let us see that if the maximum of $f$ is attained at $(\tilde{x},x)\in M$, then $\displaystyle{\sum_{j=1}^m x_j=n}$. Assume that this is not the case. Then, since $\displaystyle{\sum_{j=1}^m x_j<n<m}$, there exists $1\ls j_0\ls m$ such that $x_{j_0}<1$ and we can take $\displaystyle{0<\varepsilon<n-\sum_{j=1}^mx_j}$ such that, calling $x_{j_0}^\prime=x_{j_0}+\varepsilon$  and $x_j^\prime =x_j $ if $j\neq j_0$, we have that
$(\tilde{x}_j,x_j^\prime)_{j=1}^m\in M$, and
$$
f(\tilde{x},x^\prime)-f(\tilde{x},x)=\tilde{x}_{j_0}(\log(x_{j_0}+\varepsilon)-\log(x_{j_0}))>0,
$$
which contradicts that the maximum of $f$ is attained at $(\tilde{x},x)$.

Let us now call, for every $1\ls j\ls m$, $\delta_j:=\tilde{x}_j-\frac{1}{2}$ and notice that, $\displaystyle{\sum_{j=1}^m\tilde{x}_j=k}$ if and only if $\displaystyle{\sum_{j=1}^m\delta_j=k-\frac{m}{2}=\frac{r}{2}}$ and that $\frac{1}{2}\ls \tilde{x}_j\ls x_j\ls 1$ for every $1\ls j\ls m$ if and only if
$$
0\ls \delta_j\ls x_j-\frac{1}{2},\quad\textrm{and}\quad \frac{1}{2}\ls x_j\ls 1
$$
for every $1\ls j\ls m$. Therefore, calling
$$
f_1(\delta,x)=\frac{1}{2}\sum_{j=1}^m\log x_j+\sum_{j=1}^m\delta_j\log x_j,
$$
we have that the maximum of $f$ on $M$ is the same as the maximum of $f_1$ on
$$
M_1=\left\{(\delta_j,x_j)_{j=1}^m\,:\,0\ls \delta_j\ls x_j-\frac{1}{2},\,\frac{1}{2}\ls x_j\ls 1,\,\sum_{j=1}^m\delta_j=\frac{r}{2},\,\sum_{j=1}^m x_j\ls n\right\}
$$
and that the maximum of $f_1$ is attained at $(\delta_j,x_j)_{j=1}^m\in M_1$ if and only if the maximum of $f$ is attained at $\left(\frac{1}{2}+\delta_j,x_j\right)_{j=1}^m\in M$. Thus, if the maximum of $f_1$ at $M_1$ is attained at $(\delta,x)\in M_1$, necessarily $\displaystyle{\sum_{j=1}^m x_j=n}$.

Let us now call, for any $(\delta,x)\in M_1$, $A=\{1\ls j\ls m\,:\,x_j=1\}$, $B=\{1,\dots,m\}\setminus  A$, and let $a$ and $b$ be the cardinalities of $A$ and $B$ respectively. We are going to see that if $(\delta,x)$ is a maximizer of $f_1$ on $M_1$, then $a\gr r$.

Let us assume that $a<r$ and obtain a contradiction. On the one hand, for every $j\in A$ we have that $x_j=1$. Thus, $\delta_j\le \frac12$ for every $j\in A$ and then
\[
\sum_{j\in A}\delta_j\le \frac a2.
\]
Since $\sum_{j=1}^m\delta_j=\frac r2$,
\[
\sum_{j\in B}\delta_j=\frac r2-\sum_{j\in A}\delta_j\ge \frac r2-\frac a2=\frac{r-a}{2}\ge \frac12,
\]
since $r-a\in\mathbb Z$ and $a<r$, so $r-a\gr1$.

On the other hand, for every $j\in B$ we have $x_j<1$, hence
\[
\delta_j\le x_j-\frac12<\frac12.
\]
Therefore the inequality $\displaystyle{\sum_{j\in B}\delta_j\ge \frac12}$ forces the existence of two different indices $j_1\neq j_2$ in $B$
with $\delta_{j_1}>0$ and $\delta_{j_2}>0$.

Let  $\tilde x_\ell=\frac12+\delta_\ell$ for $\ell\in\{j_1,j_2\}$, and denote
\[
d_\ell:=x_\ell-\tilde x_\ell=x_\ell-\Big(\frac12+\delta_\ell\Big)> 0,\qquad \ell\in\{i,j\}.
\]
Consider the two-sided perturbation (keeping all other coordinates fixed)
\[
x_{j_1}(\varepsilon)=x_{j_1}+\varepsilon,\quad \tilde x_{j_1}(\varepsilon)=\tilde x_{j_1}+\varepsilon,\qquad
x_{j_2}(\varepsilon)=x_{j_2}-\varepsilon,\quad \tilde x_{j_2}(\varepsilon)=\tilde x_{j_2}-\varepsilon.
\]
Thus, $\delta_{j_1}(\varepsilon)=\delta_{j_1}+\varepsilon$ and $\delta_{j_2}(\varepsilon)=\delta_{j_2}-\varepsilon$.
Let
$$
\varepsilon_0:=\min\{1-x_{j_1},1-x_{j_2},\delta_{j_1},\delta_{j_2}\}>0
$$
Then, for every $\varepsilon\in[-\varepsilon_0,\varepsilon_0]$.

\[
\frac12\le \tilde x_{j_1}(\varepsilon)\le x_{j_1}(\varepsilon)\le 1,\qquad
\frac12\le \tilde x_{j_2}(\varepsilon)\le x_{j_2}(\varepsilon)\le 1,
\]

Moreover,
\[
\sum_{j=1}^m x_j(\varepsilon)=\sum_{j=1}^m x_j=n,
\qquad
\sum_{j=1}^m \tilde x_j(\varepsilon)=\sum_{j=1}^m \tilde x_j=k,
\]
and $(\tilde{x}(\varepsilon),x(\varepsilon))\in M$. Equivalently, $(\delta(\varepsilon),x(\varepsilon))\in M_1$ for every $\varepsilon\in[-\varepsilon_0,\varepsilon_0]$.

Let us consider, for every $\varepsilon\in[-\varepsilon_0,\varepsilon_0]$, the function
\[
\Psi(\varepsilon):=\tilde x_{j_1}(\varepsilon)\log x_{j_1}(\varepsilon)+\tilde x_{j_2}(\varepsilon)\log x_{j_2}(\varepsilon)
=(x_{j_1}(\varepsilon)-d_{j_1})\log x_{j_1}(\varepsilon)+(x_{j_2}(\varepsilon)-d_{j_2})\log x_{j_2}(\varepsilon).
\]
For any fixed $d>0$, the function
\[
\varphi_d(t):=(t-d)\log t,\qquad t\in\Big[\frac12,1\Big],
\]
satisfies that
\[
\varphi_d''(t)=\frac1t+\frac{d}{t^2}=\frac{t+d}{t^2}>0,\qquad\textrm{for every } t\in\Big[\frac12,1\Big].
\]
Hence, $\varphi_d$ is strictly convex on $[\frac12,1]$ and then
\[
\Psi(\varepsilon)=\varphi_{d_{j_1}}(x_{j_1}(\varepsilon))+\varphi_{d_{j_2}}(x_{j_2}(\varepsilon))
\]
is strictly convex on the interval $[-\varepsilon_0,\varepsilon_0]$. Since $0$ lies in the interior of this interval, strict convexity implies
\[
\max\{\Psi(-\varepsilon_0),\Psi(\varepsilon_0)\}>\Psi(0).
\]
Thus, choosing $\varepsilon^\prime=\varepsilon_0$ or $\varepsilon=-\varepsilon_0$ produces a point $(\tilde{x}(\varepsilon^\prime),x(\varepsilon^\prime))\in M$ and a point $(\delta(\varepsilon^\prime),x(\varepsilon^\prime))\in M_1$ with a strictly larger value of $f$ (resp. $f_1$),  contradicting the maximality of $(\delta,x)$. Hence, it must be $a\ge r$.

Let us now see that for any $(\delta,x)\in M_1$ with $a\ge r$, there exists $\delta^\star$ such that $(\delta^\star,x)\in M_1$ and $\displaystyle{\sum_{j=1}^m\delta_j^\star\log x_j=0}$ and, consequently,
$$
f_1(\delta,x)=\frac{1}{2}\sum_{j=1}^m\log x_j+\sum_{j=1}^m\delta_j\log x_j\ls\frac{1}{2}\sum_{j=1}^m\log x_j=f_1(\delta^\star,x).
$$
Therefore, if $f$ is maximized on $M$ at $(\tilde{x},x)$, we will necessarily have that $\delta_j=0$ and  $\tilde{x}_j=\frac{1}{2}$ at every $j$ such that $x_j\neq1$. In order to see it, let $(\delta,x)\in M_1$ with $a\ge r$. For every $j\in A$ we have that $x_j=1$ and then $x_j-\frac{1}{2}=\frac{1}{2}$. Let $S\subseteq A$ with cardinality $r\ls a$ and take
\[
\delta_j^\star=
\begin{cases}
\frac12, & \textrm{ if }j\in S,\\
0, & \textrm{ if }j\notin S.
\end{cases}
\]
Then $\sum_{j=1}^m\delta_j^\star=\frac{r}{2}$, and for each $1\ls j\ls m$ we have
$0\le \delta_j^\star\le x_j-\frac12$, since this inequality is trivial if $j\not\in S$ and $x_j=1$ if $j\in S\subseteq A$. Hence $(\delta^\star,x)\in M_1$.

Moreover, $\delta_j^\star\neq 0$ only for $j\in S\subseteq A$, and for those indices $\log x_j=\log 1=0$; therefore
\[
\sum_{j=1}^m \delta_j^\star\log x_j=0.
\]

Consequently, if $f$ attains its maximum at $(\tilde{x},x)\in M$, necessarily $\displaystyle{\sum_{j=1}^m x_j=n}$, $r\ls a$,  and $\tilde{x}_j=\frac{1}{2}$ for every $j\in B$, and

\[
f(\tilde x,x)=\frac12\sum_{j\in B}\log x_j.
\]

Since, $\displaystyle{\sum_{j=1}^m x_j=n}$ and $x_j=1$ for every $j\in A$, we have
\[
\sum_{j\in B} x_j=n-a.
\]
Taking into account that $b=m-a$, by the concavity of the logarithm,
\[
\frac{1}{m-a}\sum_{j\in B}\log x_j\le
\log\!\Big(\frac{1}{m-a}\sum_{j\in B}x_j\Big)=\log\!\Big(\frac{n-a}{m-a}\Big),
\]
so
\[
f(\tilde x,x)\le \frac{m-a}{2}\log\left(\frac{n-a}{m-a}\right),
\]
with equality if and only if $x_j$ is constant on $B$, i.e. $x_j=\frac{n-a}{m-a}$ for all $j\in B$.

Let us now call $\Phi(a):= \frac{m-a}{2}\log\Big(\frac{n-a}{m-a}\Big)$ and let us see that $\Phi$ is decreasing in $a$. In order to do it, recall that $b=m-a$ and let us define $c:=\frac{n-a}{m-a}=\frac{n-m+b}{b}\in (0,1)$, since $n<m$. Notice that
$$
\Phi(a)=\frac{1}{2}G(m-a),
$$
where $G(b)=b\log c=b\log\left(\frac{n-m+b}{b}\right)$. A direct computation yields
$$
G^\prime(b)=\log(c)+\frac{1}{c}-1.
$$
Calling $h(c)=\log(c)+\frac{1}{c}-1$ for every $c\in (0,1]$ we have that
$$
h^\prime(c)=\frac{1}{c}-\frac{1}{c^2}=\frac{c-1}{c^2}<0
$$
for every $c\in (0,1)$. Thus, $h$ is strictly decreasing in $(0,1]$ and, since $h(1)=0$, we have that $h(c)>0$ for every $c\in (0,1)$. Consequently, $G$ is strictly increasing and then $\Phi$ is strictly decreasing.

Consequently, if $f$ attains its maximum at $(\tilde{x},x)\in M$, taking into account that $r\ls a$,
\[
f(\tilde x,x)\le \frac{m-a}{2}\log\left(\frac{n-a}{m-a}\right)=\Phi(a)\ls\Phi(r)=\frac{m-r}{2}\log\left(\frac{n-r}{m-r}\right),
\]
with equality if and only if $a=r$ and $x_j=\frac{n-r}{m-r}$ for every $j\in B$. Since $m-r=2(n-k)$ and $n-r=m+n-2k$, we obtain that if $n<m$
$$
\max_{(\tilde{x},x)\in M}f((\tilde{x},x)\ls (m-k)\log\left(\frac{m+n-2k}{2(m-k)}\right).
$$
Besides, letting $a=r$ and taking
\[
x_j=1\ (j=1,\dots,r),\qquad x_j=c:=\frac{n-r}{m-r}\ (j=r+1,\dots,m),
\]
\[
\tilde x_j=1\ (j=1,\dots,r),\qquad \tilde x_j=\frac12\ (j=r+1,\dots,m),
\]
we have that $\displaystyle{\sum_{j=1}^m\tilde x_j=r+\frac{m-r}{2}=k}$, $\displaystyle{\sum_{j=1}^m x_j=r+(m-r)c=n}$, and $c\in[\frac12,1)$ is equivalent to $n\ge k$ and $n<m$,
which holds whenever $M\neq\varnothing$ and $n<m$ and
$$
f(\tilde{x},x)= (m-k)\log\left(\frac{m+n-2k}{2(m-k)}\right).
$$
Thus, the maximum is attained at this point.
\end{proof}

As a corollary, using Ball's integral estimate \eqref{eq:BallsIntegralEstimate}, we have the following estimate whenever $\tilde{c}_j\geq \frac{1}{2}$ for every $j\in J$.
\begin{corollary}\label{cor:UpperBoundVolumeBallsIntegralEstimate}
Let $K\subseteq\R^n$ be a centrally symmetric convex body in John's position and let  $(c_j,v_j)_{j=1}^m\subseteq (0,1]\times(\partial K\cap S^{n-1})$  be its associated decomposition of the identity. Let $H\in G_{n,k}$ be a $k$-dimensional subspace and let $J=\{1\leq j\leq m\,:\,P_Hv_j\neq0\}$. Assume that for every $j\in J$, $\tilde{c}_j=c_j\Vert P_Hv_j\Vert_2^2\geq\frac{1}{2}$ and that $m_0>n$. Then,
$$
\vol_k(K\cap H)\leq2^\frac{m_0-k}{2}\vol_k(B_\infty^k)\prod_{j\in J}c_j^\frac{c_j\Vert P_Hv_j\Vert_2^2}{2}\ls \left(\frac{m_0+n-2k}{m_0-k}\right)^{\frac{m_0-k}{2}}\vol_k(B_\infty^k).
$$
\end{corollary}

\begin{proof}
If $\tilde{c}_j\geq\frac{1}{2}$ then $\frac{1}{1-\tilde{c}_j}\geq2$ for every $j\in \tilde{J}=\{j\in J\,:\,\tilde{c}_j\neq1\}$ and then, applying \eqref{eq:BallsIntegralEstimate} in the inequality given by Theorem \ref{thm:UpperBoundVolume}
we obtain
$$
\vol_k(K\cap H)\ls2^\frac{m_0-k}{2}\vol_k(B_\infty^k)\prod_{j\in J}c_j^\frac{\tilde{c}_j}{2}=2^\frac{m_0-k}{2}\vol_k(B_\infty^k)\prod_{j\in J}c_j^\frac{c_j\Vert P_Hv_j\Vert_2^2}{2}.
$$
Let us prove the second inequality. Let
$$
M=\left\{(\tilde{x}_j, x_j)_{j\in J}\,:\,\frac{1}{2}\leq\tilde{x}_j\leq x_j\leq 1,\,\forall j\in J,\,\sum_{j\in J}\tilde{x}_j=k,\,\sum_{j\in J}x_j\leq n\right\}
$$
and notice that $(\tilde{c}_j,c_j)_{j\in J}\in M$. Thus $M\neq\emptyset$. Therefore, identifying $J$ with $\{1,\dots,m_0\}$ we have, by Proposition \ref{prop:exact-max},
$$
\sum_{j=1}^{m_0}\tilde{c}_j\log c_j\ls (m_0-k)\log\left(\frac{m_0+n-2k}{2(m_0-k)}\right).
$$
 Therefore,
 $$
 \prod_{j\in J}c_j^\frac{c_j\Vert P_Hv_j\Vert_2^2}{2}\ls \frac{1}{2^\frac{m_0-k}{2}}\left(\frac{m_0+n-2k}{m_0-k}\right)^{\frac{m_0-k}{2}},
 $$
which gives the second inequality.\qedhere
\end{proof}

\begin{remark}\label{rem:comp-Ball}
It is noteworthy that the upper bound in Corollary \ref{cor:UpperBoundVolumeBallsIntegralEstimate} is better than the upper bound $2^{\frac{n-k}{2}}\vol_k(B_\infty^k)$ given by the right-hand side in \eqref{eq:Ball-2} if $m_0\ls n$ and worse if $m_0>n$. Indeed, calling $\alpha=\frac{n-k}{m_0-k}\gr1$ if and only if $m_0\ls n$, and using the strict convexity of the function $2^\alpha$ in $[0,\infty)$, we have that $\alpha+1\ls2^\alpha$ for every $\alpha\gr1$ and $\alpha+1>2^\alpha$ if $0<\alpha<1$.
\end{remark}

Let us now show that, if $\tilde{c}_j\geq \frac{1}{2}$ for every $j\in J$, then the tighter estimate provided by Corollary \ref{cor:UpperBoundVolumeBallsIntegralEstimate} is better than the estimate in Theorem \ref{thm:SectionsSymmetricJohn}. Thus, in this case, even after applying Ball's integral estimate \eqref{eq:BallsIntegralEstimate}, the estimate obtained by first applying Theorem \ref{thm:parseval} and then Brascamp-Lieb inequality is better than the one we obtain if we apply Brascamp-Lieb directly to the indicator functions of the intervals $[-\sqrt{\tilde{c}_j}t_j,\sqrt{\tilde{c}_j}t_j]$. This is the content of the following proposition:

\begin{proposition}\label{prop:ComparisonEstimates}
Let $K\subseteq\R^n$ be a centrally symmetric convex body in John's position and let  $(c_j,v_j)_{j=1}^m\subseteq (0,1]\times(\partial K\cap S^{n-1})$  be its associated decomposition of the identity. Let $H\in G_{n,k}$ be a $k$-dimensional subspace and let $J=\{1\leq j\leq m\,:\,P_Hv_j\neq0\}$. Assume that for every $j\in J$, $\tilde{c}_j=c_j\Vert P_Hv_j\Vert_2^2\geq\frac{1}{2}$. Then,
\begin{eqnarray*}
\vol_k(K\cap H)&\leq&2^\frac{m_0-k}{2}\vol_k(B_\infty^k)\prod_{j\in J}c_j^\frac{c_j\Vert P_Hv_j\Vert_2^2}{2}\leq \prod_{j\in J}\left(\frac{1}{\Vert P_H v_j\Vert_2}\right)^{c_j\Vert P_Hv_j\Vert_2^2}\vol_k(B_\infty^k)\cr
&\leq&\left(\frac{n}{k}\right)^\frac{k}{2}\vol_k(B_\infty^k).
\end{eqnarray*}
\end{proposition}
\begin{proof}
The first inequality is included in Corollary \ref{cor:UpperBoundVolumeBallsIntegralEstimate} and the last one in Theorem \ref{thm:SectionsSymmetricJohn}. Therefore, we want to prove the middle one. Let us first point out that, if $\tilde{c}_j\geq\frac{1}{2}$ for every $j\in J$, then necessarily
$$
k=\sum_{j\in J}\tilde{c}_j\geq\sum_{j\in J}\frac{1}{2}=\frac{m_0}{2},
$$
so $m_0\leq 2k$. Since for every $j\in J$, $\frac{1}{\Vert P_Hv_j\Vert_2}=\left(\frac{c_j}{\tilde{c}_j}\right)^\frac{1}{2}$, we want to prove that if $\tilde{c}_j\geq\frac{1}{2}$ for every $j\in J$, then
$$
2^\frac{m_0-k}{2}\prod_{j\in J}c_j^\frac{\tilde{c}_j}{2}\leq\prod_{j\in J}\left(\frac{c_j}{\tilde{c}_j}\right)^\frac{\tilde{c}_j}{2},
$$
which is equivalent to
$$
2^\frac{m_0-k}{2}\leq\prod_{j\in J}\left(\frac{1}{\tilde{c}_j}\right)^\frac{\tilde{c}_j}{2},
$$
which is also equivalent to
$$
\prod_{j\in J}\left(\tilde c_j\right)^{\tilde c_j}\leq \left(\frac{1}{2}\right)^{m_0-k}
$$
and, therefore, to
\begin{equation}\label{eq:ComparisonBounds}
	\sum_{j\in J}\tilde c_j\log\tilde c_j\leq (m_0-k)\log\frac{1}{2}.
\end{equation}
Let us consider the set
$$
M=\left\{(x_j)_{j\in J}\,:\,\frac{1}{2}\leq x_j\leq1\,,\,\sum_{j\in J}x_j=k\right\},
$$ which is convex and invariant under permutations of coordinates, and let $f:M\to\R$ be the function given by
$$
f(x)=\sum_{j\in J}x_j\log x_j.
$$
Notice that $f$ is convex function and that for every $(x_j)_{j\in J}\in M$ we have that
$$
(x_j)_{j\in J}\prec\left(1,\dots, 1,\frac{1}{2},\dots,\frac{1}{2}\right)=(y_j)_{j\in J},
$$
where there are $2k-m_0\geq0$ coordinates equal to 1 and $2m_0-2k$ coordinates equal to $\frac{1}{2}$. That is, identifying $J$ with $\R^{m_0}$.
\begin{itemize}
	\item $\displaystyle{\sum_{j=1}^{m_0}x_j=k=\sum_{j=1}^{m_0}y_j}$ and
	\item For every $1\leq k\leq m_0$, we have $\displaystyle{\sum_{j=1}^kx_j^*\ls\sum_{j=1}^ky_j^*}$, where $(z_j^\ast)_{j=1}^{m_0}$ denotes the non-increasing rearrangement of $(z_j)_{j=1}^{m_0}$.
\end{itemize}
Then, by Karamata's inequality we have that for every $x\in K$
$$
f(x)\leq f(y)=(m_0-k)\log\frac{1}{2},
$$
which implies \eqref{eq:ComparisonBounds}, since $(\tilde{c}_j)_{j\in J}\in M$.
\end{proof}

\begin{remark}
Let us point out that if $K$ is the convex body constructed in Theorem \ref{thm:counterexample} and $H=\textrm{span}\{e_1,\dots, e_k\}$, and $\tilde{c}_j=\frac{1}{2}$ for every $j\in J$, and $m_0=2k$. In that case all the inequalities are equalities.
\end{remark}

In order to provide an estimate in the case that there exists $j\in J$ such that $\tilde{c}_j<\frac{1}{2}$, we start proving the following lemma:

\begin{lemma}\label{lem:SectionsBoxes}
Let $(x_j)_{j=1}^m\subseteq\R^m$ be an ortonormal basis of $\R^m$, $(l_j)_{j=1}^m\subseteq(0,\infty)$, and let $E\in G_{m,k}$. Let $1\ls j_0\ls m$ such that $x_{j_0}\not\in E$ and define
$$
B=\sum_{j=1}^m[-l_j,l_j]x_j,\quad\textrm{and}\quad B^\prime=\sum_{{j=1}\atop {j\neq j_0}}^m[-l_j,l_j]x_j\subseteq x_{j_0}^\perp.
$$
Then, there exists $E^\prime\in G_{x_{j_0}^\perp,k}$, a $k$-dimensional linear subspace in $x_0^\perp$ such that
$$
\vol_k(B\cap E)\ls\frac{1}{\Vert P_{E^\perp}x_{j_0}\Vert_2}\vol_k(B^\prime\cap E^\prime)=\frac{1}{\sqrt{1-\Vert P_{E}x_{j_0}\Vert_2^2}}\vol_k(B^\prime\cap E^\prime).
$$
\end{lemma}

\begin{proof}
Let us assume, without loss of generality, that $j_0=m$ and let us call
$$
\xi:=\frac{P_{E^\perp}x_{m}}{\Vert P_{E^\perp} x_m\Vert_2}\in E^\perp,\quad \gamma:=\langle x_m,\xi\rangle=\Vert P_{E^\perp} x_m\Vert_2=\sqrt{1-\Vert P_{E}x_{m}\Vert_2^2}\in(0,1],
$$
and
$$
\tilde{B}:=\left\{\sum_{j=1}^{m-1} a_jx_j+\lambda x_m\,:\,|a_j|\ls l_j,\,\lambda\in\R\right\},
$$
the cylinder with axis $x_m$ and base $B^\prime$. Notice that $B\subseteq \tilde{B}$ and then
$$
\vol_k(B\cap E)\ls\vol_k(\tilde{B}\cap E).
$$

Let us define the linear map $F:x_m^\perp\to\xi^\perp$ the linear map given by
$$
F(y):=y-\frac{\langle y,\xi\rangle}{\langle x_m,\xi\rangle}x_m,\quad\forall y\in x_m^\perp.
$$

$F$ is well defined since $\langle x_m,\xi\rangle=\gamma\neq0$ and for every $y\in x_m^\perp$
$$
\langle F(y),\xi\rangle=\langle y,\xi\rangle-\frac{\langle y,\xi\rangle}{\langle x_m,\xi\rangle}\langle x_m,\xi\rangle=0,
$$
so $F(y)\in\xi^\perp$. Moreover, $F$ is an isomorphism as $\langle x_m,\xi\rangle=\gamma\neq0$.

Notice that for every $y\in B^\prime\subseteq x_m^\perp$, $F(y)\in \tilde{B}$ and by construction, $F(y)\in\xi^\perp$. Thus, $F(B^\prime)\subseteq\tilde{B}\cap\xi^\perp$. Conversely, every $z\in\tilde{B}\cap\xi^\perp$ can be uniquely written as $z=y+\lambda x_m$ for some $y\in B^\prime\subseteq x_m^\perp$ and some $\lambda\in\R$. Since $z\in\xi^\prime$, necessarily $\lambda=-\frac{\langle y,\xi\rangle}{\langle x_m,\xi\rangle}$ and then $z=F(y)$. Thus, $\tilde{B}\cap\xi^\perp\subseteq F(B^\prime)$ and we have
$$
F(B^\prime)=\tilde{B}\cap\xi^\perp.
$$
Let us now consider the ($m-2$)-dimensional subspace of $x_m^\perp$
$$
W:=x_m^\perp\cap\xi^\perp.
$$
Notice that $F(y)=y$ for every $y\in W$. Moreover, if we call
$$
\eta=\xi-\gamma x_m\in\xi_m^\perp\quad\textrm{and}\quad\zeta:=x_{m}-\gamma \xi\in \xi^\perp,
$$
we have that $\langle\eta, y\rangle =0$ and $\langle \zeta, y\rangle=0$ for every $y\in W$. Therefore, $\eta$ spans the orthogonal complement of $W$ in $x_m^\perp$ and $\zeta$ spans the orthogonal complement of $W$ in $\xi^\perp$. Besides $\Vert \eta\Vert_2=\Vert\zeta\Vert_2=\sqrt{1-\gamma^2}$ and $F(\eta)=-\frac{1}{\gamma}\zeta$. Hence, the singular values of $F$ are
$$
\sigma(F)=\left\{1,\dots,1,\frac{1}{\gamma}\right\}
$$
and for every $k$-dimensional subspace $E^\prime\in G_{x_m^\perp,k}$, we have that $|\textrm{det}F_{|E^\prime}|$ is bounded above by the product of the $k$ largest singular values, which is $\frac{1}{\gamma}$. Thus, for every measurable set $A\in E^\prime$,
$$
\vol_k(F(A))\ls\frac{1}{\gamma}\vol_k(A).
$$
Taking into account that $\xi\in E^\perp$, we have that $E\subseteq\xi^\perp$, and taking $E^\prime=F^{-1}(E)\in G_{x_m^\perp,k}$, we have
$$
\tilde{B}\cap E=\tilde{B}\cap (E\cap\xi^\prime)=(\tilde{B}\cap\xi^\prime)\cap E=F(B^\prime)\cap F(E^\prime)=F(B^\prime\cap E^\prime).
$$
Therefore,
\[
\vol_k(B\cap E)\ls\vol_k(\tilde{B}\cap E)=\vol_k(F(B^\prime\cap E^\prime))\ls\frac{1}{\gamma}\vol_k(B^\prime\cap E^\prime).\qedhere
\]
\end{proof}

As a consequence, we have the following

\begin{lemma}\label{lem:SectionsBoxesIterated}
Let $(x_j)_{j=1}^m\subseteq\R^m$ be an ortonormal basis of $\R^m$, $(l_j)_{j=1}^m\subseteq(0,\infty)$, and let $E\in G_{m,k}$. There exists $0\ls q\ls m-k$ such that there exists $J_q\subseteq\{1,\dots, m\}$ with cardinality $m-q$, $(\theta_j)_{j=1}^q\subseteq\left[0,\frac{1}{2}\right)$ and a $k$-dimensional subspace $E_q\subseteq\{x_j\,:\,j\not\in J_q\}^\perp$ which satisfies that $\Vert P_{E_q}x_j\Vert_2^2\geq\frac{1}{2}$ for every $j\in J_q$ and, calling
$$
B=\sum_{j=1}^m[-l_j,l_j]x_j,\quad\textrm{and}\quad B_q=\sum_{j\in J_q}^m[-l_j,l_j]x_j\subseteq \{x_j\,:\,j\not\in J_q\}^\perp,
$$
we have
$$
\vol_k(B\cap E)\ls\vol_{k}(B_q\cap E_q)\prod_{j=1}^q\frac{1}{\sqrt{1-\theta_j}}<2^{q\overline{\theta}}\vol_{k}(B_q\cap E_q),
$$
where
$$
\overline{\theta}=\frac{1}{q}\sum_{j=1}^q\theta_j\in\left[0,\frac{1}{2}\right).
$$
Moreover, if $m>2k$, then $q\gr m-2k$ and for every $1\ls j\ls m-2k$ we can assume that $\theta_j\ls \frac{k}{m-j-1}<\frac{1}{2}$ and then
$$
\vol_k(B\cap E)\ls\vol_{k}(B_q\cap E_q)\prod_{j=1}^q\frac{1}{\sqrt{1-\theta_j}}<2^\frac{q-(m-2k)}{2}\sqrt{\frac{{{m}\choose{k}}}{{{2k}\choose{k}}}}\vol_{k}(B_q\cap E_q).
$$
\end{lemma}

\begin{proof}
If $\Vert P_E x_j\Vert_2^2\geq\frac{1}{2}$ for every $1\ls j\ls m$, we set $q=0$, $J_1=\emptyset$, and $E_q=E$. Otherwise, let $j_1\in J$ be such that $\theta_1:=\Vert P_E x_{j_1}\Vert_2^2<\frac{1}{2}$ and let us call $J_1=\{1,\dots, m\}\setminus\{j_1\}$. By Lemma \ref{lem:SectionsBoxes}, there exists a $k$-dimensional subspace $E_1\subseteq x_{j_1}^\perp$ such that
$$
\vol_k(B\cap E)\ls\frac{1}{\sqrt{1-\theta_1}}\vol_k(B_1\cap E_1),
$$
with
$$
B_1=\sum_{j\in J_1}[-l_j,l_j]x_j\subseteq x_{j_1}^\perp.
$$
If $\Vert P_{E_1} x_j\Vert_2^2\geq\frac{1}{2}$ for every $j\in J_1$, we set $q=1$ with $J_1$ and $E_1$ as defined. Otherwise, there exists $j_2\in J_1$ such that $\theta_2:=\Vert P_{E_1} x_{j_2}\Vert_2^2<\frac{1}{2}$ and call $J_2=J_1\setminus\{j_2\}$. By Lemma \ref{lem:SectionsBoxes}, there exists a $k$-dimensional subspace $E_2\subseteq \{x_{j_1},x_{j_2}\}^\perp$ such that
$$
\vol_k(B_1\cap E_1)\ls\frac{1}{\sqrt{1-\theta_2}}\vol_k(B_2\cap E_2),
$$
with
$$
B_2=\sum_{j\in J_2}^m[-l_j,l_j]x_j\subseteq \{x_{j_1}, x_{j_2}\}^\perp.
$$
We iterate the process and notice that we stop after at most $m-k$ steps since, in that case $E_{m-k}$ is a $k$ dimensional subspace of $\{x_j\,:\,j\not\in J_{m-k}\}^\perp$, which is $k$-dimensional and, therefore, $E_{m-k}=\{x_j\,:\,j\not\in J_{m-k}\}^\perp$ and $\Vert P_{E_{m-k}}x_j\Vert_2^2=1\geq\frac{1}{2}$ for every $j\in J_{m-k}$.

Notice that the function $h(x):=-\frac12\ln(1-x)$ is strictly convex on $\left[0,\frac{1}{2}\right)$ with $h(0)=0$ and $h(1/2)=\frac12\ln 2$.
Thus $h(x)< (\ln 2)\,x$ for every $x\in\left[0,\frac{1}{2}\right)$. Therefore, $(1-x)^{-1/2}<2^x$ for every $x\in\left[0,\frac{1}{2}\right)$.

Finally, assume that $m>2k$ and notice that for every $j$ we have that $1\ls j\ls m-2k$  we have that $\sharp J_{j-1}=m-j+1>2k$ and
$$
\sum_{i\in J_{j-1}}\Vert P_{E_{j-1}} x_i\Vert_2^2=k.
$$
Therefore, if $\sharp J_{j-1}=m-j+1>2k$ we have that
$$
\min_{i\in J_{j-1}}\Vert P_{E_{j-1}} x_i\Vert_2^2\ls\frac{k}{m-j+1}<\frac{1}{2},
$$
and we can take $\theta_j=\Vert P_{E_{j-1}} x_{i_0}\Vert_2^2\ls\frac{k}{m-j+1}$, where $i_0$ is the index for which the minimum is attained.

Making this choice, we have that for every $1\ls j\ls m-2k$
$$
\frac{1}{\sqrt{1-\theta_j}}\ls\frac{1}{\sqrt{1-\frac{k}{m-j+1}}}
$$
and then
$$
\prod_{j=1}^{m-2k}\frac{1}{\sqrt{1-\theta_j}}\ls\prod_{j=1}^{m-2k}\frac{1}{\sqrt{1-\frac{k}{m-j+1}}}=\prod_{l=2k+1}^{m}\frac{1}{\sqrt{1-\frac{k}{l}}}=\sqrt{\prod_{l=2k+1}^{m}\frac{l}{l-k}}.
$$
Since
$$
\prod_{l=2k+1}^{m}\frac{l}{l-k}=\frac{m!k!}{(2k)!(m-k)!}=\frac{{{m}\choose{k}}}{{{2k}\choose{k}}},
$$
we have that
$$
\prod_{j=1}^{m-2k}\frac{1}{\sqrt{1-\theta_j}}\ls\sqrt{\frac{{{m}\choose{k}}}{{{2k}\choose{k}}}}.
$$
Using the upper bound
$$
\frac{1}{\sqrt{1-\theta_j}}\ls 2^{\theta_j}\ls\sqrt{2}
$$
for every $m-2k+1\ls j\ls q$ we finish the proof.
\end{proof}

As a consequence, we have the following

\begin{theorem}
Let $K\subseteq\R^n$ be a centrally symmetric convex body in John's position and let  $(c_j,v_j)_{j=1}^m\subseteq (0,1]\times(\partial K\cap S^{n-1})$  be its associated decomposition of the identity. Let $H\in G_{n,k}$ be a $k$-dimensional subspace and let $J=\{1\leq j\leq m\,:\,P_Hv_j\neq0\}$, and $m_0$ its cardinality. Then, there exists $0\ls q\ls m-k$ such that there exists $J_q\subseteq J$ with cardinality $m_0-q$, $(\theta_j)_{j=1}^q\subseteq\left[0,\frac{1}{2}\right)$ and $(d_j)_{j\in J_1}\in\left[\frac{1}{2},1\right)$ such that $\displaystyle{\sum_{j\in J_q}d_j\ls q}$, satisfying
$$
\vol_k(K\cap H)\ls2^\frac{m_0-q-k}{2}\vol_k(B_\infty^k)\left(\prod_{j=1}^q\frac{1}{\sqrt{1-\theta_j}}\right)\left(\prod_{j\in J_q}c_j^\frac{d_j}{2}\right)<2^\frac{m_0-k}{2}2^{-\frac{q}{2}(1-2\overline{\theta})}\vol_k(B_\infty^k)\prod_{j\in J_q}c_j^\frac{d_j}{2}.
$$
where
$$
\overline{\theta}=\frac{1}{q}\sum_{j=1}^q\theta_j\in\left[0,\frac{1}{2}\right).
$$
Moreover, if $m_0>2k$, then $q\gr m-2k$ and
$$
\vol_k(K\cap H)\ls2^\frac{k}{2}\vol_k(B_\infty^k)\sqrt{\frac{{{m_0}\choose{k}}}{{{2k}\choose{k}}}}\left(\prod_{j\in J_q}c_j^\frac{d_j}{2}\right).
$$
\end{theorem}

\begin{proof}
Recall that, using the notation explained in Section \ref{sec:GeneralSetting}, there exists an ortonormal basis $(x_j)_{j=1}^{m_0}\subseteq\R^{m_0}$ such that
$$
P_H x_j=\sqrt{\tilde{c_j}}u_j
$$
for every $j\in J$. Thus,
\begin{eqnarray*}
L\cap H&=&\{x\in H\,:\,|\langle x,u_j\rangle|\ls t_j, \forall j\in J\}=\{x\in H\,:\,|\langle x, x_j\rangle|\ls\sqrt{\tilde{c_j}}t_j, \forall j\in J\}\cr
&=&\{x\in H\,:\,|\langle x, x_j\rangle|\ls\sqrt{c_j}, \forall j\in J\}= B\cap H,
\end{eqnarray*}
where $B=\prod_{j=1}^{m_0}[-\sqrt{c_j},\sqrt{c_j}]x_j$. By Lemma \ref{lem:SectionsBoxesIterated}, there exists $0\ls q\ls m_0-k$ such that there exists $J_q\subseteq J$ with cardinality $m_0-q$, $(\theta_j)_{j=1}^q\subseteq\left[0,\frac{1}{2}\right)$ and a $k$-dimensional subspace $E_q\subseteq\{x_j\,:\,j\in J_q\}^\perp$ which satisfies that $\Vert P_{E_q}x_j\Vert_2^2\geq\frac{1}{2}$ for every $j\in J_q$ and
$$
\vol_k(B\cap E)\ls\vol_{k}(B_q\cap E_q)\prod_{j=1}^q\frac{1}{\sqrt{1-\theta_j}},
$$
where
$$
B_q\cap E_q=\{x\in E_q\,:\,|\langle x,x_j\rangle|\ls \sqrt{c_j}, \forall j\in J_q\}.
$$
Therefore, proceeding like in the proof of Theorem \ref{thm:UpperBoundVolume}, and calling $d_j=\Vert P_{E_q} x_j\Vert_2^2\in\left[\frac{1}{2},1\right]$ and $z_j=\frac{P_{E_q}x_j}{\sqrt{d_j}}$ for every $j\in J_q$,
\begin{eqnarray*}
\vol_{k}(B_q\cap E_q)&=&\int_{E_q}\prod_{j\in J_q}\mathds{1}_{[-\sqrt{c_j},\sqrt{c_j}]}(\langle x, x_j\rangle)dx=\int_{E_q}\prod_{j\in J_q}\mathds{1}_{[-\sqrt{c_j},\sqrt{c_j}]}(\sqrt{d_j}\langle x, z_j\rangle)dx\cr
&\ls& \vol_k(B_\infty^k)\prod_{j\in J_q}\left(\frac{I_{\frac{1}{1-d_j}}}{\pi\sqrt{1-d_j}}\right)^{1-d_j}c_j^\frac{d_j}{2}\ls 2^\frac{m_0-q-k}{2}\vol_k(B_\infty^k)\prod_{j\in J_q}c_j^\frac{d_j}{2},
\end{eqnarray*}
Consequently,
\begin{eqnarray*}
\vol_k(K\cap H)&\ls&\vol_k(L\cap H)\ls2^\frac{m_0-q-k}{2}\vol_k(B_\infty^k)\left(\prod_{j=1}^q\frac{1}{\sqrt{1-\theta_j}}\right)\left(\prod_{j\in J_q}c_j^\frac{d_j}{2}\right)\cr
&<&2^\frac{m_0-k}{2}2^{-\frac{q}{2}(1-2\overline{\theta})}\vol_k(B_\infty^k)\prod_{j\in J_q}c_j^\frac{d_j}{2}.\hspace{6cm}\qedhere
\end{eqnarray*}
\end{proof}

\subsection{Sections of non-necessarily symmetric convex bodies}\label{subsec:Non-symmetric}

In this section we will provide upper estimates for sections of non-necessarily symmetric convex bodies in John's position. Given $K\subseteq\R^n$ a convex body in John's positions, by \eqref{eq:John-dec}, there exists an associated decomposition of the identity $(c_j,v_j)_{j=1}^m\subseteq (0,1]\times (\partial K\cap S^{n-1})$ such that $\displaystyle{\sum_{j=1}^m c_jv_j=0}$. We will follow the notation introduced in Section \ref{sec:GeneralSetting} and, since $K\subseteq C$, provide upper estimates for $C\cap H$, with $H\in G_{n,k}$.

The main result in this section is the following estimate of the volume of sections of a convex body in John's position by a $k$-dimensional linear subspace:

\begin{theorem}\label{thm:nonsym-sec-john}
		Let $K$ be a convex body in $\mathbb{R}^n$ in John's position and $H\in G_{n,k}$. Assume that $\tilde{\delta}_j\gr \frac{1}{2}$ for every $1\ls j\ls m$, then
		\[
		\vol_{k}(K\cap H) \ls \vol_k(S_k)2^\frac{k+1-m}{2}\frac{n^\frac{k}{2}(n+1)^\frac{k+1}{2}}{k^\frac{k}{2}(k+1)^\frac{k+1}{2}}\left(\frac{2k-n+1}{m-2(n-k)}\right)^{-\frac{2k-n+1}{2}}.
		\]
	\end{theorem}

First of all, we will prove the following technical lemma:
\begin{lemma}\label{lem:FourierTransformExponentialOneSide}
Let $\alpha>0$ and $f_{\alpha}:\R\to\R$ the function given by $f_{\alpha}(x)=e^{-\alpha x}\mathds{1}_{[0,\infty)(x)}$.Then $\widehat{f_{\alpha}}:\R\to\R$ is given by
		$$
		\widehat{f_{\alpha}}(s)=\frac{1}{\alpha-is},\quad\forall s\in\R.
		$$
	\end{lemma}
	
	\begin{proof}
		For every $s\in\R$ we have that
		\begin{eqnarray*}
			\widehat{f_{\alpha}}(s)&=&\int_{\R}f_{\alpha}(x)e^{ixs}dx=\int_0^\infty e^{-\alpha x}\left(\cos(xs)+i\sin(xs)\right)dx\cr
&=&\int_0^\infty e^{-\alpha x}\cos(xs)dx+i\int_0^\infty e^{-\alpha x}\sin(xs)dx.\cr
		\end{eqnarray*}
		Integrating by parts twice, we have
		\begin{eqnarray*}
			\int_0^\infty e^{-\alpha x}\cos(xs)dx&=&\frac{1}{\alpha}-\frac{s}{\alpha}\int_0^\infty e^{-\alpha x}\sin(xs)dx\cr
			&=&\frac{1}{\alpha}-\frac{s^2}{\alpha^2}\int_0^\infty e^{-\alpha x}\cos(xs)dx.
		\end{eqnarray*}
		Therefore,
		$$
		\left(1+\frac{s^2}{\alpha^2}\right)\int_0^\infty e^{-\alpha x}\cos(xs)dx=\frac{1}{\alpha}\Leftrightarrow\int_0^\infty e^{-\alpha x}\cos(xs)dx=\frac{1}{\alpha}\frac{1}{1+\frac{s^2}{\alpha^2}}=\frac{\alpha}{\alpha^2+s^2}.
		$$
In the same way, integrating by parts,
$$
\int_0^\infty e^{-\alpha x}\sin(xs)dx=\frac{s}{\alpha}\int_0^\infty e^{-\alpha x}\cos(xs)dx=\frac{s}{\alpha^2+s^2},
$$
		and then
		\[
		\widehat{f_{\alpha}}(s)=\frac{\alpha+is}{\alpha^2+s^2}=\frac{1}{\alpha-is}.\qedhere
		\]
\end{proof}
Let us first provide the following intermediate estimate for sections of any dimension.

	\begin{lemma}\label{lem:nonsym.sec.fourier}
		Let $K\subseteq\R^n$ be convex body in John's position and let $(c_j,v_j)_{j=1}^m\subseteq (0,1]\times(\partial K\cap S^{n-1})$ be its associated decomposition of the identity. Let $H\in G_{n,k}$ and let $(\delta_j)_{j=1}^m$ and $(\tilde{\delta}_j)_{j=1}^m$ be defined as in Section \ref{sec:GeneralSetting}. Assume that $\tilde{\delta}_j\gr \frac{1}{2}$ for every $1\ls j\ls m$. Then,
		\[
		\vol_k(K\cap H) \ls \vol_k(S_k)2^\frac{k+1-m}{2}\frac{n^\frac{k}{2}(n+1)^\frac{k+1}{2}}{k^\frac{k}{2}(k+1)^\frac{k+1}{2}}\prod_{j=1}^m\delta_j^{-\frac{\tilde{\delta}_j}{2}}.
		\]
	\end{lemma}
	\begin{proof}
		Our starting point is the formula, established in \cite[Section 4]{AB}, where we follow the general setting established in Section \ref{sec:GeneralSetting}:
		\[
		\frac{k^\frac{k}{2}(k+1)^\frac{k+1}{2}}{n^\frac{k}{2}(n+1)^\frac{k+1}{2}}\frac{\vol_k(C\cap H)}{\vol_k(S_k)} = \int_{L\cap F} e^{-\sum_{j=1}^m\sqrt{\delta_j\tilde{\delta}_j}\langle y,u_j\rangle}dy,
		\]
Since, $(\tilde{\delta}_j, u_j)_{j=1}^m$ provide a decomposition of the identity in $F$ and
\begin{eqnarray*}
\int_{L\cap F} e^{-\sum_{j=1}^m\sqrt{\delta_j\tilde{\delta}_j}\langle y,u_j\rangle}dy&=&\int_F \prod_{j=1}^me^{-\sqrt{\delta_j\tilde{\delta}_j}\langle y,u_j\rangle}\mathds{1}_{[0,\infty)}(\langle y,\tilde{v}_j\rangle)dy\cr
&=&\int_F \prod_{j=1}^me^{-\sqrt{\delta_j\tilde{\delta}_j}\langle y,u_j\rangle}\mathds{1}_{[0,\infty)}(\langle y,P_F(\tilde{v}_j\rangle))dy\cr
&=&\int_F \prod_{j=1}^me^{-\sqrt{\delta_j\tilde{\delta}_j}\langle y,u_j\rangle}\mathds{1}_{[0,\infty)}\left(\sqrt{\tilde{\delta}_j}\langle y,u_j\rangle\right)dy,\cr
\end{eqnarray*}
we have, by Corollary \ref{cor:UpperBoundParsevalAndBrascampLieb}, that
\begin{eqnarray*}
\int_{L\cap F} e^{-\sum_{j=1}^m\sqrt{\delta_j\tilde{\delta}_j}\langle y,u_j\rangle}dy&\leq&\frac{1}{(2\pi)^{m-(k+1)}}\prod_{j=1}^m\left(\int_\R\left(\frac{1}{\delta_j+(1-\tilde{\delta}_j)t^2}\right)^\frac{1}{2(1-\tilde{\delta}_j)}dt\right)^{1-\tilde{\delta}_j}\cr
&=&\prod_{j=1}^m\left(\int_\R\left(\frac{1}{\delta_j+4\pi^2(1-\tilde{\delta}_j)t^2}\right)^\frac{1}{2(1-\tilde{\delta}_j)}dt\right)^{1-\tilde{\delta}_j}\cr
&=&\prod_{j=1}^m\delta_j^{-\frac{1}{2}}\left(\int_\R\left(\frac{1}{1+4\pi^2(1-\tilde{\delta}_j)\delta_j^{-1}t^2}\right)^\frac{1}{2(1-\tilde{\delta}_j)}dt\right)^{1-\tilde{\delta}_j}.\cr
\end{eqnarray*}
Taking into account that for every $\alpha>0$ and $0<\beta\ls 1/2$, by Bernoulli's inequality,
		\[
		1+\frac{\alpha}{2}\ls(1+\alpha\beta)^\frac{1}{2\beta}\Leftrightarrow\left(\frac{1}{1+\alpha\beta}\right)^\frac{1}{2\beta}\ls \frac{1}{1+\frac{\alpha}{2}},
		\]
		 we have that if $1-\tilde{\delta}_j\ls \frac{1}{2}$ for every $1\ls j\ls m$, then
\begin{eqnarray*}
\prod_{j=1}^m\delta_j^{-\frac{1}{2}}\left(\int_\R\left(\frac{1}{1+4\pi^2(1-\tilde{\delta}_j)\delta_j^{-1}t^2}\right)^\frac{1}{2(1-\tilde{\delta}_j)}dt\right)^{1-\tilde{\delta}_j}&\leq&\prod_{j=1}^m\delta_j^{-\frac{1}{2}}\left(\int_{\R}\frac{1}{1+2\pi^2\delta_j^{-1}t^2}dt\right)^{1-\tilde{\delta}_j}\cr
&=&\prod_{j=1}^m\delta_j^{-\frac{1}{2}}\left(\frac{\delta_j}{2}\right)^{\frac{1-\tilde{\delta}_j}{2}}\cr
&=&2^{\frac{k+1-m}{2}}\prod_{j=1}^m\delta_j^{-\frac{\tilde{\delta}_j}{2}},
\end{eqnarray*}
which finishes the proof.
	\end{proof}

Let us now prove  Theorem \ref{thm:nonsym-sec-john}.
	
	\begin{proof}[Proof of Theorem \ref{thm:nonsym-sec-john}]

Taking into account the previous lemma, our goal is to establish the bound $\prod_{j=1}^m\delta_j^{-\frac{\tilde{\delta}_j}{2}}\ls\left(\frac{2k-n+1}{m-2(n-k)}\right)^{-\frac{2k-n+1}{2}}$.

Let us find the minimum of the function
$$
f(\delta,\tilde{\delta})=\sum_{j=1}^m\tilde{\delta}_j\log\delta_j
$$
on the compact set
$$
M=\left\{(\delta_j,\tilde{\delta}_j)_{j=1}^m\,:\,\frac{1}{2}\ls\tilde{\delta_j}\ls\delta_j\ls 1\,:\,\sum_{j=1}^m\tilde{\delta_j}=k+1,\,\sum_{j=1}^m\delta_j=n+1\right\}.
$$
First of all, notice that for every $(\delta_j,\tilde{\delta}_j)_{j=1}^m$ we have
$$
n-k=\sum_{j=1}^m(\delta_j-\tilde{\delta}_j)\ls\frac{m}{2}
$$
and then $m\gr 2(n-k)$. Notice also that if $(\delta_j,\tilde{\delta}_j)_{j=1}^m\in M$ is a minimizer, 
necessarily there exists at most one coordinate of $\delta$ which is different from the corresponding coordinate of $\tilde{\delta}$ or from one $1$. Thus, we can assume without loss of generality that there exists $0\ls m_1\ls m$ such that
\begin{itemize}
\item $\delta_j=1$ for every $1\ls j\ls m_1$, and
\item $\delta_j=\tilde{\delta}_j$ for every $m_1+1\ls j\ls m$.
\end{itemize}
Assume that this is not the case, then, there exist $1\ls j_1<j_2\ls m$ such that $\tilde{\delta}_{j_1}<\delta_{j_1}<1$, and $\tilde{\delta}_{j_2}<\delta_{j_2}<1$. Let us call $\varepsilon_0=\min\{\delta_{j_1}-\tilde{\delta}_{j_1},\delta_{j_2}-\tilde{\delta}_{j_2}, 1-\delta_{j_1}, 1-\delta_{j_2}\}>0$, and consider the function $h:[-\varepsilon_0,\varepsilon_0]\to\R$ given by
$$
h(\varepsilon)=\tilde{\delta}_{j_1}\log(\delta_{j_1}-\varepsilon)+\tilde{\delta}_{j_2}\log(\delta_{j_2}+\varepsilon).
$$
Since
$$
h^{\prime\prime}(\varepsilon)=-\frac{\tilde{\delta}_{j_1}}{(\delta_{j_1}-\varepsilon)^2}-\frac{\tilde{\delta}_{j_2}}{(\delta_{j_2}+\varepsilon)^2}<0
$$
we have that $h$ is strictly concave on $[-\varepsilon_0,\varepsilon_0]$ and then $h(0)>\min\{h(-\varepsilon_0),h(\varepsilon_0)\}$. Therefore, if this minimum is attained at $\epsilon\varepsilon_0$ with $\epsilon=\pm1$,  calling
$$
\delta_j^*=\begin{cases}\delta_j &\textrm{ if } j\neq j_1,j_2\cr
\delta_{j_1}-\epsilon\varepsilon_0&\textrm{ if } j= j_1\cr
\delta_{j_2}+\epsilon\varepsilon_0&\textrm{ if } j= j_2,\cr
\end{cases}
$$
we have that $(\delta^*_j,\tilde{\delta}_j)_{j=1}^m\in M$ and $f(\delta,\tilde{\delta})>f(\delta^*,\tilde{\delta})$, contradicting the minimality of $f$ at $(\delta_j,\tilde{\delta}_j)_{j=1}^m$. Therefore, if $(\delta_j,\tilde{\delta}_j)_{j=1}^m\in M$ is a minimizer of $f$ in $M$,  
there exists $0\ls m_1\ls m$ such that $\delta_j=1$ for every $1\ls j\ls m_1$, $\delta_j<1$ for every $m_1+1\ls j\ls m$, and $\tilde{\delta}_j=\delta_j$ for every $m_1+1\ls j\ls m$.

Moreover, since $\displaystyle{\sum_{j=1}^m\tilde{\delta}_j}=k+1$, and $\displaystyle{\sum_{j=1}^m\delta_j}=n+1$, we have that
$$
n-k=\sum_{j=1}^m\delta_j-\sum_{j=1}^m\tilde{\delta}_j=\sum_{j=1}^{m_1}(1-\tilde{\delta}_j)+(\delta_{m_1}-\tilde{\delta}_{m_1})<\frac{m_1+1}{2}.
$$
Therefore, $m_1+1>2(n-k)$ and then $m_1>2(n-k)-1$ and, since $m_1\in\mathbb{N}$ we have that $m_1\gr 2(n-k)$, and, if $(\delta_j,\tilde{\delta}_j)_{j=1}^m\in M$ is a minimizer of $f$ in $M$, then necessarily there exists $2(n-k)\ls m_1\ls m$ such that $\delta_j=1$ for every $1\ls j\ls m_1$, $\delta_j<1$ for every $m_1+1\ls j\ls m$, and $\tilde{\delta}_j=\delta_j$ for every $m_1+1\ls j\ls m$. We will also assume that $m_1<m$ since, if $m_1=m$, then $\delta_j=1$ for every $1\ls j\ls m$, $m=n+1$ and $f(\delta,\tilde{\delta})=0$.

Now, for such minimizer
\begin{eqnarray}\label{eq:Minimality}
f(\delta,\tilde{\delta})&=&\sum_{j=m_1+1}^m\tilde{\delta}_j\log\delta_j=\sum_{j=m_1+1}^m\delta_j\log\delta_j+(\delta_{m_1+1}-\tilde{\delta}_{m_1+1})\log\frac{1}{\delta_{m_1+1}}\gr\sum_{j=m_1+1}^m\delta_j\log\delta_j\cr
&=&f(\delta,\tilde{\delta}^*),
\end{eqnarray}
where
$$
\delta_j^{**}=\begin{cases}\frac{1}{2} &\textrm{ if } 1\ls j\ls 2(n-k)\cr
\delta_{j}&\textrm{ if } 2(n-k)+1\ls j\ls n.\cr
\end{cases}
$$
Since $m_1\geq 2(n-k)$, we have $\displaystyle{\sum_{j=2(n-k)+1}^m\delta_j=n+1-2(n-k)=2k+1-n}$. Therefore, we have  $\displaystyle{\sum_{j=1}^m\tilde{\delta}_j^{**}=n-k+(2k+1-n)=k+1}$ and $(\delta_j,\tilde{\delta}_j^{**})_{j=1}^m\in M$. Consequently, inequality \eqref{eq:Minimality} is an equality and we have that there exists $(\delta_j,\tilde{\delta}_j)_{j=1}^m$, a minimizer of $f$ in $M$ such that there exists $2(n-k)\ls m_1\ls m$ such that $\delta_j=1$ for every $1\ls j\ls m_1$, $\delta_j<1$ for every $m_1+1\ls j\ls m$, $\tilde{\delta}_j=\frac{1}{2}$ for every $1\ls j\ls 2(n-k)$ and $\tilde{\delta}_j=\delta_j$ for every $2(n-k)+1\ls j\ls m$.

Therefore, for such minimizer,
$$
f(\delta,\tilde{\delta})=\sum_{j=m_1+1}^m\delta_j\log\delta_j
$$
and $\sum_{j=m_1+1}^m\delta_{j}=n+1-m_1$. Notice that the function $h_1:\left[\frac{1}{2},1\right]$ given by $h(t)=t\log t$ is strictly convex on $\left[\frac{1}{2},1\right]$. Therefore,
\begin{eqnarray*}
f(\delta,\tilde{\delta})&=&\sum_{j=m_1+1}^m\delta_j\log\delta_j=\sum_{j=m_1+1}^m h(\delta_j)\gr (m-m_1)h\left(\frac{1}{m-m_1}\sum_{j=m_1+1}^m\delta_j\right)\cr
&=&(m-m_1)h\left(\frac{n+1-m_1}{m-m_1}\right)=(n+1-m_1)\log\left(\frac{n+1-m_1}{m-m_1}\right)=f(\delta^{m_1},\tilde{\delta}^{m_1}),
\end{eqnarray*}
where,
$$
\tilde{\delta}_j^{m_1}=\begin{cases}\frac{1}{2} &\textrm{ if } 1\ls j\ls 2(n-k)\cr
1&\textrm{ if } 2(n-k)+1\ls j\ls m_1.\cr
\frac{n+1-m_1}{m-m_1}&\textrm{ if } m_1+1\ls j\ls n.\cr
\end{cases}
$$
and
$$
\delta_j^{m_1}=\begin{cases} 1 &\textrm{ if } 1\ls j\ls m_1\cr
\frac{n+1-m_1}{m-m_1}&\textrm{ if } m_1+1\ls j\ls n.\cr
\end{cases}
$$
Since $(\delta^{m_1}_j,\tilde{\delta}^{m_1}_j)_{j=1}^m\in M$, we have that the minimum value of $f$ on $M$ is
$$
\min_{(\delta,\tilde{\delta})\in M}f(\delta,\tilde{\delta})=(n+1-m_1)\log\left(\frac{n+1-m_1}{m-m_1}\right)
$$
for some $2(n-k)\ls m_1\ls m$ and, for each $m_1$ this value is attained at a point of the form $(\delta^{m_1}_j,\tilde{\delta}^{m_1}_j)_{j=1}^m$. Since this function is increasing on $m_1$,
$$
\min_{(\delta,\tilde{\delta})\in M}f(\delta,\tilde{\delta})=(2k-n+1)\log\left(\frac{2k-n+1}{m-2(n-k)}\right)
$$
and it is attained at the point with exactly $2(n-k)$ coordinates of $(\delta_j,\tilde{\delta}_j)=\left(1,\frac{1}{2}\right)$ and $m-2(n-k)$ coordinates $(\delta_j,\tilde{\delta}_j)=\left(\frac{k+1}{m-2(n-k)},\frac{2k-n+1-m_1}{m-2(n-k)}\right)$.
\end{proof}

\begin{remark}
We close with a remark regarding sections of an arbitrary, not necessarily symmetric, convex body $K$ in $\mathbb{R}^n$ with affine subspaces of any dimension. The following argument was originally sketched by Webb's PhD thesis \cite{We96} for the case that $K=S_n$ is the regular $n$-dimensional simplex. Recall that the volume ratio of a convex body $K$ in $\mathbb{R}^m$ is defined by
\[
\mathrm{vr}(K)=\left(\frac{\vol_m(K)}{\vol_m({\mathcal E}_K)}\right)^\frac{1}{m},
\]
where ${\mathcal E}_K$ denotes the maximum volume ellipsoid contained in $K$. It is known by a result of Ball \cite[Thm 1']{Ba1991} that the regular simplex $S_m$ is a maximiser for the volume ratio in $\mathbb{R}^m$. It follows that if $K$ is a convex body in $\mathbb{R}^n$ and $H$ is any $k$-dimensional affine subspace of $\mathbb{R}^n$ then
\begin{equation}\label{eq:Ball-vr}
\vol_k(K\cap H)\ls\ \vol_k(S_k)\cdot \frac{\vol_k({\mathcal E}_{K\cap H})}{\vol_k (B_2^k)}.
\end{equation}
What is more, Ball has also shown \cite[Proposition 4]{Ba3} that among all  $k$-dimensional ellipsoids in a convex body $K$ in $\mathbb{R}^n$ that is in John's position, the one with the largest volume is the Eulidean ball of radius $\sqrt{\frac{k(k+1)}{n(n+1)}}$. Combined with \eqref{eq:Ball-vr}, this shows that $\vol_k(K\cap H)\ls \frac{(k(k+1))^\frac{k}{2}}{(n(n+1))^\frac{k}{2}}\vol_{k}(S_k)$.
\end{remark}

\section{The Wills functional of sections of centrally symmetric convex bodies in John's position}\label{sec:wills}

In this section we will use the method developed in Section \ref{sec:Parseval} in order to give an upper bound for the Wills functional of sections of centrally symmetric convex bodies in John's position. We will also show that from such upper bound, one can recover the upper bound for the volume of sections provided by Theorem \ref{thm:UpperBoundVolume} and the upper bound for the mean width of sections given in Theorem \ref{thm:SectionsSymmetricJohn}.

We first carry out some preliminary work towards the results of this section, starting with the following lemma.
	\begin{lemma}\label{lem:FourierTransformDistanceSquared}
		Let $\alpha>0$ and $f:\R\to\R$ be the function
		$$
		f_\alpha(x)=e^{-\pi d^2(x,[-\alpha,\alpha])}.
		$$
		Then $\widehat{f_\alpha}:\R\to\R$ is given by
		$$
		\widehat{f_\alpha}(s)=\begin{cases}\frac{2\sin(\alpha s)}{s}+\cos(\alpha s)e^{-\frac{s^2}{4\pi}}-2\sin(\alpha s)\int_0^\infty e^{-\pi y^2}\sin(ys)dy&\textrm{ if }s\neq0\\
2\alpha +1&\textrm{ if }s=0.
\end{cases}
		$$
	\end{lemma}
	
	\begin{proof}
		For every $s\in\R$, taking into account \eqref{eq:FourierTransformGaussian}, we have that
		\begin{eqnarray*}
			\widehat{f_\alpha}(s)&=&\int_\R e^{-\pi d^2(x,[-\alpha ,\alpha])}e^{ixs}dx=2\int_0^\infty e^{-\pi d^2(x,[-\alpha,\alpha])}\cos(xs)dx\cr
			&=&2\int_0^\alpha\cos(xs)dx+2\int_\alpha^\infty e^{-\pi(x-\alpha)^2}\cos(xs)dx\cr
			&=&2\int_0^\alpha\cos(xs)dx+2\int_0^\infty e^{-\pi y^2}\cos((y+\alpha)s)dy\cr
			&=&2\int_0^\alpha\cos(xs)dx+2\cos(\alpha s)\int_0^\infty e^{-\pi y^2}\cos(ys)dy-2\sin(\alpha s)\int_0^\infty e^{-\pi y^2}\sin(ys)dy\cr
			&=&2\int_0^\alpha \cos(xs)dx+\cos(\alpha s)e^{-\frac{s^2}{4\pi}}-2\sin(\alpha s)\int_0^\infty e^{-\pi y^2}\sin(ys)dy.\cr
		\end{eqnarray*}
Since
$$
\int_0^\alpha \cos(xs)dx=\begin{cases}
\frac{\sin(\alpha s)}{s} &\textrm{ if }s\neq0\\
\alpha  &\textrm{ if }s=0\\
\end{cases}
$$
we obtain the result.
	\end{proof}

Along this section, given $\alpha>0$, we will denote, for any $s\in\R$,
\begin{equation}\label{eq:abc}
a_\alpha(s)=\begin{cases}\frac{2\sin(\alpha s)}{s}&\textrm{ if }s\neq0\\2\alpha&\textrm{ if }s=0\end{cases},\quad b_\alpha(s)=\cos(\alpha s)e^{-\frac{s^2}{4\pi}},\quad\textrm{ and }\quad c_\alpha(s)=-2\sin(\alpha s)I(s),
\end{equation}
where,
$$
I(s)=\int_0^{\infty} e^{-\pi y^2}\sin(ys)dy.
$$
With this notation, for any $\alpha>0$ we have
$$
\widehat{f_\alpha}(s)=a_\alpha(s)+b_\alpha(s)+c_\alpha(s),\quad\forall s\in\R.
$$

The following lemma concerns the behavior of $I(s)$.

\begin{lemma}\label{lem:UpperBound 1-sI(s)}
There exists \(M>0\) such that
\[
\big|1-sI(s)\big|\le \frac{M}{1+s^2},\quad\forall s\in\R.
\]
\end{lemma}

\begin{proof}
Let $s\neq 0$ and call \(f(y)=e^{-\pi y^2}\). We first integrate by parts once in the definition of \(I(s)\).
Taking \(dv=\sin(ys)\,dy\) and \(u=f(y)\) we get \(v=-\frac{\cos(ys)}{s}\) and \(du=f'(y)dy=-2\pi y f(y)dy\). Hence
\[
I(s)=\Big[-\frac{f(y)\cos(ys)}{s}\Big]_{0}^{\infty}+\frac{1}{s}\int_0^\infty f'(y)\cos(ys)\,dy.
\]
Using that $\displaystyle{\lim_{y\to\infty}f(y)=0}$, \(f(0)=1\) and \(f'(y)=-2\pi y f(y)\) we obtain
\[
I(s)=\frac{1}{s}-\frac{2\pi}{s}\int_0^\infty y f(y)\cos(ys)\,dy.
\]
Therefore, calling
$$
J(s):=\int_0^\infty y f(y)\cos(ys)\,dy,
$$
we have
$$
1-sI(s)=2\pi J(s),\quad\forall s\neq0.
$$

Let us now show that \(J(s)=O(1/s^2)\) as \(|s|\to\infty\). Integrate \(J(s)\) by parts twice. Taking \(u=y f(y)\) and \(dv=\cos(ys)dy\), we have \(v=\frac{\sin(ys)}{s}\) and $du=g(y)dy$, with
\[
g(y):=u'(y)=\frac{d}{dy}\big(y e^{-\pi y^2}\big)=(1-2\pi y^2)e^{-\pi y^2}.
\]
Thus,
\[
J(s)=\Big[\frac{y f(y)\sin(ys)}{s}\Big]_0^\infty-\frac{1}{s}\int_0^\infty g(y)\sin(ys)\,dy
=-\frac{1}{s}\int_0^\infty g(y)\sin(ys)\,dy,
\]
since $\displaystyle{\lim_{y\to\infty}f(y)=0}$.

Let us integrate the remaining integral by parts again, with $u=g(y)$ and \(dv=\sin(ys)\,dy\). Thus, \(v=-\frac{\cos(ys)}{s}\) and $du=g^\prime(y)dy$ and then
\[
\int_0^\infty g(y)\sin(ys)\,dy
=\Big[-\frac{g(y)\cos(ys)}{s}\Big]_0^\infty+\frac{1}{s}\int_0^\infty g'(y)\cos(ys)\,dy.
\]
Since  $\displaystyle{\lim_{y\to\infty}g(y)=\lim_{y\to\infty}(1-2\pi y^2)e^{-\pi y^2}=0}$, and $g(0)=1$,
\[
\int_0^\infty g(y)\sin(ys)\,dy=\frac{1}{s}+\frac{1}{s}\int_0^\infty g'(y)\cos(ys)\,dy,
\]
and therefore
\[
J(s)=-\frac{1}{s^2}-\frac{1}{s^2}\int_0^\infty g'(y)\cos(ys)\,dy.
\]

From this representation we obtain that for every \(s\neq0\),
\[
|J(s)| \le \frac{1}{s^2}+\frac{1}{s^2}\int_0^\infty |g'(y)|\,dy
= \frac{C}{s^2},
\]
where
\[
C:=1+\int_0^\infty |g'(y)|\,dy,
\]
as it can be verified through a direct calculation that the integral on the right hand side is finite.

Consequently, for every \(s\neq0\),
\[
|1-sI(s)|=2\pi|J(s)|\le \frac{2\pi C}{s^2}.
\]

To obtain a single inequality valid for all \(s\in\mathbb{R}\), including small \(|s|\), define
\[
M_0:=\sup_{|s|\le 1}|1-sI(s)|\leq 1+\int_0^\infty e^{-\pi y^2}dy=\frac{3}{2},
\]
and set
\[
M:=\max\{2M_0,\;4\pi C\,\}.
\]
If \(|s|\le 1\) then \(1+s^2\le 2\) and so
\[
\frac{M}{1+s^2}\ge \frac{2M_0}{2}=M_0\ge |1-sI(s)|.
\]
If \(|s|\ge 1\) then $\frac{1}{s^2}\le \frac{2}{1+s^2}$. Hence, from the bound above
\[
|1-sI(s)|\le \frac{2\pi C}{s^2}\le \frac{4\pi C}{1+s^2}\le\frac{M}{1+s^2}.
\]
Thus for every \(s\in\mathbb{R}\) we have
\[
\big|1-sI(s)\big|\le \frac{M}{1+s^2},
\]
as desired.
\end{proof}

We will estimate from above the Wills functional of the section of a convex body $K$ in John's position applying the geometric Brascamp-Lieb inequality, in the spirit of Section \ref{sec:Parseval}. In order to state the next proposition let us denote, in view of Lemma \ref{lem:FourierTransformDistanceSquared}, for any $j\in J$ such that $\tilde{c}_j\in (0,1)$,
$$
h_{\lambda,j}(s)=\left|\widehat{f}_{\lambda\sqrt{\tilde{c}_j}t_j}(s)\right|^\frac{1}{1-\tilde c_j}.
$$
That is, for any $s\in\R\setminus\{0\}$
		$$
			h_{\lambda,j}(s)=\left|\frac{2\sin(\lambda\sqrt{\tilde{c}_j} t_js)}{s}+\cos(\lambda\sqrt{\tilde{c}_j} t_js)e^{-\frac{s^2}{4\pi}}-2\sin(\lambda\sqrt{\tilde{c}_j} t_js)\int_0^\infty e^{-\pi y^2}\sin(ys)dy\right|^{\frac{1}{1-\tilde{c}_j}}
		$$
and
$$
h_{\lambda,j}(0)=(2\lambda\sqrt{\tilde{c}_j}t_j+1)^\frac{1}{1-\tilde{c}_j}.
$$

\begin{proposition}\label{prop:WillsSection}
		Let $K\subseteq\R^n$ be a centrally symmetric convex body in John's position and let $H\in G_{n,k}$ and let $(c_j,v_j)_{j=1}^m\subseteq (0,1]\times(\partial K\cap S^{n-1})$ be its associated decomposition of the identity. Let, for every $j\in J=\{1\ls j\ls m\,:\,P_Hv_j\neq 0\}$, $\tilde{c}_j=c_j\Vert P_Hv_j\Vert_2^2$. Then, for any $\lambda>0$ we have that
		$$
		\mathcal{W}(\lambda(K\cap H))\leq\frac{1}{(2\pi)^{m_0-k}}\prod_{j\in J}\left(\frac{\int_\R h_{\lambda,j}(s)ds}{\sqrt{1-\tilde{c}_j}}\right)^{1-\tilde{c}_j},
		$$
		where, for every $j\in J$ such that $\tilde{c}_j=1$, necessarily $c_j=\Vert P_Hv_j\Vert_2=1$ and the corresponding factor is understood as $\left|\widehat{f}_{\lambda\sqrt{\tilde{c}_j}t_j}(0)\right|=2\lambda\sqrt{\tilde{c}_j}t_j+1=2\lambda+1$.
	\end{proposition}
	
	\begin{proof}
		Let $L_1:=L\cap H$. We have that for every $j\in J$
		$$
		P_{\langle u_j\rangle}(\lambda L_1)\subseteq [-\lambda t_j,\lambda t_j]u_j.
		$$
		Therefore, for every $x\in H$,
		$$
		d(\langle x,u_j\rangle u_j,P_{\langle u_j\rangle}(\lambda L_1))\geq d(\langle x, u_j\rangle u_j,[-\lambda t_j,\lambda t_j]u_j)=d(\langle x, u_j\rangle ,[-\lambda t_j,\lambda t_j]).
		$$
		Since for every $x_0\in \lambda L_1$, every $x\in H$ and every $j\in J$ we have that
		$$
		d^2(\langle x,u_j\rangle u_j,P_{\langle u_j\rangle}(\lambda L_1))\leq d^2(\langle x,u_j\rangle u_j,\langle x_0, u_j\rangle u_j)=\langle x-x_0,u_j\rangle^2,
		$$
		we have that for every $x_0\in\lambda L_1$ and every $x\in H$,
		$$
		\sum_{j\in J}\tilde{c}_jd^2(\langle x,u_j\rangle u_j,P_{\langle u_j\rangle}(\lambda L_1))\leq\sum_{j\in J}\tilde{c}_j\langle x-x_0,u_j\rangle^2=\Vert x-x_0\Vert_2^2
		$$
		and, taking infimum in $x_0\in\lambda L_1$, we have that
\[
			d^2(x,\lambda L_1)\geq\sum_{j\in J}\tilde{c}_jd^2(\langle x,u_j\rangle u_j,P_{\langle u_j\rangle}(\lambda L_1))\geq\sum_{j\in J}\tilde{c}_jd^2(\langle x, u_j\rangle ,[-\lambda t_j,\lambda t_j]).
\]
		Therefore, calling $f_j(t)=e^{-\pi d^2(t ,[-\lambda\sqrt{\tilde{c}_j} t_j,\lambda\sqrt{\tilde{c}_j} t_j])}$ for every $j\in J$, by \eqref{eq:WillsIntegral}, we have that
		\begin{align*}
			\mathcal{W}(\lambda(K\cap H))&\leq\mathcal{W}(\lambda L_1) =\int_H e^{-\pi d^2(x,\lambda C)}dx\leq\int_He^{-\sum_{j\in J}\tilde{c}_j\pi d^2(\langle x, u_j\rangle ,[-\lambda t_j,\lambda t_j])}dx\\
			&=\int_He^{-\sum_{j\in J}\pi d^2(\sqrt{\tilde{c}_j}\langle x, u_j\rangle ,[-\lambda\sqrt{\tilde{c}_j} t_j,\lambda\sqrt{\tilde{c}_j} t_j])}dx\\
			&=\int_H\prod_{j\in J} f_j(\sqrt{\tilde{c}_j}\langle x, u_j\rangle) dx.
\end{align*}
By Corollary \ref{cor:UpperBoundParsevalAndBrascampLieb} and Lemma \ref{lem:FourierTransformDistanceSquared} it follows then that,
\begin{align*}
			\mathcal{W}(\lambda C)&\leq\frac{1}{(2\pi)^{m_0-k}}\prod_{j\in J}\left(\int_{\R}\left|\widehat{f_j}^\frac{1}{1-\tilde{c}_j}(\sqrt{1-\tilde{c}_j}t)\right|dt\right)^{1-\tilde{c}_j}\\
			&=\frac{1}{(2\pi)^{m_0-k}}\prod_{j\in J}\left(\frac{\int_\R h_{\lambda,j}(s)ds}{\sqrt{1-\tilde{c}_j}}\right)^{1-\tilde{c}_j}.\qedhere
\end{align*}
\end{proof}

In the following subsections, we are going to recover estimates for the volume and the mean width of a centrally symmetric convex body in John's position, from the estimate given for the Wills functional in Proposition \ref{prop:WillsSection}.

\subsection{The volume of sections via the Wills functional}\label{subsec:VolumeSectionsWills}
In this Section we show how, from Proposition \ref{prop:WillsSection}, we can recover Theorem \ref{thm:UpperBoundVolume}. 



\begin{proof}[Proof of Theorem \ref{thm:UpperBoundVolume} via the Wills functional]
Let, for any  $\alpha>0$, $a_\alpha(s),b_\alpha(s)$, and $c_\alpha(s)$ be defined as in \eqref{eq:abc}. For any $p>1$, By Minkowski's inequality in $L^p(\R)$, we have that
$$
\left(\int_\R|\widehat{f_\alpha}(s)|^pds\right)^\frac{1}{p}\leq\left(\int_\R|a_\alpha(s)|^p\,ds\right)^\frac{1}{p}+\left(\int_\R|b_\alpha(s)|^p\,ds\right)^\frac{1}{p}+\left(\int_\R|c_\alpha(s)|^p\,ds\right)^\frac{1}{p}.
$$
Let us estimate each term:
\begin{itemize}
	\item By the scaling $u=\alpha s$,
	\[
	\left(\int_{\mathbb R}|a_\alpha(s)|^{p}\,ds\right)^{1/p}
	=2I_{p}^{1/p}\,|\alpha|^{(p-1)/p}.
	\]
	
	\item For the second term,
	\[
	\left(\int_{\mathbb R}|b_\alpha(s)|^{p}\,ds\right)^{1/p}\leq\left(\int_{\mathbb R}|b_0(s)|^{p}\,ds\right)^{1/p}
	=\left(\int_{\mathbb R}e^{-\frac{p s^2}{4\pi}}\,ds\right)^{1/p}
	=\left(\frac{2\pi}{\sqrt{p}}\right)^{1/p}.
	\]
	
	\item For the third term, notice that if $|s|\leq 1$ we have
$$
|I(s)|=\left|\int_0^\infty e^{-\pi y^2}\sin(ys)\,dy\right|\leq\int_0^\infty e^{-\pi y^2}\,dy=\frac{1}{2}\leq\frac{1}{1+s}\leq\frac{4}{1+s},
$$
and that, if $|s|\geq1$,
$$
\left|\frac{2\pi}{s}\int_0^\infty ye^{-\pi y^2}\cos(ys)\,dy\right|\leq\frac{2\pi}{|s|}\int_0^\infty ye^{-\pi y^2}\,dy=\frac{1}{|s|}
$$
and then
$$
|I(s)|=\left|\frac{1}{s}-\frac{2\pi}{s}\int_0^\infty ye^{-\pi y^2}\cos(ys)\,dy\right|\leq\frac{2}{|s|}\leq\frac{4}{1+|s|}.
$$
Therefore,
$$
\left(\int_{\mathbb R}|c_\alpha(s)|^{p}\,ds\right)^{1/p}\leq8\left(\int_\R\frac{ds}{(1+|s|)^p}\right)^{1/p}=8\frac{2^{1/p}}{(p-1)^{1/p}}.
$$
\end{itemize}
As a consequence, we have that for every fixed $j\in J$ such that $\tilde{c}_j\neq1$, taking $p_j=\frac{1}{1-\tilde{c}_j}$, there exists $M_j:[0,\infty)\to[0,\infty)$ such that $M_j(\alpha)\sim 2I_{p_j}^{1/p_j}\alpha^{\tilde{c}_j}$, as $\alpha\to\infty$, and
$$
\left(\int_\R|\widehat{f_\alpha}(s)|^{p_j}ds\right)^\frac{1}{p_j}\leq M_j(\alpha).
$$
Thus, the function $\displaystyle{M(\lambda):=\prod_{j\in J}M_j(\lambda\sqrt{\tilde{c}_j}t_j)}$, where if $\tilde{c}_j=1$, then necessarily $c_j=\Vert P_Hv_j\Vert_2=1$ and the corresponding factor is understood as $2\lambda\sqrt{\tilde{c_j}}t_j+1=2\lambda+1$, satisfies that
$$
M(\lambda)\sim 2^{m_0}\prod_{j\in J}I_{p_j}^{1/p_j}\tilde{c}_j^\frac{\tilde{c}_j}{2}t_j^{\tilde{c}_j}\lambda^k=2^{m_0}\prod_{j\in J}I_{\frac{1}{1-\tilde{c}_j}}^{1-\tilde{c}_j}c_j^\frac{\tilde{c}_j}{2}\lambda^k,\quad\lambda\to\infty,
$$
where if $\tilde{c}_j=1$, the term $I_{p_j}^{1/p_j}=1$. By the estimate granted by Proposition \ref{prop:WillsSection},
\begin{eqnarray*}
\mathcal{W}(\lambda(K\cap H))&\leq& \frac{1}{(2\pi)^{m_0-k}}\prod_{j\in J}\left(\frac{\int_{\R}h_{\lambda,j}(s)ds}{\sqrt{1-\tilde{c}_j}}\right)^{1/p_j}\leq\frac{M(\lambda)}{(2\pi)^{m_0-k}}\prod_{j\in J}\left(\frac{1}{\sqrt{1-\tilde{c}_j}}\right)^{1-\tilde{c}_j}\cr
&=&\frac{M(\lambda)\vol_k(B_\infty^k)}{2^{m_0}}\prod_{j\in J}\left(\frac{1}{\pi\sqrt{1-\tilde{c}_j}}\right)^{1-\tilde{c}_j},
\end{eqnarray*}
where if $\tilde{c}_j=1$, the corresponding factor is understood as $1$. Therefore, by \eqref{eq:WillsVolumeMeanWidth},
$$
\vol_k(K\cap H)=\lim_{\lambda\to\infty}\frac{\mathcal{W}(K\cap H)}{\lambda^k}\leq\vol_k(B_\infty^k)\prod_{j\in J}\left(
\frac{I_{\frac{1}{1-\tilde{c}_j}}}{\pi\sqrt{1-\tilde{c}_j}}\right)^{1-\tilde{c}_j}c_j^{\frac{\tilde c_j}{2}},
$$
which proves the result.
\end{proof}

\subsection{The mean width of sections via the Wills functional}\label{subsec:MeanWidthSectionsWills}

In this section we are going to showcase how, from the estimate in Proposition \ref{prop:WillsSection}, we can recover the estimate for $V_1(K\cap H)$ given in Theorem \ref{thm:SectionsSymmetricJohn}.

\begin{proposition}\label{prop:MeanWidthSectionsViaWills}
Let $K\subseteq\R^n$ be a centrally symmetric convex body in John's position and let $(c_j,v_j)_{j=1}^m\subseteq (0,1]\times(\partial K\cap S^{n-1})$ be its associated decomposition of the identity. Let $H\in G_{n,k}$ be a $k$-dimensional linear subspace and let $J=\{1\leq j\leq m\,:\,P_Hv_j\neq0\}$. Then
$$
V_1(K\cap H)\leq2\sum_{j\in J}c_j\|P_H v_j\|_2.
$$
\end{proposition}

For the proof we will rely on the following technical lemma, whose proof we postpone until the end of the section.

\begin{lemma}\label{lem:ExpansionIntegral}
	For every $p>1$ we have that
	\[
	\int_{\mathbb R}|\widehat{f_\alpha}(s)|^p\,ds = S_0 + \alpha\,S_1 + o(\alpha),\quad\alpha\to0^+,
	\]
	where
	\[
	S_0=\int_{\mathbb R} b_\alpha(s)^p\,ds=\frac{2\pi}{\sqrt p},\quad\textrm{and}\quad S_1=4\pi\sqrt{p-1}.
	\]
\end{lemma}

\begin{proof}[Proof of Proposition \ref{prop:MeanWidthSectionsViaWills}]
Let, for every $j\in J$ such that $\tilde{c}_j=c_j\Vert P_Hv_j\Vert_2^2\neq1$, $p_j=\frac{1}{1-\tilde{c}_j}>1$, and $t_j=\frac{1}{\Vert P_Hv_j\Vert_2}$. By Lemma \ref{lem:ExpansionIntegral} we have that for every $j\in J$ and every $\lambda>0$, calling $\alpha_j=\lambda\sqrt{\tilde{c}_j}t_j$
\begin{eqnarray*}
\frac{\int_{\mathbb R}h_{\lambda,j}(t)dt}{\sqrt{1-\tilde{c}_j}}
&=&\frac{1}{\sqrt{1-\tilde{c}_j}}\left(\frac{2\pi}{\sqrt{p_j}}+4\pi\sqrt{p_j-1}\alpha_j+o(\alpha_j)\right)
=2\pi + \alpha\cdot 4\pi\sqrt{p_j(p_j-1)} + o(\alpha_j)\cr
&=&2\pi\left(1 + 2\alpha_j\sqrt{p_j(p_j-1)} + o(\alpha_j)\right),  \quad\alpha_j\to0^+.
\end{eqnarray*}

Thus,
\begin{eqnarray*}
\left(\frac{\int_{\mathbb R}h_{\lambda,j}(t)dt}{\sqrt{1-\tilde{c}_j}}\right)^{1/p_j}
&=&(2\pi)^{1/p_j}\left(1 + \tfrac{2}{p_j}\alpha_j\sqrt{p_j(p_j-1)} + o(\alpha_j)\right),\cr
&=&(2\pi)^{1/p_j}\left(1 + 2\sqrt{\tilde{c}_j}\alpha_j+ o(\alpha_j)\right),\quad\alpha_j\to0^+.
\end{eqnarray*}
Equivalently, taking into account that $\alpha_j=\lambda\sqrt{\tilde{c}_j}t_j$,
$$
\left(\frac{\int_{\mathbb R}h_{\lambda,j}(t)dt}{\sqrt{1-\tilde{c}_j}}\right)^{1-\tilde c_j}
=(2\pi)^{1-\tilde{c}_j}\left(1 + 2\lambda\tilde{c}_jt_j + o(\lambda)\right),\lambda\to0^+.
$$
By Proposition \ref{prop:WillsSection}, Multiplying over $j\in J$, understanding the factors corresponding to $\tilde{c}_j=1$ as $1+2\lambda\sqrt{\tilde{c}_j}t_j=1+2\lambda=1+2\lambda\tilde{c}_jt_j$ and taking into account that $\sum_{j\in J}(1-\tilde{c}_j)=m_0-k$, we obtain
$$
\mathcal W(\lambda(K\cap H))
\leq 1 + 2\lambda\sum_{j\in J}\tilde{c}_j t_j + o(\lambda),\quad\lambda\to0^+.
$$
Therefore, by Lemma \eqref{eq:WillsVolumeMeanWidth},
$$
V_1(K\cap H)=\lim_{\lambda\to0^+}\frac{\mathcal W(\lambda(K\cap H))-1}{\lambda}\leq2\sum_{j\in J}\tilde{c}_j t_j=2\sum_{j\in J}c_j\|P_H v_j\|_2.
\eqno\qedhere$$
\end{proof}
It remains to justify the estimate in Lemma \ref{lem:ExpansionIntegral}.
\begin{proof}[Proof of Lemma \ref{lem:ExpansionIntegral}]
Let $g:[0,\infty)\to\R$ be the function
$$
g(\alpha)=\int_{\mathbb R}|\widehat{f_\alpha}(s)|^p\,ds.
$$
Let us prove that $g$ is differentiable at 0 with $g^\prime(0)=S_1=4\pi\sqrt{p-1}$. Therefore, since
$$
g(0)=\int_\R|A_0(s)|^pds=\int_\R b_0(s)ds=\int_{\R}e^{-p\frac{s^2}{4\pi}}ds=\frac{2\pi}{\sqrt p},
$$
we will have the result.

Since $p>1$, we have that the function $|x|^p$ is differentiable on $\R$ and, by the mean value theorem, for every $s\in\R$ there exists $\xi_\alpha(s)$ between $\widehat{f_\alpha}(s)$ and $b_0(s)$ such that
$$
|\widehat{f_\alpha}(s)|^p-b_0^p(s)=p|\xi_\alpha(s)|^{p-1}\textrm{sign}(\xi)(A_\alpha(s)-b_0(s)).
$$
Therefore, taking into account that for any $s\in\R$ $\displaystyle{\lim_{\alpha\to0^+}A_\alpha(s)=b_0(s)>0}$ and then $\displaystyle{\lim_{\alpha\to0^+}\xi_{\alpha}(s)}=b_0(s)$, we have that for every $s\in\R$,
\begin{eqnarray*}
\lim_{\alpha\to0^+}\frac{|\widehat{f_\alpha}(s)|^p-b_0^p(s)}{\alpha}&=&\lim_{\alpha\to0^+}\frac{p|\xi_\alpha(s)|^{p-1}\textrm{sign}(\xi_\alpha(s))(\widehat{f_\alpha}(s)-b_0(s))}{\alpha}\cr
&=&\lim_{\alpha\to0^+}\frac{p|\xi_\alpha(s)|^{p-1}(a_\alpha(s)+b_{\alpha}(s)-b_0(s)+c_\alpha(s))}{\alpha}\cr
&=&pb_0(s)^{p-1}(2-2sI(s)),
\end{eqnarray*}
where the last limit is directly obtained from the definitions of $a_\alpha(s), b_\alpha(s)$, and $c_\alpha(s)$ given in \eqref{eq:abc}. Moreover, for every $0<\alpha<1$ and every $s\in\R$,
\begin{eqnarray*}
\left|\frac{|\widehat{f_\alpha}(s)|^p-b_0^p(s)}{\alpha}\right|&=&\frac{p|\xi_\alpha(s)|^{p-1}|\widehat{f_\alpha}(s)-b_0(s)|}{\alpha}\cr
&=&\frac{p|\xi_\alpha(s)|^{p-1}|a_\alpha(s)+b_{\alpha}(s)-b_0(s)+c_\alpha(s)|}{\alpha}\cr
&\leq&p\max\{|\widehat{f_\alpha}(s)|^{p-1}, b_0^{p-1}(s)\}\left(\frac{|a_\alpha(s)+c_\alpha(s)|}{\alpha}+\frac{|b_\alpha(s)-b_0(s)|}{\alpha}\right)
\end{eqnarray*}

On the one hand, notice that for every $0<\alpha<1$ and every $s\in\R$, by the mean value theorem there exists $0<\beta_\alpha(s)<\alpha s$ such that
$$
\frac{|a_\alpha(s)+c_\alpha(s)|}{\alpha}=2|\cos(\beta_\alpha(s))||1-sI(s)|\leq 2|1-sI(s)|\leq \frac{2M}{1+s^2},.
$$
where $M$ is the constant given by Lemma \ref{lem:UpperBound 1-sI(s)}.


On the other hand, for every $0<\alpha<1$ and every $s\in\R$, by the mean value theorem there exists $0<\gamma_\alpha(s)<\alpha s$ such that
$$
\frac{|b_\alpha(s)-b_0(s)|}{\alpha}=|\sin(\gamma_\alpha(s))||s|e^{-\frac{s^2}{4\pi}}\leq|s|e^{-\frac{s^2}{4\pi}}\leq \frac{C_1}{1+s^2},
$$
where $C_1>0$ is an absolute constant. Therefore, there exists an absolute constant $C_2>0$ such that for every $0<\alpha<1$ and every $s\in\R$
$$
\left|\frac{|\widehat{f_\alpha}(s)|^p-b_0^p(s)}{\alpha}\right|\leq \frac{C_2p}{1+s^2}\max\{|\widehat{f_\alpha}(s)|^{p-1}, b_0^{p-1}(s)\}
$$

Now, notice that for every $p> 1$ we have that, for every $0<\alpha<1$ and every $s\in\R$,
\begin{eqnarray*}
|\widehat{f_\alpha}(s)|^{p-1}&\leq& 3^{p-1}\max\{|a_\alpha(s)|^{p-1}+|b_\alpha(s)|^{p-1}+|c_\alpha(s)|^{p-1}\}\cr
&\leq&3^{p-1}\max\{|a_\alpha(s)|^{p-1}+|b_0(s)|^{p-1}+|c_\alpha(s)|^{p-1}\}\cr
\end{eqnarray*}
and then
$$
\max\{|\widehat{f_\alpha}(s)|^{p-1}, b_0^{p-1}(s)\}\leq3^{p-1}\max\{|a_\alpha(s)|^{p-1}+|b_0(s)|^{p-1}+|c_\alpha(s)|^{p-1}\}.
$$

Since for any $p>1$, $0<\alpha<1$, and every $s\in\R$,
$$
|a_\alpha(s)|^{p-1}\leq2^{p-1}\min\left\{1,\frac{1}{|s|^{p-1}}\right\},\quad|b_0(s)|^{p-1}= e^{-(p-1)\frac{s^2}{4\pi}}\leq\frac{C_3}{|s|^{p-1}},
$$
where $C_3>0$ is an absolute constant, and there exists an absolute constant $C_4>0$ such that
$$
|c_\alpha(s)|^p\leq2^{p-1}|I(s)|^{p-1}\leq C_4^{p-1}\min\left\{1,\frac{1}{|s|^{p-1}}\right\}.
$$
we have that for every $p>1$ there exists an absolute constant $C>0$  such that for every $0<\alpha<1$  and every $s\in\R$
$$
\left|\frac{|\widehat{f_\alpha}(s)|^p-b_0^p(s)}{\alpha}\right|\leq C\min\left\{1,\frac{1}{|s|^{p+1}}\right\}\in L^1(\R).
$$
Therefore, by the dominated convergence theorem,
\begin{eqnarray*}
\lim_{\alpha\to0^+}\frac{g(\alpha)-g(0)}{\alpha}&=&\lim_{\alpha\to0^+}\int_{\R}\frac{|\widehat{f_\alpha}(s)|^p-b_0^p(s)}{\alpha}ds=\int_\R pb_0^{p-1}(s)(2-2sI(s))ds\cr&=&
2p\int_\R b_0^{p-1}(s)ds-2p\int_\R sI(s)b_0^{p-1}(s)ds
\end{eqnarray*}
The first integral is
	\[
	\int_{\mathbb R} b_0^{p-1}(s)\,ds=\frac{2\pi}{\sqrt{p-1}}.
	\]
	For the second one, by Fubini's theorem and the Gaussian–Fourier identity \eqref{eq:FourierTransformGaussian}, one computes
	\[
	\int_{\mathbb R} sI(s)b_0^{p-1}(s)\,ds=\frac{2\pi}{p\sqrt{p-1}}.
	\]
	Hence,
	$$
	S_1=g^\prime(0)=2p\left(\frac{2\pi}{\sqrt{p-1}}-\frac{2\pi}{p\sqrt{p-1}}\right)=4\pi\sqrt{p-1}.
	\eqno\qedhere$$
\end{proof}

\section{Sections of Generalized $\ell_p^n$ Balls}\label{sec:Kp}

In this sections we are going to provide lower bounds for the sections of generalized $\ell_p^n$-balls, $K_p((v_j)_{j=1}^m, (\alpha_j)_{j=1}^m)$, assuming that there exist some scalars $(c_j)_{j=1}^m$ such that $(c_j,v_j)_{j=1}^m\subseteq (0,1]\times S^{n-1}$ provide a decomposition of the identity. For that matter, we will use Theorem \ref{thm:parseval} in order to obtain a formula analogous to \eqref{eq:MeyerPajor}. We will follow the notation established in Section \ref{sec:GeneralSetting}.

Given $\alpha>0$, we will also denote by $f_{\alpha,p}(x)=e^{-\alpha|x|^p}$. Since for any $\alpha>0$ we have that $f_{\alpha,p}(x)=f_{1,p}\left(\alpha^{1/p}x\right)$ for every $x\in\R$, taking into account \eqref{eq:FourierScaling} and denoting
$$
\gamma_p(s):=\widehat{f_{1,p}}=\int_{-\infty}^{+\infty}e^{-|x|^p}e^{ixs}dx,
$$
we have that for any $\alpha>0$,
\begin{equation}\label{eq:FourierTransformGammapScaled}
\widehat{f_{\alpha,p}}(s)=\frac{\gamma_p\left(\frac{s}{\alpha^{1/p}}\right)}{\alpha^{1/p}}.
\end{equation}
With this notation, we have the following extension of \eqref{eq:MeyerPajor}:
\begin{lemma}\label{lem:VolumeSectionsK_p}
Let $(c_j,v_j)_{j=1}^m\subseteq (0,1]\times S^{n-1}$ providing a decomposition of the identity in $\R^n$ and let $(\alpha_j)_{j=1}^m\in(0,\infty)$. Then, for any $H\in G_{n,k}$ we have
$$
\Gamma\left(1+\frac{k}{p}\right)\vol_k(K_p\cap H)=\frac{1}{(2\pi)^{\,m_0-k}}\;
	\int_{H^\perp}\prod_{j\in J}\widehat{f_{\frac{\alpha_j}{c_j^{p/2}},p}}\left(\langle y,x_j\rangle\right)dy,
$$
where $H^\perp$ denotes the orthogonal subspace of $H$ in $\R^{m_0}$. Equivalently, calling $\alpha_p=2\Gamma\left(1+\frac{1}{p}\right)$
$$
\vol_k(K_p\cap H)=\frac{\vol_k(B_p^k)}{(2\pi)^{m_0-k}}\;
	\int_{H^\perp}\prod_{j\in J}\widehat{f_{\frac{\alpha_p\alpha_j}{c_j^{p/2}},p}}\left(\langle y,x_j\rangle\right)dy.
$$
\end{lemma}

\begin{proof}
On the one hand, notice that
$$
\int_{H} e^{-\sum_{j=1}^m\alpha_j|\langle x,v_j\rangle|^p}dx=\int_{H} e^{-\Vert x\Vert_{K_p\cap H}^p}dx=\Gamma\left(1+\frac{k}{p}\right)\vol_k(K_p\cap H).
$$
On the other hand, taking into account that $\tilde{c}_j=c_j\Vert P_Hv_j\Vert_2^2$ for every $j\in J$
\begin{eqnarray*}
\int_{H} e^{-\sum_{j=1}^m\alpha_j|\langle x,v_j\rangle|^p}dx &=& \int_{H} e^{-\sum_{j=1}^m\alpha_j|\langle x,P_Hv_j\rangle|^p}dx=\int_{H} e^{-\sum_{j\in J}\alpha_j|\langle x,P_Hv_j\rangle|^p}dx\cr
			&=&\int_{H} e^{-\sum_{j\in J}\alpha_j\Vert P_Hv_j\Vert_2^p|\langle x,u_j\rangle|^p}dx=\int_{H} e^{-\sum_{j\in J}\frac{\alpha_j}{c_j^{p/2}}|\langle x,\sqrt{\tilde{c}_j}u_j\rangle|^p}dx\cr
			&=&\int_{H} e^{-\sum_{j\in J}\frac{\alpha_j}{c_j^{p/2}}|\langle x,x_j\rangle|^p}dx.
\end{eqnarray*}
By Theorem \ref{thm:parseval}, calling for every $j\in J$, $f_j(x)=e^{-\frac{\alpha_j}{\tilde{c}_j^{p/2}}|x|^p}=f_{\frac{\alpha_j}{c_j^{p/2}},p}(x)$
\[
\int_{H} e^{-\sum_{j\in J}\frac{\alpha_j}{c_j^{p/2}}|\langle x,x_j\rangle|^p}dx = \int_{H} \prod_{j\in J}e^{-\frac{\alpha_j}{c_j^{p/2}}|\langle x,x_j\rangle|^p}dx = \frac{1}{(2\pi)^{m_0-k}}\int_{H^\perp}\prod_{j\in J}\widehat{f_{\frac{\alpha_j}{c_j^{p/2}},p}}\langle y,x_j\rangle dy.\qedhere
\]
\end{proof}

\begin{remark}
If $m=n$ $v_j=e_j$, $c_j=1$, and $\alpha_j=1$ for every $1\leq j\leq n$, and assuming that $P_He_j\neq 0$ for every $1\leq j\leq n$ (otherwise consider $H$ embedded in a lower-dimensional space), we recover the identity \eqref{eq:MeyerPajor}.
\end{remark}

Having the identity provided by Lemma \ref{lem:VolumeSectionsK_p} as a starting point, we are going to provide estimates for the volumes of $K_p\cap H$ whenever $1\leq p\leq2$. We will treat the case $p=1$ separately, since in this case the Fourier transform of $f_{\alpha,p}$ can be computed explicitly.
\subsection{Sections of $K_1$}

In this section we are going to provide estimates for $\vol_k(K_1\cap H)$, whenever $(c_j,v_j)_{j=1}^m\subseteq (0,1]\times S^{n-1}$ provide a decomposition of the identity in $\R^n$ and $H$ is a $k$-dimensional linear subspace. More precisely, we will prove the following Theorem, which extends the estimate \eqref{eq:MeyerPajorSectionsB_1^n}. We also provide an upper bound for these volumes.
\begin{theorem}\label{thm:K1-sec}
Let $(c_j,v_j)_{j=1}^m\subseteq (0,1]\times S^{n-1}$ providing a decomposition of the identity in $\R^n$ and let $(\alpha_j)_{j=1}^m\in(0,\infty)$. Then, for any $H\in G_{n,k}\subseteq G_{m_0,k}$ we have
\[
\vol_k(K_1\cap H)\gr \frac{m_0^{m_0}\Gamma\left(\frac{m_0+k}{2}\right)\prod_{j\in J}\left(\frac{\alpha_j}{\sqrt{c_j}}\right)}{\pi^{\frac{m_0-k}{2}}\Gamma(m_0)\left(\sum_{j\in J}\frac{\alpha_j^2}{c_j}\right)^\frac{m_0+k}{2}}\vol_k(B_1^k).
\]
On the other hand,
\begin{eqnarray*}
\vol_k(K_1\cap H)&\leq&  \vol_k(B_1^k)\frac{\prod_{j\in J}\left(\frac{\sqrt{c_j}}{\alpha_j}\right)^{\tilde{c}_j}}{\pi^{\frac{m_0-k}{2}}}\prod_{j\in J_1}\left(\frac{\Gamma\left(\frac{1}{1-\tilde{c}_j}-\frac{1}{2} \right)}{\sqrt{1-\tilde{c}_j}\Gamma\left(\frac{1}{1-\tilde{c}_j}\right)}\right)^{1-\tilde{c}_j}\cr
&\leq&\vol_k(B_1^k)\prod_{j=1}^m\left(\frac{\sqrt{c_j}}{\alpha_j}\right)^{c_j\Vert P_Hv_j\Vert_2^2},
\end{eqnarray*}
where $J=\{1\leq j\leq m\,:\,P_Hv_j\neq0\}$, for every $j\in J$, $\tilde{c}_j=c_j\Vert P_Hv_j\Vert_2^2$, and $\tilde{J}=\{j\in J\,:\,\tilde{c}_j\neq1\}$.
\end{theorem}
	\begin{remark}
If $m=n$ $v_j=e_j$, $c_j=1$, and $\alpha_j=1$ for every $1\leq j\leq n$, and assuming that $P_He_j\neq 0$ for every $1\leq j\leq n$ (otherwise consider $H$ embedded in a lower-dimensional space), we recover inequality \eqref{eq:MeyerPajorSectionsB_1^n}.
	\end{remark}

\begin{remark}
Let us point out that, since $\sqrt{c}_j\Vert P_H v_j\Vert_2\leq 1$ for every $j\in J$, we have that the upper estimate for $\vol_k(K_1\cap H)$ given by \eqref{eq:UpperBoundSectiosKp} is always tighter than the one given by the last term in Theorem \ref{thm:K1-sec}. Nevertheless, since the proof follows the method introduced in Section \ref{sec:Parseval}, the upper estimate given by the intermediate term, will improve on the one given by \eqref{eq:UpperBoundSectiosKp} whenever $\tilde{c}_j\geq\frac{1}{2}$ for every $j\in J$.
\end{remark}

In order to prove it, we will need a couple of technical lemmas. In the first one, we compute explicitly the Fourier transform of $f_{\alpha,1}$.

\begin{lemma}\label{lem:FourierTransformExponential}
		Let $\alpha>0$ and $f_{\alpha,1}:\R\to\R$ the function given by $f_{\alpha,1}(x)=e^{-\alpha|x|}$.Then $\widehat{f_{\alpha,1}}:\R\to\R$ is given by
		$$
		\widehat{f_{\alpha,1}}(s)=\frac{2\alpha}{\alpha^2+s^2},\quad\forall s\in\R.
		$$
	\end{lemma}
	
	\begin{proof}
		For every $s\in\R$ we have that
		\begin{eqnarray*}
			\widehat{f_{\alpha,1}}(s)&=&\int_{\R}f_{\alpha,1}(x)e^{ixs}dx=\int_{\R}e^{-\alpha|x|}\left(\cos(xs)+i\sin(xs)\right)dx=\int_{\R}e^{-\alpha|x|}\cos(xs)dx\cr
			&=&2\int_0^\infty e^{-\alpha x}\cos(xs)dx.
		\end{eqnarray*}
		Integrating by parts twice, we have
		\begin{eqnarray*}
			\int_0^\infty e^{-\alpha x}\cos(xs)dx&=&\frac{1}{\alpha}-\frac{s}{\alpha}\int_0^\infty e^{-\alpha x}\sin(xs)dx\cr
			&=&\frac{1}{\alpha}-\frac{s^2}{\alpha^2}\int_0^\infty e^{-\alpha x}\cos(xs)dx.
		\end{eqnarray*}
		Therefore,
		$$
		\left(1+\frac{s^2}{\alpha^2}\right)\int_0^\infty e^{-\alpha x}\cos(xs)dx=\frac{1}{\alpha}\Leftrightarrow\int_0^\infty e^{-\alpha x}\cos(xs)dx=\frac{1}{\alpha}\frac{1}{1+\frac{s^2}{\alpha^2}}=\frac{\alpha}{\alpha^2+s^2},
		$$
		and then
		\[
		\widehat{f_{\alpha,1}}(s)=2\int_0^\infty e^{-\alpha x}\cos(xs)dx=\frac{2\alpha}{\alpha^2+s^2}.\qedhere
		\]
	\end{proof}

The following lemma provides an estimate for the intermediate upper bound we  obtain for the volume of $K_1\cap H$.

\begin{lemma}\label{lem:Gamma-Karamata}
Let $(c_j,v_j)_{j=1}^m\subseteq (0,1]\times S^{n-1}$ providing a decomposition of the identity in $\R^n$. Let $H\in G_{n,k}$ be a $k$-dimensional linear subspace and let $J=\{1\leq j\leq m\,:\,P_Hv_j\neq 0\}$ and  $\tilde{J}=\{j\in J: \tilde c_j\neq 1\}$. Then
\[
\prod_{j\in \tilde{J}}\left(\frac{\Gamma\left(\frac{1}{1-\tilde{c}_j}-\frac{1}{2} \right)}{\sqrt{1-\tilde{c}_j}\Gamma\left(\frac{1}{1-\tilde{c}_j}\right)}\right)^{1-\tilde{c}_j}\leq\pi^{\frac{m_0-k}{2}}.
\]
\end{lemma}

\begin{proof}
Notice that the wanted inequality is equivalent to
		\begin{equation}\label{eq:EquivalentUpperBoundLemGammaKaramata}
		\sum_{j\in \tilde{J}}(1-\tilde{c}_j)\log\left(\frac{\Gamma\left(\frac{1}{1-\tilde{c}_j}-\frac{1}{2} \right)}{\sqrt{1-\tilde{c}_j}\Gamma\left(\frac{1}{1-\tilde{c}_j}\right)}\right)\leq(m_0-k)\log\Gamma\left(\frac{1}{2}\right).
		\end{equation}
		Let $f:[0,1]\to\R$ be the function given by
		$$
		f(x)=\begin{cases}x\log\left(\frac{\Gamma\left(\frac{1}{x}-\frac{1}{2} \right)}{\sqrt{x}\Gamma\left(\frac{1}{x}\right)}\right) &\textrm{ if } x\in(0,1],\\
\lim_{x\to0^+}f(x)=0&\textrm{ if } x=0.
\end{cases}
		$$
		We have that $f$ is continuous on $[0,1]$. Besides, for every $x\in (0,1)$ we have that
		$$
		f^{\prime\prime}(x)=-\frac{x^2-2\psi^\prime\left(\frac{1}{x}-\frac{1}{2}\right)+2\psi^\prime\left(\frac{1}{x}\right)}{2x^3},
		$$
		where $\psi$ denotes the logarithmic derivative of the Gamma function.
		
The convexity of $\psi'$ implies that
\[
2\left(\psi'\left(\frac{1}{x}-\frac{1}{2}\right)-\psi'\left(\frac{1}{x}\right)\right)\gr-\psi''\left(\frac{1}{x}\right)=2\sum_{k=0}^\infty\left(\frac{1}{\frac{1}{x}+k}\right)^3\gr 2\int_0^\infty \frac{dt}{(\frac{1}{x}+t)^3}=x^2,
\]
for any $x\in(0,1)$. This implies that $f''\gr 0$, hence $f$ is convex on $(0,1)$. Then the function $F:M\to\R$ given by
		$$
		F(x)=\sum_{j\in J}f(x_j),
		$$
		where
$$
M=\{(x_j)_{j\in J}\,:\,0\leq x_j\leq 1\,\forall j\in J,\,\sum_{j\in J}x_j=m_0-k\},$$
which is convex and invariant under permutations of coordinates. Besides, for every $(x_j)_{j\in J}\in M$ we have that
		$$
		(x_j)_{j\in J}\prec(1,\dots,1,0,\dots,0)=(y_j)_{j\in J},
		$$
		where there are $m_0-k$ 1's and $k$ 0's. That is, identifying $J$ with $\R^{m_0}$.
\begin{itemize}
	\item $\displaystyle{\sum_{j=1}^{m_0}x_j=m_0-k=\sum_{j=1}^{m_0}y_j}$ and
	\item For every $1\leq k\leq m_0$, we have $\displaystyle{\sum_{j=1}^kx_j^*\ls\sum_{j=1}^ky_j^*}$, where $(z_j^\ast)_{j=1}^{m_0}$ denotes the non-increasing rearrangement of $(z_j)_{j=1}^{m_0}$.
\end{itemize}


Therefore, by Karamata's inequality, for every $(x_j)_{j\in J}\in K$ we have
		$$
		\sum_{j\in J}f(x_j)\leq \sum_{j\in J}f(y_j)=(m_0-k)\Gamma\left(\frac{1}{2}\right).
		$$
		In particular, since $(1-\tilde{c}_j)_{j\in J}\in M$ and
		$$
		\sum_{j\in J}f(1-\tilde{c}_j)=\sum_{j\in \tilde{J}} f(1-\tilde{c}_j)=\sum_{j\in \tilde{J}}(1-\tilde{c}_j)\log\left(\frac{\Gamma\left(\frac{1}{1-\tilde{c}_j}-\frac{1}{2} \right)}{\sqrt{1-\tilde{c}_j}\Gamma\left(\frac{1}{1-\tilde{c}_j}\right)}\right),
		$$
		we obtain \eqref{eq:EquivalentUpperBoundLemGammaKaramata}.
\end{proof}

Now, we proceed to the proof of Theorem \ref{thm:K1-sec}. We follow the ideas in \cite{MP} and, using the explicit expression of the Fourier transform of $f_{\alpha,1}$, together with Lemma \ref{lem:VolumeSectionsK_p}.

\begin{proof}[Proof of Theorem \ref{thm:K1-sec}]

Combining Lemma \ref{lem:VolumeSectionsK_p} and Lemma \ref{lem:FourierTransformExponential}, we obtain

\begin{eqnarray*}
k!\vol_k(K_1\cap H)&=&\frac{1}{(2\pi)^{\,m_0-k}}\;
	\int_{H^\perp}\prod_{j\in J}\widehat{f_{\frac{\alpha_j}{\sqrt{c_j}},1}}\left(\langle y,x_j\rangle\right)dy\cr
&=&\frac{\prod_{j\in J}\frac{2\alpha_j}{\sqrt{c_j}}}{(2\pi)^{\,m_0-k}}\;
	\int_{H^\perp}\prod_{j\in J}\frac{1}{\frac{\alpha_j^2}{c_j}+\langle y,x_j\rangle^2}dy,
\end{eqnarray*}
where $H^\perp$ denotes the orthogonal subspace to $H$ in $\R^{m_0}$. Thus,
\begin{equation}\label{eq:IdentitySectionsK_1}
\vol_k(K_1\cap H)=\vol_k(B_1^k)\frac{\prod_{j\in J}\frac{\alpha_j}{\sqrt{c_j}}}{\pi^{\,m_0-k}}\;
	\int_{H^\perp}\prod_{j\in J}\frac{1}{\frac{\alpha_j^2}{c_j}+\langle y,x_j\rangle^2}dy.
\end{equation}
Let us first prove the lower bound: By the arithmetic-geometric mean inequality, we have that
		$$
		\prod_{j\in J}\left(\frac{\alpha_j^2}{c_j}+\langle y,x_j\rangle^2\right)^\frac{1}{m_0}\leq\frac{1}{m_0}\sum_{j\in J}\left(\frac{\alpha_j^2}{c_j}+\langle y,x_j\rangle^2\right)=\frac{1}{m_0}\sum_{j\in J}\frac{\alpha_j^2}{c_j}+\frac{\Vert y\Vert_2^2}{m_0}
		$$
Therefore, integrating in polar coordinates in $H^\perp$,
\begin{eqnarray*}
\int_{H^\perp}\prod_{j\in J}\frac{1}{\frac{\alpha_j^2}{c_j}+\langle y,x_j\rangle^2}dy&\geq&
	\int_{H^\perp}\frac{m_0^{m_0}}{\left(\sum_{j\in J}\frac{\alpha_j^2}{c_j}+\Vert y\Vert_2^2\right)^{m_0}}dy\cr
	&=& \frac{m_0^{m_0}(m_0-k)\pi^\frac{m_0-k}{2}}{\Gamma\left(1+\frac{m_0-k}{2}\right)}\int_0^\infty\frac{r^{m_0-k-1}}{\left(\sum_{j\in J}\frac{\alpha_j^2}{c_j}+r^2\right)^{m_0}}dr\cr
	&=&\frac{m_0^{m_0}(m_0-k)\pi^\frac{m_0-k}{2}}{\Gamma\left(1+\frac{m_0-k}{2}\right)\left(\sum_{j\in J}\frac{\alpha_j^2}{c_j}\right)^\frac{m_0+k}{2}}\int_0^\infty\frac{s^{\frac{m_0-k}{2}-1}}{2\left(1+s\right)^{m_0}}ds.\cr
\end{eqnarray*}
Taking into account that
\begin{eqnarray*}
\int_0^\infty\frac{s^{\frac{m_0-k}{2}-1}}{2\left(1+s\right)^{m_0}}ds&=&\frac{1}{2}\beta\left(\frac{m_0-k}{2},\frac{m_0+k}{2},\right)=\frac{\Gamma\left(\frac{m_0-k}{2}\right)\Gamma\left(\frac{m_02k}{2}\right)}{2\Gamma(m_0)}\cr
&=&\frac{\Gamma\left(1+\frac{m_0-k}{2}\right)\Gamma\left(\frac{m_0+k}{2}\right)}{(m_0-k)\Gamma(m_0)},
\end{eqnarray*}
we obtain
$$
\int_{H^\perp}\prod_{j\in J}\frac{1}{\frac{\alpha_j^2}{c_j}+\langle y,x_j\rangle^2}dy\geq\frac{m_0^{m_0}\Gamma\left(\frac{m_0+k}{2}\right)\pi^\frac{m_0-k}{2}}{\Gamma(m_0)\left(\sum_{j\in J}\frac{\alpha_j^2}{c_j}\right)^\frac{m_0+k}{2}}.
$$
which, together with \eqref{eq:IdentitySectionsK_1}, gives the lower bound for $\vol_k(K\cap H)$. Let us now proceed with the upper bound.

We first manipulate the term in the right-hand side of \eqref{eq:IdentitySectionsK_1} as follows,
\begin{eqnarray*}
\prod_{j\in J}\left(\frac{\alpha_j}{\sqrt{c_j}}\right)\int_{H^\perp}\prod_{j\in J}\frac{1}{\frac{\alpha_j^2}{c_j}+\langle y,x_j\rangle^2}dy&=&
\prod_{j\in J}\left(\frac{\sqrt{c_j}}{\alpha_j}\right)\int_{H^\perp}\prod_{j\in J}\frac{1}{1+\frac{c_j(1-\tilde{c}_j)}{\alpha_j^2}\langle y,w_j\rangle^2}dy\\
&=&\prod_{j\in J}\left(\frac{\sqrt{c_j}}{\alpha_j}\right)\int_{H^\perp}\prod_{j\in \tilde{J}}\frac{1}{1+\frac{c_j(1-\tilde{c}_j)}{\alpha_j^2}\langle y,w_j\rangle^2}dy,
\end{eqnarray*}
where $\tilde{J}=\{j\in J\,:\, \tilde{c}_j\neq1\}=\{j\in J\,:\, x_j\not\in H\}=\{j\in J\,:\,P_{H^\perp}x_j\neq0\}$. Since $(1-\tilde{c}_j,w_j)_{j\in \tilde{J}}$ provide a decomposition of the identity in $H^\perp$, applying the geometric Brascamp-Lieb inequality (Theorem \ref{thm:GeometricBrascampLieb}) in $H^\perp$ we obtain
\begin{align*}
\int_{H^\perp}\prod_{j\in \tilde{J}}\frac{1}{1+\frac{c_j(1-\tilde{c}_j)}{\alpha_j^2}\langle y,w_j\rangle^2}dy &\ls \prod_{j\in \tilde{J}}\left(\int_{\R}\frac{1}{\left(1+\frac{c_j(1-\tilde{c}_j)}{\alpha_j^2}x^2\right)^\frac{1}{1-\tilde{c}_j}}\,dx\right)^{1-\tilde{c}_j}\\
                                      &= \prod_{j\in \tilde{J}}\left(\frac{\alpha_j}{\sqrt{c_j(1-\tilde{c}_j})}\int_{\R}\frac{1}{\left(1+x^2\right)^\frac{1}{1-\tilde{c}_j}}\,dx\right)^{1-\tilde{c}_j}.
\end{align*}
Since
\[
\int_{\R}\frac{1}{\left(1+x^2\right)^\frac{1}{1-\tilde{c}_j}}\,dx =\beta\left(\frac{1}{2},\frac{1}{1-\tilde{c}_j}-\frac{1}{2}\right)= \frac{\sqrt{\pi}\Gamma\left(\frac{1}{1-\tilde{c}_j}-\frac{1}{2} \right)}{\Gamma\left(\frac{1}{1-\tilde{c}_j}\right)},
\]
we obtain
\begin{equation}
\prod_{j\in J}\left(\frac{\alpha_j}{\sqrt{c_j}}\right)\int_{H^\perp}\prod_{j\in J}\frac{1}{\frac{\alpha_j^2}{c_j}+\langle y,x_j\rangle^2}dy\ls \prod_{j\in J}\left(\frac{\sqrt{c_j}}{\alpha_j}\right)^{\tilde{c}_j}\prod_{j\in J_1}\left(\frac{\Gamma\left(\frac{1}{1-\tilde{c}_j}-\frac{1}{2} \right)}{\sqrt{1-\tilde{c}_j}\Gamma\left(\frac{1}{1-\tilde{c}_j}\right)}\right)^{1-\tilde{c}_j}.
\end{equation}
Taking into account \eqref{eq:IdentitySectionsK_1},
\begin{equation}\label{eq:I1-upper-intermediate}
\vol_k(K_1\cap H)\leq  \vol_k(B_1^k)\frac{\prod_{j\in J}\left(\frac{\sqrt{c_j}}{\alpha_j}\right)^{\tilde{c}_j}}{\pi^{\frac{m_0-k}{2}}}\prod_{j\in J_1}\left(\frac{\Gamma\left(\frac{1}{1-\tilde{c}_j}-\frac{1}{2} \right)}{\sqrt{1-\tilde{c}_j}\Gamma\left(\frac{1}{1-\tilde{c}_j}\right)}\right)^{1-\tilde{c}_j}.
\end{equation}
Lemma \ref{lem:Gamma-Karamata} concludes the proof.
\end{proof}

We stress that the method of the proof provides the intermediate bound that improves upon the general estimate, by means of an extra term taking into account the lengths of the projections of the vectors $v_j$ onto the complementary subspace $H^\perp$. In the case that $m=n$, $v_j=e_j$, $c_j=1$, and $(\alpha_j)_{j=1}^n$ for every $1\leq j\leq n$, then $K_1=B_1^n$ and, in view of \eqref{eq:I1-upper-intermediate} we obtain the following:
	\begin{corollary}\label{cor:B1n-bound}
		For any $H\in G_{n,k}$ such that $e_j\not\in H\cup H^\perp$ and for every $1\leq j\leq n$, we have that
		$$
		\vol_k(B_1^n\cap H)\leq\vol_k(B_1^k)\frac{1}{\pi^\frac{n-k}{2}}\prod_{j=1}^n\left(\frac{\Gamma\left(\frac{1}{\Vert P_{H^\perp}e_j\Vert_2^2}-\frac{1}{2} \right)}{\Vert P_{H^\perp}e_j\Vert_2\Gamma\left(\frac{1}{\Vert P_{H^\perp}e_j\Vert_2^2}\right)}\right)^{\Vert P_{H^\perp}e_j\Vert_2^2}\leq\vol_k(B_1^k).
		$$
	\end{corollary}

\subsection{Sections of $K_p$, $1\leq p\leq2$}
In this section we will provide estimates for the volume of  $K_p\cap H$ whenever $p\in(1,2]$ and $(c_j,v_j)_{j=1}^m\subseteq (0,1]\times S^{n-1}$ provide a decomposition of the identity in $\R^n$ and $H$ is a $k$-dimensional linear subspace. Our result reads as follows.
\begin{theorem}\label{thm:general_p_bounds}
Let $1\leq p\leq 2$, $(c_j,v_j)_{j=1}^m\subseteq (0,1]\times S^{n-1}$ providing a decomposition of the identity in $\R^n$ and let $(\alpha_j)_{j=1}^m\in(0,\infty)$. Then, for any  $H\in G_{n,k}\subseteq G_{m_0,k}$ we have
$$
\vol_k(K_p\cap H)\geq \frac{\vol_k(B_p^k)(\pi(m_0-k))^{\frac{m_0-k}{2}}}{\alpha_p^k(2\pi)^{m_0-k}}\prod_{j\in J}\left(\frac{\sqrt{c_j}}{\alpha_j^\frac{1}{p}}\right)\frac{1}{\Gamma(\beta)} \int_0^\infty t^{\beta-1} \prod_{j\in \tilde{J}}\gamma_p\left(\sqrt{t \tfrac{c_j}{\alpha_j^{2/p}}(1-\tilde{c}_j)}\right)  dt.
$$
where $\beta=(m_0-k)/2$. On the other hand,
\begin{eqnarray*}
\vol_k(K_p\cap H)&\leq&\frac{\vol_k(B_p^k)\prod_{j\in J}\left(\frac{c_j^\frac{p}{2}}{\alpha_j}\right)^{\frac{c_j\Vert P_Hv_j\Vert_2^2}{p}}}{\alpha_p^k(2\pi)^{m_0-k}}
	\prod_{j\in\tilde{J}}\left(\int_\R \gamma_p^\frac{1}{1-\tilde{c}_j}\left(\sqrt{1-\tilde{c}_j}t\right)dt\right)^{1-\tilde{c}_j}\cr
&\leq&\vol_k(B_p^k)\prod_{j\in J}\left(\frac{c_j^\frac{p}{2}}{\alpha_j}\right)^{\frac{c_j\Vert P_Hv_j\Vert_2^2}{p}},
\end{eqnarray*}
where, $J=\{1\leq j\leq m\,:\, P_Hv_j\neq0\}$, $\tilde{J}=\{j\in J\,:\,\tilde{c}_j\neq 1\}$ being, for every $j\in J$, $\tilde{c}_j=c_j\Vert P_Hv_j\Vert_2^2$.
\end{theorem}

\begin{remark}
Let us point out that, if $p\neq1$ we cannot explicitly compute the function $\gamma_p$. Besides, since for every $1\leq p\leq 2$ we have that
$c_j^{1-\frac{p}{2}}\Vert P_Hv_j\Vert_2^{2-p}\leq 1$ for every $j\in J$, the upper estimate provided by \eqref{eq:UpperBoundSectiosKp} is tighter than the largest upper bound given in Theorem \ref{thm:general_p_bounds}. Nevertheless, since the immediate term is obtained following the method introduced in Section \ref{sec:Parseval}, the upper estimate given by the intermediate term, will improve on the one given by \eqref{eq:UpperBoundSectiosKp} whenever $\tilde{c}_j\geq\frac{1}{2}$ for every $j\in J$.
\end{remark}

In order to study the upper bound in Theorem \ref{thm:general_p_bounds} we will need the following lemma about the Fourier transform, $\gamma_p$, of the function $e^{-|x|^p}$.
\begin{lemma}\label{gamma_p}
Let \(0<p<2\) and \(f(x)=e^{-|x|^{p}}\). Let
\[
	\gamma_p(s):=\int_{\mathbb R}e^{ixs}f(x)dx.
	\]
Then, for all \(r\ge 1\) we have
	\[
	\int_{\mathbb R}\gamma_p\!\left(\frac{t}{\sqrt{r}}\right)^{\!r}\,dt
	\le 2\pi\left(2\Gamma\left(1+\tfrac{1}{p}\right)\right)^{r-1}.
	\]
\end{lemma}

\begin{proof}
It is well known (see, for instance, \cite[Theorem 2.]{ENT}) that there is a positive measure \(\mu\) on \([0,\infty)\) such that for every $s\in\R$
	\[
	\gamma_p(s)=\int_{0}^{\infty}e^{-s^{2}t}\,d\mu(t),
	\qquad \mu([0,\infty))=\gamma_p(0)=\int_{\mathbb R}f(x)\,dx=2\Gamma\left(1+\frac{1}{p}\right)=\alpha_p.
	\]
Let $\nu$ be the probability measure on $[0,\infty)$ given by \(\nu=\frac{1}{\alpha_p}\mu\) and let us write  \(\gamma_p(s)=\alpha_p\int_{0}^{\infty}e^{-s^{2}t}\,d\nu(t)\). An application of H\"{o}lder's inequality gives that for any $r\geq1$
	\[
	\left(\int_{0}^{\infty}e^{-s^{2}t/r}\,d\nu(t)\right)^{r}\le\int_{0}^{\infty}e^{-s^{2}t}\,d\nu(t).
	\]
Multiplying by \(\alpha_p^{r}\) and integrating with respect to  \(s\in\mathbb R\) we get
	\[
	\int_{\mathbb R}\gamma_p\left(\frac{s}{\sqrt{r}}\right)^{r}\,ds
	\le \alpha_p^{r-1}\int_{\mathbb R}\gamma_p(s)\,ds.
	\]
To this end, note that by \eqref{eq:FourierInversionAt0}
	\[
	f(0)=\frac{1}{2\pi}\int_{\mathbb R}\gamma_p(s)\,ds.
	\]
Therefore \(\displaystyle\int_{\mathbb R}\gamma_p(s)\,ds=2\pi f(0)=2\pi\) and the lemma follows.
\end{proof}
Let us now prove Theorem \ref{thm:general_p_bounds}.

\begin{proof}[Proof of Theorem \ref{thm:general_p_bounds}]

Let us first prove the lower bound for $\vol_k(K_p\cap H)$. Fix $1\leq p\leq 2$. Starting from Lemma \ref{lem:VolumeSectionsK_p}, and taking into account \eqref{eq:FourierTransformGammapScaled}, we obtain

\begin{eqnarray}\label{eq:ChangeIntegralK_p}
\Gamma\left(1+\frac{k}{p}\right)\vol_k(K_p\cap H)&=&\frac{1}{(2\pi)^{m_0-k}}
	\int_{H^\perp}\prod_{j\in J}\widehat{f_{\frac{\alpha_j}{c_j^{p/2}},p}}\left(\langle y,x_j\rangle\right)dy\cr
&=&\frac{1}{(2\pi)^{m_0-k}}\prod_{j\in J\setminus\tilde{J}}\widehat{f_{\frac{\alpha_j}{c_j^{p/2}},p}}(0)
	\int_{H^\perp}\prod_{j\in \tilde{J}}\widehat{f_{\frac{\alpha_j}{c_j^{p/2}},p}}\left(\sqrt{1-\tilde{c}_j}\langle y,w_j\rangle\right)dy\cr
&=&\frac{1}{(2\pi)^{m_0-k}}\prod_{j\in J\setminus\tilde{J}}\left(\frac{\sqrt{c}_j}{\alpha_j^\frac{1}{p}}\right)
	\int_{H^\perp}\prod_{j\in \tilde{J}}\frac{\sqrt{c_j}}{\alpha_j^\frac{1}{p}}\gamma_p\left(\frac{\sqrt{c_j}\sqrt{1-\tilde{c}_j}\langle y,w_j\rangle}{\alpha_j^\frac{1}{p}}\right)dy\cr
&=&\frac{1}{(2\pi)^{m_0-k}}\prod_{j\in J}\left(\frac{\sqrt{c_j}}{\alpha_j^\frac{1}{p}}\right)
	\int_{H^\perp}\prod_{j\in \tilde{J}}\gamma_p\left(\frac{\sqrt{c_j}\sqrt{1-\tilde{c}_j}\langle y,w_j\rangle}{\alpha_j^\frac{1}{p}}\right)dy.\cr
\end{eqnarray}

Let $\mu$ be the positive measure on $[0,\infty)$ such that
\begin{equation}\label{eq:GaussianMixture}
	\gamma_p(s)=\int_0^\infty e^{-s^2 t}\,d\mu(t),\qquad \forall s\in\R .
\end{equation}
	and recall that $\mu([0,\infty))=\alpha_p$.

Hence, calling $m_1=\sharp \tilde{J}$ the product inside the integral equals
	\[
	\prod_{j\in \tilde{J}}\gamma_p\left(\frac{\sqrt{c_j}}{\alpha_j^{1/p}}\langle y,x_j\rangle\right)
	= \int_{[0,\infty)^{m_1}} \exp\left(-\sum_{j\in \tilde{J}} \tfrac{c_j}{\alpha_j^{2/p}} t_j \langle y,x_j\rangle^2\right)\,d\mu^{\otimes m_1}(t).
	\]
Using Fubini's theorem we obtain
	\[
\Gamma\left(1+\frac{k}{p}\right)\vol_k(K_p\cap H)=\frac{1}{(2\pi)^{m_0-k}}\prod_{j\in J}\left(\frac{\sqrt{c_j}}{\alpha_j^\frac{1}{p}}\right)
	\int_{[0,\infty)^{m_1}}\left(\int_{H^\perp} e^{-y^T S(t) y}\,dy\right)\,d\mu^{\otimes m_1}(t),
	\]
	where, for every $t\in[0,\infty)^{m_1}$, $S(t)$ is the symmetric positive operator on $H^\perp$ given by
	\[
	S(t)=\sum_{j\in \tilde{J}}\tfrac{c_j}{\alpha_j^{2/p}} t_j\,(x_j\otimes x_j)\Big|_{H^\perp}.
	\]
	Computing  the Gaussian integral and using the arithmetic-geometric mean inequality,
$$
\int_{H^\perp} e^{-y^T S(t) y}dy=\frac{\pi^\frac{m_0-k}{2}}{\det(S(t))^{1/2}}\geq \left(\frac{\pi(m_0-k)}{\textrm{tr}(S(t))}\right)^\frac{m_0-k}{2}.
$$

	Since for every $j\in\tilde{J}$  we have $\textrm{tr}(x_j\otimes x_j|_{H^\perp})=\|P_{H^\perp} x_j\|_2^2=1-\tilde c_j$, we have that for every $t\in[0,\infty)^{m_1}$
	\[
	\textrm{tr}(S(t))=\sum_{j\in \tilde{J}}\tfrac{c_j}{\alpha_j^{2/p}} (1-\tilde c_j)t_j,
	\]
and then
$$
\Gamma\left(1+\frac{k}{p}\right)\vol_k(K_p\cap H)\geq \frac{(\pi(m_0-k))^{\frac{m_0-k}{2}}}{(2\pi)^{m_0-k}}\prod_{j\in J}\left(\frac{\sqrt{c_j}}{\alpha_j^\frac{1}{p}}\right)\int_{[0,\infty)^{m_1}}\left(\textrm{tr}(S(t))\right)^{-\frac{m_0-k}{2}}d\mu(t)
$$
To obtain the desired bound we will use the formula:
\[
x^{-\beta} = \frac{1}{\Gamma(\beta)} \int_{0}^{\infty} r^{\beta-1} e^{-rx}dr, \qquad \text{for } x>0, \beta>0,
\]
with $\beta=(m_0-k)/2$ and ,
\[
x = \textrm{tr}(S(t)) = \sum_{j\in \tilde{J}}\tfrac{c_j}{\alpha_j^{2/p}} (1-\tilde c_j)t_j
\]
Substituting this expression in the integral over the variables $(t_j)_{j\in \tilde{J}}$ and applying Fubini's theorem, we get:
\begin{align*}
	\int_{[0,\infty)^{m_1}} & ( \textrm{tr}(S(t)))^{-\beta} \,d\mu^{\otimes m_1}(t)= \\
	&= \frac{1}{\Gamma(\beta)} \int_0^\infty r^{\beta-1} \left( \int_{[0,\infty)^{m_1}} \exp\left(-r \sum_{j\in \tilde{J}}\tfrac{c_j}{\alpha_j^{2/p}} (1-\tilde c_j) t_j\right) d\mu^{\otimes m_1}(t) \right) dr \\
	&= \frac{1}{\Gamma(\beta)} \int_0^\infty r^{\beta-1} \prod_{j\in \tilde{J}} \left( \int_0^\infty \exp\left(-r \Big[\tfrac{c_j}{\alpha_j^{2/p}} (1-\tilde c_j)\Big] t_j\right) d\mu(t_j) \right) dr.
\end{align*}
From \eqref{eq:GaussianMixture}, we have that for each $j\in\tilde{J}$,
$$
\int_0^\infty \exp\left(-r \Big[\tfrac{c_j}{\alpha_j^{2/p}} (1-\tilde c_j)\Big] t_j\right) d\mu(t_j)=\gamma_p\left(\sqrt{r \tfrac{c_j}{\alpha_j^{2/p}}(1-\tilde{c}_j)}\right).
$$
Therefore,
$$
\int_{[0,\infty)^{m_1}} ( \textrm{tr}(S(t)))^{-\beta} \,d\mu^{\otimes m_1}(t)=  \frac{1}{\Gamma(\beta)} \int_0^\infty r^{\beta-1} \prod_{j\in \tilde{J}}\gamma_p\left(\sqrt{r \tfrac{c_j}{\alpha_j^{2/p}}(1-\tilde{c}_j)}\right)  dr
$$
and then
$$
\Gamma\left(1+\frac{k}{p}\right)\vol_k(K_p\cap H)\geq \frac{(\pi(m_0-k))^{\frac{m_0-k}{2}}}{(2\pi)^{m_0-k}}\prod_{j\in J}\left(\frac{\sqrt{c_j}}{\alpha_j^\frac{1}{p}}\right)\frac{1}{\Gamma(\beta)} \int_0^\infty r^{\beta-1} \prod_{j\in \tilde{J}}\gamma_p\left(\sqrt{r \tfrac{c_j}{\alpha_j^{2/p}}(1-\tilde{c}_j)}\right)  dr.
$$

Let us now prove the upper bound for $\vol_k(K_p\cap H)$.

Starting from equation \eqref{eq:ChangeIntegralK_p} and applying the geometric Brascamp--Lieb inequality (Theorem \ref{thm:GeometricBrascampLieb}) in $H^\perp$, and taking into account that if $j\in J\setminus\tilde{J}$, then $\tilde{c}_j=1$ and necessarily $c_j=1$,
\begin{eqnarray*}
\Gamma\left(1+\frac{k}{p}\right)\vol_k(K_p\cap H)&\leq&\frac{1}{(2\pi)^{m_0-k}}\prod_{j\in J}\left(\frac{\sqrt{c_j}}{\alpha_j^\frac{1}{p}}\right)
	\prod_{j\in\tilde{J}}\left(\int_\R \gamma_p^\frac{1}{1-\tilde{c}_j}\left(\frac{\sqrt{c_j}\sqrt{1-\tilde{c}_j}t}{\alpha_j^\frac{1}{p}}\right)dt\right)^{1-\tilde{c}_j}\cr
&=&\frac{1}{(2\pi)^{m_0-k}}\prod_{j\in J}\left(\frac{\sqrt{c_j}}{\alpha_j^\frac{1}{p}}\right)
	\prod_{j\in\tilde{J}}\left(\frac{\alpha_j^\frac{1}{p}}{\sqrt{c_j}}\right)^{1-\tilde{c}_j}\left(\int_\R \gamma_p^\frac{1}{1-\tilde{c}_j}\left(\sqrt{1-\tilde{c}_j}t\right)dt\right)^{1-\tilde{c}_j}\cr
&=&\frac{\prod_{j\in J}\left(\frac{\sqrt{c_j}}{\alpha_j^\frac{1}{p}}\right)^{\tilde{c}_j}}{(2\pi)^{m_0-k}}
	\prod_{j\in\tilde{J}}\left(\int_\R \gamma_p^\frac{1}{1-\tilde{c}_j}\left(\sqrt{1-\tilde{c}_j}t\right)dt\right)^{1-\tilde{c}_j}.\cr
\end{eqnarray*}
Therefore, since $\alpha_p=2\Gamma\left(1+\frac{1}{p}\right)$, we have
$$
\vol_k(K_p\cap H)\leq\vol_k(B_p^k)\frac{\prod_{j\in J}\left(\frac{\sqrt{c_j}}{\alpha_j^\frac{1}{p}}\right)^{\tilde{c}_j}}{\alpha_p^k(2\pi)^{m_0-k}}
	\prod_{j\in\tilde{J}}\left(\int_\R \gamma_p^\frac{1}{1-\tilde{c}_j}\left(\sqrt{1-\tilde{c}_j}t\right)dt\right)^{1-\tilde{c}_j}.
$$
By Lemma \ref{gamma_p}  we have that, for every $j\in \tilde{J}$,
	\[
	\left(\int_\R \gamma_p^\frac{1}{1-\tilde{c}_j}\left(\sqrt{1-\tilde{c}_j}t\right)dt\right)^{1-\tilde{c}_j}
	\le (2\pi)^{1-\tilde c_j}\alpha_p^{\tilde c_j}.
	\]
 Ultimately, using the fact that $\sum_{j\in \tilde{J}}(1-\tilde c_j)=m_0-k$, that $\alpha_p\gr1$ and that  $\sum_{j\in \tilde{J}}\tilde{c}_j\ls \sum_{j\in J}\tilde c_j=k$, we obtain
$$
\vol_k(K_p\cap H)\leq\vol_k(B_p^k)\prod_{j\in J}\left(\frac{\sqrt{c_j}}{\alpha_j^\frac{1}{p}}\right)^{\tilde{c}_j},
$$
which completes the proof.
\end{proof}


\begin{remark}
We should remark that the proof of Theorem \ref{thm:general_p_bounds} does not provide an explicit intermediate bound in the likes of \eqref{eq:I1-upper-intermediate}, and thus we do not have a result similar to Corollary \ref{cor:B1n-bound} for $p>1$. This is a natural consequence of the fact that in contrast to the case $p=1$, we cannot explicitly compute $\gamma_p$ for general $p>1$.
\end{remark}


{\bf Acknowledgements.} The  authors acknowledge support by the Hellenic Foundation for Research and Innovation (H.F.R.I.) under the call “Basic research Financing (Horizontal support of all Sciences)” under the National Recovery and Resilience Plan “Greece 2.0” funded by the European Union–NextGeneration EU (H.F.R.I. Project Number:15445). The first named author is also supported by Project PID2022-137294NB-I00  funded by   MICIU/AEI/10.13039/501100011033/.\\ The authors would also like to thank Alexandros Eskenazis for many helpful discussions.

\appendix\label{Appendix}

\section{Parseval's identity for characteristic functions}

We include here a proof of the following statement, that justifies the validity of Theorem \ref{thm:parseval} for the case of characteristic functions on bounded intervals. We begin with the following definition:

\begin{defn}[Brascamp--Lieb feasibility on $H^\perp$]
\label{def:BL-feasible}
 Let $m\in\mathbb{N}$, $n_1,\ldots,n_m\in\mathbb{N}$ and set $N=n_1+\ldots+n_m$. Let $H\in G_{N,k}$  be a $k$-dimensional linear subspace in $\mathbb{R}^N=\bigoplus_{j=1}^m \R^{n_j}$. Let $d=\dim(H^\perp)$, $J_0\subseteq\{1,\dots, m\}$ and, for each $j\in J_0$, let
\[
A_j:H^\perp\to \R^{n_j}
\]
be linear maps from $H^\perp$ to the corresponding $\R^{n_j}$. Let $(p_j)_{j\in J_0}\subseteq[1,\infty]$. We say that the
Brascamp--Lieb datum $(H^\perp;(A_j)_{j\in J_0};(p_j)_{j\in J_0})$ is
\emph{feasible} if
\begin{enumerate}
\item[(i)]
\[
d=\sum_{j\in J_0}\frac{n_j}{p_j},
\]
(with the convention $n_j/\infty=0$), and
\item[(ii)]  for every linear subspace
$V\subseteq H^\perp$,
\[
\dim V \le \sum_{j\in J_0}\frac{\dim(A_jV)}{p_j}.
\]
\end{enumerate}
\end{defn}

\begin{remark}\label{rem:BL-ineq}
It was proved in \cite[Theorem 1.13]{BCCT} that if the datum in Definition~\ref{def:BL-feasible} is feasible, then there exists
a finite constant $C_{\mathrm{BL}}=C_{\mathrm{BL}}(H^\perp;(A_j)_{j\in J_0;(p_j)_{j\in J_0}})<\infty$
such that for every measurable functions $f_j:\R^{n_j}\to[0,\infty)$,
\[
\int_{H^\perp}\prod_{j\in J_0} f_j(A_j z)\,dz
\le
C_{\mathrm{BL}}\prod_{j\in J_0}\|f_j\|_{L^{p_j}(\R^{n_j})}.
\]
\end{remark}

\begin{theorem}
\label{thm:parseval-chi}
Let $m\in\mathbb{N}$, $n_1,\ldots,n_m\in\mathbb{N}$ and set $N=n_1+\ldots+n_m$. Let $H\in G_{N,k}$  be a $k$-dimensional linear subspace in $\mathbb{R}^N=\bigoplus_{j=1}^m \R^{n_j}$ and let $I_j\subseteq\mathbb{R}^{n_j}$ be a bounded measurable set for every $1\ls j\ls m$. If $G(z):=\prod_{j=1}^m \left|\widehat{\mathds{1}_{I_j}}(P_{\R^{n_j}}z)\right|$ is integrable in $H^\perp$ then
\begin{equation}\label{eq:parseval-chi}
\int_H \prod_{j=1}^m \mathds{1}_{I_j}(P_{\R^{n_j}}y)dy
=\frac{1}{(2\pi)^{N-n}}\int_{H^\perp}\prod_{j=1}^m \widehat{\mathds{1}_{I_j}}(P_{\R^{n_j}}z)dz
\end{equation}
\end{theorem}

\begin{proof}
We proceed by approximating characteristic functions by functions in the Schwartz space. Let \(\phi_j\in\mathcal S(\R^{n_j})\) be a mollifier for any $1\le j\le m$. That is $\phi_j\in\mathcal{C}^{\infty)}(\R^{n_j})$ is compactly supported with $\displaystyle{\int_{\mathbb{R}^{n_j}}\phi_j=1}$ that satisfies that $\displaystyle{\lim_{\varepsilon\to0^+}\varepsilon^{-n_j}\phi_j\left(\frac{\cdot}{\varepsilon}\right)}=\delta_0$ in the space of Schartz distributions.

 Let, for any $k\in\mathbb{N}$, \(\phi_{j,k}(x)=k^{n_j}\phi_j(kx)\) for any $x\in\mathbb{R}^n$ and \(g_{j,k}=\mathds{1}_{I_j}*\phi_{j,k}\). Then \(g_{j,k}\in\mathcal S(\R^{n_j})\), \(\displaystyle{\lim_{k\to\infty}g_{j,k}=\mathds{1}_{I_j}}\) a.e. and in \(L^1(\R^{n_j})\). Besides, \(\|g_{j,k}\|_\infty\le 1\). Applying Theorem \ref{thm:parseval} to the functions \(g_{j,k}\) we have that for every $1\le j\le m$ and every $k\in\mathbb{N}$,
\[
\int_H \prod_{j=1}^m g_{j,k}(P_{\R^{n_j}}y)dy
=\frac{1}{(2\pi)^{N-k}} \int_{H^\perp} \prod_{j=1}^m \widehat{g_{j,k}}(P_{\R^{n_j}}z)dz.
\]
By the dominated convergence and the \(L^1\)-approximation of \(\mathds{1}_{I_j}\) by \(g_{j,k}\), we have that
$$
\lim_{k\to\infty}\int_H \prod_{j=1}^m g_{j,k}(P_{\R^{n_j}}y)dy=\int_H \prod_{j=1}^m \mathds{1}_{I_j}(P_{\R^{n_j}}y)\,dy
$$

On the other hand, since $\widehat{g_{j,k}}=\widehat{\mathds{1}_{I_j}}\,\widehat{\phi_{j,k}}$ and, for each fixed $\xi\in\mathbb{R}^n$, $\displaystyle{\lim_{k\to\infty}\widehat{\phi_{j,k}}(\xi)=\lim_{k\to\infty}\widehat{\phi}\left(\frac{\xi}{k}\right)= 1}$, it follows that for each fixed \(z\in H^\perp\),
\[
\lim_{k\to\infty}\prod_{j=1}^m \widehat{g_{j,k}}(P_{\R^{n_j}}z)
=\prod_{j=1}^m \widehat{\mathds{1}_{I_j}}(P_{\R^{n_j}}z).
\]
Since  for every $z\in\mathbb{R}^{n_j}$, for every $1\le j\le m$ and every $k\in\mathbb{N}$ we have $|\widehat{\phi_{j,k}}(\xi)|=|\widehat{\phi}\left(\frac{\xi}{k}\right)|\le\Vert\phi\Vert_\infty$, we get that for every $z\in\mathbb{R}^{n_j}$ and every $1\le j\le m$
\[
\Big|\prod_{j=1}^m \widehat{g_{j,k}}(P_{\R^{n_j}}z)\Big|
\le \Vert\phi\Vert_\infty^m \prod_{j=1}^m \left|\widehat{\mathds{1}_{I_j}}(P_{\R^{n_j}}z)\right|=\Vert\phi\Vert_\infty^m G(z).
\]
Since by hypothesis $G$ is integrable in $H^\perp$, the dominated convergence theorem yields
\[
\lim_{k\to\infty}\frac{1}{(2\pi)^{N-k}}\int_{H^\perp}\prod_{j=1}^m \widehat{g_{j,k}}(P_{\R^{n_j}}z)dz
= \frac{1}{(2\pi)^{N-k}}\int_{H^\perp}\prod_{j=1}^m \widehat{\mathds{1}_{I_j}}(P_{\R^{n_j}}z)dz.\qedhere
\]
\end{proof}

The following lemma gives a condition for the integrability of the function $G$ in theorem under feasibility conditions on the datum on $H^\perp$ and good decay of the fourier transform of the functions. Such conditions will be fulfilled in the cases we considered in this paper.

\begin{lemma}
\label{lem:integrability_BL}
Let $m\in\mathbb{N}$, $n_1,\ldots,n_m\in\mathbb{N}$ and set $N=n_1+\ldots+n_m$. Let $H\in G_{N,k}$  be a $k$-dimensional linear subspace in $\mathbb{R}^N=\bigoplus_{j=1}^m \R^{n_j}$ and let $I_j\subset \R^{n_j}$ be bounded measurable sets, for every $1\ls j\ls m$, and $G:H^\perp\to[0,\infty)$ be the function
\[
G(z):=\prod_{j=1}^m\left|\widehat{\mathds{1}_{I_j}}\!\left(P_{\R^{n_j}}z\right)\right|.
\]
Let $J_0:=\{1\le j\le m\,:\, P_{\R^{n_j}}|_{H^\perp}\not\equiv 0\}=\{1\le j\le m\,:\,\R^{n_j}\not\subseteq H\}$.
Assume that for each $j\in J_0$ there exists $C_j>0$ such that
\begin{equation}
\label{eq:Fourier_decay_chi}
\left|\widehat{\mathds{1}_{I_j}}(\xi)\right|
\le \frac{C_j}{1+\|\xi\|_2},
\qquad \forall\,\xi\in\R^{n_j}.
\end{equation}
Assume furthermore that there exist exponents $p_j\in(1,\infty]$, $j\in J_0$,
such that
\begin{equation}
\label{eq:p_condition}
p_j>n_j\quad \text{for all }j\in J_0,
\end{equation}
and the Brascamp--Lieb datum $(H^\perp;(P_{\R^{n_j}}|_{H^\perp})_{j\in J_0};(p_j)_{j\in J_0})$
is feasible in the sense of Definition~\ref{def:BL-feasible}.
Then $G\in L^1(H^\perp)$.
\end{lemma}

\begin{proof}
Let us denote, for every $1\ls j\ls m$, $A_j=P_{\R^{n_j}}|_{H^\perp}$, the orthogonal projection onto the corresponding $\R^{n_j}$. If $j\notin J_0$ we have $A_j\equiv 0$ on $H^\perp$, hence
$\widehat{\mathds{1}_{I_j}}(P_{\R^{n_j}}z)=\widehat{\mathds{1}_{I_j}}(0)=|I_j|$
is constant and does not affect integrability. Thus it suffices to show that
\[
\int_{H^\perp}\prod_{j\in J_0}\left|\widehat{\mathds{1}_{I_j}}(A_j z)\right|\,dz<\infty.
\]
Set $f_j(\xi):=\left|\widehat{\mathds{1}_{I_j}}(\xi)\right|$ on $\R^{n_j}$.
By the decay assumption \eqref{eq:Fourier_decay_chi}, for every $\xi\in\R^{n_j}$ we have
\[
0\le f_j(\xi)\le \frac{C_j}{1+\|\xi\|_2}.
\]
We claim that $f_j\in L^{p_j}(\R^{n_j})$ whenever $p_j>n_j$. Indeed, if $p_j<\infty$,
integration in polar coordinates yields,
\[
\|f_j\|_{L^{p_j}(\R^{n_j})}^{p_j}
\le C_j^{p_j}\int_{\R^{n_j}}\frac{1}{(1+\|\xi\|_2)^{p_j}}\,d\xi
= C_j^{p_j} n_j\vol_{n_j}(B_2^{n_j})\int_{0}^{\infty}\frac{r^{n_j-1}}{(1+r)^{p_j}}\,dr<\infty,
\]
since $p_j>n_j$ implies $(1+r)^{-p_j}r^{n_j-1}$ is integrable at infinity.
If $p_j=\infty$, then $\|f_j\|_{L^\infty}\le C_j$ is immediate from
\eqref{eq:Fourier_decay_chi}. This proves $f_j\in L^{p_j}(\R^{n_j})$ for all $j\in J_0$.

Now apply the Brascamp--Lieb inequality (Remark~\ref{rem:BL-ineq}) to the feasible
datum $(H^\perp;(A_j)_{j\in J_0};(p_j)_{j\in J_0})$:
\[
\int_{H^\perp}\prod_{j\in J_0} f_j(A_j z)\,dz
\le
C_{\mathrm{BL}}\prod_{j\in J_0}\|f_j\|_{L^{p_j}(\R^{n_j})}
<\infty.
\]
Hence $\prod_{j\in J_0}|\widehat{\mathds{1}_{I_j}}(A_j z)|$ is integrable on $H^\perp$.
Multiplying by the constant factors $\prod_{j\notin J_0}|I_j|$ yields $G\in L^1(H^\perp)$.
\end{proof}

\begin{remark}
The estimate  \eqref{eq:Fourier_decay_chi} is very general and holds all bounded sets of finite perimeter (for example convex bodies, Lipschitz domains, intervals etc.)
\end{remark}

\begin{remark}
Notice that, in our setup in Corollary \ref{cor:UpperBoundParsevalAndBrascampLieb}, given $1\leq n\leq m$,  $(c_j,v_j)_{j=1}^m\subseteq(0,1]\times S^{n-1}$ providing a decomposition of the identity in $\R^n$, and denoting for any $1\leq k\leq m$ by $H=\textrm{span}\{v_j\,:\,1\leq j\leq k\}$, we identified $\mathbb{R}^n$ with $H_1:=\textrm{span}\{e_i\,:\,1\le i\le n\}\subseteq\mathbb{R}^m$, we took $(x_j)_{j=1}^m$ an orthonormal basis of $\mathbb{R}^m$ such that $P_Hx_j=\sqrt{c_j}v_j$, identified $\mathbb{R}^k=\bigoplus_{j=1}^k \textrm{span}\{x_j\}$ with $\textrm{span}\{x_j\,:\,1\le j\le k\}$ and considered $H^\perp$, the orthogonal complement of $H$ in $\R^k$.

Therefore, in this case $n_j=1$ for every $1\le j\le k$ and we have that $P_Hx_j=\sqrt{c_j}v_j\neq0$ for every $j\in J_0=\{1\le j\le k\,:\,x_j\not\in H^\perp\}$. Besides, for every $j\in J_0$ we have that
$$
n_j=1<\frac{1}{1-c_j}=p_j
$$
and the datum $(H^\perp, (P_{\R^{n_j}}|_{H^\perp})_{j\in J_0}, \left(\frac{1}{1-c_j}\right)_{j\in J_0})$ is feasible.
\end{remark}
	
\end{document}